\documentclass[12pt]{article}
\usepackage{amsfonts,color,curves,graphicx,url,amssymb,amsgen,soul}

\newcommand{\Unif}{\mathop{\mathrm{Unif}}}
\newcommand{\rank}{\mathop{\mathrm{rank}}}
\newcommand{\Rank}{\mathop{\mathrm{rank}}}
\newcommand{\Sym}{\mathop{\mathrm{Sym}}}
\newcommand{\Aut}{\mathop{\mathrm{Aut}}}

\newcommand{\Hull}{\mathop{\mathrm{Hull}}}
\newcommand{\End}{\mathop{\mathrm{End}}}
\newcommand{\Gr}{\mathop{\mathrm{Gr}}}

\renewcommand{\Im}{\mathop{\mathrm{Im}}}
\newcommand{\Ker}{\mathop{\mathrm{Ker}}}
\newcommand{\NS}{\mathop{\mathrm{NS}}}

\newcommand{\cart}{\mathbin{\square}}
\newcommand{\ra}{\mathop{\rightarrow}\limits}
\newcommand{\ffi}[1]{$\mathbb{#1}\mathrm{I}$}

\newcommand{\PSL}{\mathop{\mathrm{PSL}}}
\newcommand{\PSp}{\mathop{\mathrm{PSp}}}
\newcommand{\POm}{\mathop{\mathrm{P}\Omega}}
\newcommand{\PSU}{\mathop{\mathrm{PSU}}}

\newcommand{\oG}{\bar{G}}
\newcommand{\hG}{\hat{G}}

\newtheorem{theorem}{Theorem}[section]
\newtheorem{prop}[theorem]{Proposition}
\newtheorem{lemma}[theorem]{Lemma}
\newtheorem{cor}[theorem]{Corollary}
\newtheorem{prob}[theorem]{Problem}

\newtheorem{unnumber}{}

\newtheorem{prepf}[unnumber]{Proof}
\newenvironment{proof}{\prepf\rm}{\endprepf}
\newcommand{\qed}{\hfill$\Box$}

\newtheorem{preex}{Example}[section]
\newenvironment{example}{\preex\rm}{\endpreex}

\newtheorem{preconj}{Conjecture}[section]
\newenvironment{conj}{\preconj\rm}{\endpreconj}

\begin{document}
\title{Between primitive and $2$-transitive:\\Synchronization and its friends}
\author{J. Ara\'{u}jo\footnote{Universidade Aberta
and
Centro de \'{A}lgebra,
Universidade de Lisboa,
Av. Gama Pinto, 2, 1649-003 Lisboa, Portugal,
\texttt{jaraujo@ptmat.fc.ul.pt}},
P.J. Cameron\footnote{School of Mathematics and Statistics,
University of St Andrews, St Andrews, Fife KY16 9SS, UK,
\texttt{pjc20@st-andrews.ac.uk}},
B. Steinberg\footnote{Department of Mathematics,
City College of New York, New York, NY 10031, USA,
\texttt{bsteinberg@ccny.cuny.edu}}
\thanks{The third author was supported in part by a grant from the Simons Foundation (\#245268 to Benjamin Steinberg), the Binational Science Foundation of Israel and the US (\#2012080), a CUNY Collaborative Research Incentive Grant \#2123 and by a PSC-CUNY grant.}}
\date{}
\maketitle

\clearpage

\tableofcontents

\clearpage

\begin{abstract}
An automaton (consisting of a finite set of states with given transitions) is
said to be synchronizing if there is a word in the transitions  which sends all states of the automaton to a single
state. Research on this
topic has been driven by the \v{C}ern\'y conjecture, one of the oldest and most famous problems
in automata theory, according to which a
synchronizing $n$-state automaton has a reset word of length at most $(n-1)^2$.
The transitions of an automaton generate a transformation monoid on the set
of states, and so an automaton can be regarded as a transformation monoid with
a prescribed set of generators. In this setting, an automaton is synchronizing if the transitions generate a constant map.

A permutation group $G$ on a set $\Omega$ is said to synchronize a map $f$
if the monoid $\langle G,f\rangle$ generated by $G$ and $f$ is synchronizing
in the above sense; we say $G$ is synchronizing if it synchronizes every
non-permutation.

The classes of synchronizing groups and friends form an hierarchy of natural and elegant classes of groups  lying strictly between the classes of
primitive and $2$-homogeneous groups. These classes have been floating around
for some years and it is now time to provide a unified  reference on them.
The study of all these classes has been prompted by the \v{C}ern\'y conjecture, but it is of independent interest since  it involves a rich mix of
group theory, combinatorics, graph endomorphisms, semigroup theory, 
finite geometry, and representation theory, and has
interesting computational aspects as well. So as to make the paper
self-contained, we have provided background material on these topics.

Our purpose here is to present recent work on synchronizing groups and related
topics. In addition to the results that show the connections between the various areas of mathematics mentioned above, we include a new result on the \v{C}ern\'y
conjecture (a strengthening of a theorem of Rystsov), some challenges to
finite geometers (which classical polar spaces can be partitioned into ovoids?),
some thoughts about infinite analogues, and a long list of open problems to
stimulate further work.
\end{abstract}

\clearpage

\section{Introduction}

The study of primitive and multiply-transitive permutation groups is one of
the oldest parts of group theory, going back to Jordan and Mathieu in the
nineteenth century.

Recently, in the study of synchronization of automata, various other classes
of permutation groups have been considered, most notably the
\emph{synchronizing groups}, those permutation groups which, together with
any transformation which is not a permutation, generate a constant map.
The class of synchronizing groups contains the $2$-transitive
(or $2$-homogeneous) groups, and is contained in the class of primitive groups
(or indeed the class of \emph{basic} groups in the O'Nan--Scott
classification).

The purpose of this paper is to introduce the theories of synchronizing
groups, and of various related classes of groups. We stress the links to
semigroup theory and automata theory and give numerous examples to show that
for the most part all our classes of groups are distinct.

Since the paper covers a wide variety of subject matter, we have taken some
trouble to include enough background material to make it self-contained.

After a brief introduction we give in Section 1 some theory of
permutation groups, transformation monoids, graphs and digraphs. Section 2
further develops the theory of permutation groups, introducing the notions of
transitive, primitive, $2$-homogeneous and $2$-transitive groups, the
O'Nan--Scott theorem, and the Classification of Finite Simple Groups.

In Section 3 we introduce our main concern, the notion of synchronization for
finite automata, which can be expressed as a property of transformation
monoids. We introduce one of the main research problems in this area, the
famous \v{C}ern\'y conjecture. We define the class of synchronizing groups,
and study its relation to the classes already defined.

Section 4 introduces graph homomorphisms and endomorphisms, and characterizes
synchronizing monoids and groups in terms of these concepts. Section 5
defines several related classes of permutation groups. Section 6 gives some
examples, concentrating on the action of the symmetric group on $k$-sets and
the action of a classical group on its polar space.

Section 7 links some of the properties of permutation groups we have considered
with representation theory, and introduces some new classes.
Section 8
gives alternative characterizations of some of our classes in terms of
functions. Section 9 explains the connection between some of our results and
instances of the \v{C}ern\'y conjecture, including a new theorem which
strengthens a result of Rystsov.

In Section 10, we look at other classes of permutation groups lying between
primitive and $2$-transitive, and in Section 11 we take a brief look at what
happens in the infinite case. The final Section 12 lists a number of unsolved
problems.

In the remainder of this section, we give a brief outline to permutation groups,
transformation {monoids}, graphs and digraphs.

This survey grew from a course given by the second author in 2010 at the
London Taught Course Centre; we are grateful to the course participants
for their comments.

Apart from the present authors, many others have contributed to the
theory presented here. We are particularly grateful to Peter Neumann, whose
contributions are discussed in many places. Csaba Schneider wrote
\textsf{GAP} code for determining which primitive permutation groups of
small degree are synchronizing, which guided many of our conjectures.
Others to whom we are grateful include Wolfram Bentz, Michael Brough,
Ian Gent, Nick Gravin, Cristy Kazanidis, Tom Kelsey, James Mitchell,
Dima Pasechnik, Colva Roney-Dougal, Gordon Royle, Nik Ru\v{s}ku\'c, Jan Saxl,
Artur Schaefer and Pablo Spiga.

\subsection{Permutation groups}

For general references on permutation groups, we recommend
\cite{c:permgps,dixon_mortimer,wielandt}.

The \emph{symmetric group} $\Sym(\Omega)$ on a set $\Omega$ is the group whose
elements are all the permutations of $\Omega$ and whose operation is
composition. If $\Omega=\{1,2,\ldots,n\}$, we write the group as $S_n$, the
symmetric group of degree $n$. We write permutations on the right of their
argument, so that $ag$ is the image of $a$ under the permutation $g$: this
has the advantage that the composition ``$g$ followed by $h$'' is $gh$, rather
than $hg$.

A \emph{permutation group} $G$ on $\Omega$ is a subgroup of $\Sym(\Omega)$.
The \emph{degree} of $G$ is the cardinality of $\Omega$.

Almost always in this paper, $\Omega$ is a finite set.

It is more usual to define an \emph{action} of a group $G$ on a set $\Omega$,
this being defined as a homomorphism from $G$ to the symmetric group
$\Sym(\Omega)$. The advantage is that the same group may act on several
different sets. Note that the image of an action is a permutation group.
Most concepts defined below, starting with transitivity, can be extended to
group actions by saying that (for example) an action is transitive if its image
is a transitive permutation group.

It is well known from elementary discrete mathematics that a permutation on
a finite set can be decomposed into disjoint cycles. A similar decomposition
applies to a permutation group. Let $G$ be a permutation group on $\Omega$.
Define an equivalence relation $\equiv$ on $\Omega$ by the rule that
$a\equiv b$ if $ag=b$ for some $g\in G$. (The reflexive, symmetric and
transitive laws for $\equiv$ follow immediately from the identity, inverse
and closure axioms for $G$.) The equivalence classes are the \emph{orbits}
of $G$ on $\Omega$.

We say that $G$ is \emph{transitive} if it has just one orbit.

\subsection{Transformation {monoids}}

The set of all mappings from $\Omega$ to itself, with the operation of
composition, is a \emph{monoid}: that is, it is closed and associative and
has an identity element. It is called the \emph{full transformation monoid}
on $\Omega$, denoted by $T(\Omega)$, or $T_n$ if $\Omega=\{1,\ldots,n\}$.
As for permutations, we write a transformation on the right of its argument.

Note that $\Sym(\Omega)$ is a subgroup of $T(\Omega)$. The difference
$T(\Omega)\setminus\Sym(\Omega)$ consists of all the \emph{singular maps} on
$\Omega$.

Let $f\in T(\Omega)$. The \emph{image} of $f$, which we write as
$\Im(f)$ or $\Omega f$, is the subset
$\{af:a\in\Omega\}$ of $\Omega$. The \emph{rank} of $f$ is the
cardinality of its image. The \emph{kernel} $\Ker(f)$ of $f$ is the equivalence
relation $\equiv$ on $\Omega$ defined by $a\equiv b$ if $af=bf$; we will
not distinguish between the equivalence relation and the corresponding
partition. Note that the number of equivalence classes of $\Ker(f)$ is equal
to the rank of $f$.

\subsection{Graphs and digraphs}

A \emph{graph} on the vertex set $\Omega$ can be regarded in several ways: as
a symmetric binary relation called \emph{adjacency} on $\Omega$, or as a
collection of subsets called \emph{edges}, each of cardinality $1$ or $2$.
Note that this definition forbids multiple edges. Usually we will also forbid
loops: that is, we do not allow a vertex to be adjacent to itself (so edges
cannot have cardinality $1$).

Let $\Gamma$ be a graph on the vertex set $\Omega$. An \emph{induced subgraph}
of $\Gamma$ is obtained by choosing a subset $A$ of $\Omega$ as vertex set,
and including all edges of $\Gamma$ which join two vertices of $A$. A
\emph{spanning subgraph} is obtained by choosing the whole of $\Omega$ as
vertex set, but taking a subset of the edge set of $\Gamma$. Thus, for
example, the Petersen graph (Figure~\ref{f:peter}) does not contain a
$4$-cycle as induced subgraph,
but does have the disjoint union of two $5$-cycles as a spanning subgraph.

A graph is \emph{connected} if, given any two vertices $v$ and $w$, there is
a sequence $v=x_0,x_1,\ldots,x_d=w$ of vertices such that $x_{i-1}$ and
$x_i$ are adjacent for $i=1,\ldots,d$. The smallest such $d$ (for given $v$
and $w$) is the \emph{distance} from $v$ to $w$.

The \emph{complete graph} on a given vertex set has all possible edges; the
\emph{null graph} has no edges. These graphs are denoted by $K_n$ and $N_n$ if
there are $n$ vertices. The \emph{line graph} of a graph $\Gamma$ is
the graph $L(\Gamma)$ whose vertex set is the edge set of $\Gamma$, two
vertices of $L(\Gamma)$ being adjacent if the corresponding edges of $\Gamma$
have a common vertex. The \emph{complement} $\overline{\Gamma}$ of $\Gamma$
is the graph with the same vertex set as $\Gamma$, two vertices being adjacent
in the complement if and only if they are not adjacent in $\Gamma$.

Let $P$ be a partition of a set $\Omega$. The \emph{complete multipartite
graph} on $\Omega$ with multipartition $P$ is the graph in which two vertices
are adjacent if and only if they belong to different parts of $P$. If the
number of parts is $2$, we speak of a \emph{complete bipartite graph} with
bipartition $P$. A spanning subgraph of a complete multipartite graph is
called \emph{multipartite}, and similarly for bipartite.

An \emph{automorphism} of a graph is a permutation of the vertex set which
maps edges to edges. (If the graph is infinite, we must also require that it
maps non-edges to non-edges.) The set of all automorphisms of a graph is a
group, a permutation group on the vertex set, called the \emph{automorphism
group} of the graph.

The best known graph is the \emph{Petersen graph}, the graph with ten vertices
and fifteen edges shown in Figure~\ref{f:peter}.

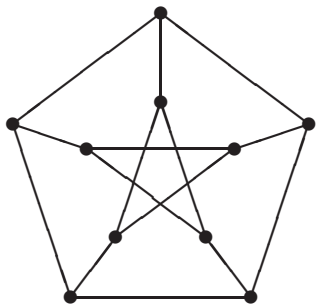
\begin{figure}[htbp]
\begin{center}
\setlength{\unitlength}{0.6mm}
\begin{picture}(80,90)
\thicklines
\put(20,20){\line(1,0){40}}
\put(60,20){\line(1,3){12.8}}
\put(72.8,58.4){\line(-4,3){32.8}}
\put(40,83){\line(-4,-3){32.8}}
\put(7.2,58.4){\line(1,-3){12.8}}
\put(20,20){\line(3,4){10}}
\put(60,20){\line(-3,4){10}}
\put(72.8,58.4){\line(-3,-1){16.4}}
\put(40,83){\line(0,-1){19.7}}
\put(7.2,58.4){\line(3,-1){16.4}}
\put(23.6,52.9){\line(1,0){32.8}}
\put(50,33.3){\line(-4,3){26.4}}
\put(30,33.3){\line(4,3){26.4}}
\put(50,33.3){\line(-1,3){10}}
\put(30,33.3){\line(1,3){10}}
\put(20,20){\circle*{3}}
\put(60,20){\circle*{3}}
\put(72.8,58.4){\circle*{3}}
\put(40,83){\circle*{3}}
\put(7.2,58.4){\circle*{3}}
\put(30,33.3){\circle*{3}}
\put(50,33.3){\circle*{3}}
\put(23.6,52.9){\circle*{3}}
\put(56.4,52.9){\circle*{3}}
\put(40,63.3){\circle*{3}}
\end{picture}
\end{center}
\caption{\label{f:peter}The Petersen graph}
\end{figure}

The Petersen graph has $120$ automorphisms, forming a group isomorphic to the
symmetric group $S_5$. (To prove this, one first argues directly that there
cannot be more than $120$ automorphisms. Then label the vertices by pairs of
elements from $\{1,\ldots,5\}$ in such a way that two vertices are adjacent if
and only if their labels are disjoint. In other words, the Petersen graph is
the complement of the line graph of $K_5$. So the symmetric group on
$\{1,\ldots,5\}$ has an action on the vertices which preserves the adjacency
relation.)

\medskip

A \emph{directed graph}, or \emph{digraph}, is similarly defined, except that
an edge is an ordered pair of vertices. So each edge has a direction, say
from $v$ to $w$, and can be represented by an arrow with tail at $v$ and head
at $w$. Such directed edges are called \emph{arcs}.

We say that a directed graph is \emph{connected}, if, when we ignore the
directions of the arcs (and keep only one of any pair of edges thus created),
the undirected graph so obtained is connected.
A directed graph is \emph{strongly connected} if, given any two
vertices $v$ and $w$, there is a sequence $v=x_0,x_1,\ldots,x_d=w$ of vertices
such that there is an arc from $x_{i-1}$ to $x_i$ for $i=1,\ldots,d$.

Automorphisms of directed graphs are defined similarly to the undirected case.
The following important result holds.

\begin{theorem}
Let $D$ be a finite directed graph whose automorphism group is transitive on
the vertices. If $D$ is connected, then it is strongly connected.
\label{t:chicago}
\end{theorem}

\begin{proof}
Let $R(x)$ be the set of vertices which can be reached by directed paths
starting at $x$. Clearly we have
\begin{enumerate}\itemsep0pt
\item $x\in R(x)$;
\item if $y\in R(x)$, then $R(y)\subseteq R(x)$;
\item if an automorphism $g$ carries $x$ to $y$, then it maps $R(x)$ to $R(y)$.
\end{enumerate}
Now by the third property, if $\Aut(D)$ is vertex-transitive, then $|R(x)|$ is
constant for all vertices $x$; by the second property, if $y\in R(x)$, then
$R(y)=R(x)$, and so by the first property $x\in R(y)$.

Now suppose that $D$ is connected. Given $v$ and $w$, take a path from $v$
to $w$ in the undirected graph. Now by what we have proved, any arc in the
``wrong'' direction can be replaced by a path all of whose arcs are in the
``right'' direction; so $D$ is strongly connected.\qed
\end{proof}

\clearpage

\section{Transitivity and primitivity}

In this section, we introduce some of the basic notions of permutation group
theory, especially primitivity.

We begin this section with a convention that will prove useful later. We say
that a structure on a set $\Omega$ is \emph{trivial} if it is preserved by
the whole symmetric group on $\Omega$, and is \emph{non-trivial} otherwise.

So, for example,
\begin{enumerate}\itemsep0pt
\item the trivial subsets of $\Omega$ are the empty set and the whole of
$\Omega$;
\item the trivial partitions of $\Omega$ are the partition all of whose parts
are singletons (corresponding to the equivalence relation of equality), and
the partition with a single part;
\item the trivial graphs on $\Omega$ are the complete and null graphs.
\end{enumerate}

\subsection{Transitivity}

As we saw in the last section, a permutation group $G$ on $\Omega$ is
\emph{transitive} if we can map any element of $\Omega$ to any other by
an element of $G$. In the convention introduced above, $G$ is transitive if
the only $G$-invariant subsets of $\Omega$ are the trivial ones. The same
definition applies to an action of $G$ on $\Omega$.

The \emph{stabilizer} $G_a$ of a point $a\in\Omega$ is the set
\[\{g\in G:ag=a\}\]
of elements of $G$ (which is easily seen to be a subgroup of $G$).

There is an internal description of the transitive actions of a group, as
follows. Let $H$ be a subgroup of $G$. The \emph{coset space} $H\backslash G$
consists of all right cosets $Hx$ of $H$ in $G$; there is an action of $G$
on $H\backslash G$, where the permutation corresponding to the group element
$g$ maps the coset $Hx$ to the coset $Hxg$. The action of $G$ on $H\backslash G$
is transitive, and the stabilizer of the coset $H$ is the subgroup $H$.

Two actions of $G$ on
sets $\Omega_1$ and $\Omega_2$ are \emph{isomorphic} if there is a bijection
$\phi\colon \Omega_1\to\Omega_2$ commuting with the action of $G$, that is, such
that $(ag)\phi=(a\phi)g$ for all $a\in\Omega_1$ and $g\in G$.

\begin{theorem}
\begin{enumerate}\itemsep0pt
\item Any transitive action of $G$ is isomorphic to the action of $G$ on
a coset space (specifically, on $H\backslash G$, where $H$ is the stabilizer of
a point).
\item The actions of $G$ on coset spaces $H\backslash G$ and
$K\backslash G$ are isomorphic if and only if $H$ and $K$ are conjugate
subgroups of $G$.
\end{enumerate}
\end{theorem}

In particular, a group $G$ has a unique (up to isomorphism) \emph{regular}
action, characterized by the fact that it is transitive and the stabilizer
of a point is the identity. If we identify a singleton subset of $G$ (a coset
of the identity subgroup) with an element, this is the action of $G$ on itself
by right multiplication, as used by Cayley to show that every group is
isomorphic to a permutation group.

\subsection{Primitivity}

A transitive permutation group $G$ on $\Omega$ is said to be \emph{primitive}
if the only $G$-invariant partitions of $\Omega$ are the trivial ones.

Primitivity is a very important concept in permutation group theory, and we
will see several further characterizations of it.

A subset $B$ of $\Omega$ is called a \emph{block}, or \emph{block of
imprimitivity}, for $G$ if, for all $g\in G$, either $Bg=B$ or
$Bg\cap B=\emptyset$.

\begin{prop}
The transitive permutation group $G$ on $\Omega$ is primitive if and only if
the only blocks for $G$ are the empty set, singletons, and the whole of
$\Omega$.
\end{prop}

\begin{proof}
A part of any $G$-invariant partition is clearly a block. Conversely, if $B$
is a non-empty block, then for all $g,h\in G$, we have $Bg=Bh$ or
$Bg\cap Bh=\emptyset$; so the translates of $B$ under $G$ form a $G$-invariant
partition. The result follows.\qed
\end{proof}

We saw in the last subsection that any transitive group $G$ can be identified
with $G$ acting on a coset space $H\backslash G$.

\begin{prop}
The action of $G$ in $H\backslash G$ is primitive if and only if $H$ is a
maximal subgroup of $G$.
\end{prop}

\begin{proof}
If $H\le K\le G$, then the cosets of $H$ contained in $K$ form a block for $G$;
every block containing the coset $H$ arises in this way.\qed
\end{proof}

\subsubsection{Normal subgroups}

As Cayley observed, every group is isomorphic to a transitive permutation
group. However, not every group is isomorphic to a primitive permutation
group; primitive groups have strong restrictions on their normal subgroup
structure. The basic observation is:

\begin{prop}
A non-trivial normal subgroup of a primitive group is transitive.
\end{prop}

\begin{proof}
This follows from the observation that the orbits of a normal subgroup of a
transitive group are blocks for the group.\qed
\end{proof}

\begin{theorem}
A primitive permutation group has at most two minimal normal subgroups; if
there are two, then they are isomorphic, non-abelian, and regular, and each
is the centralizer of the other in the symmetric group.
\end{theorem}

\begin{proof}
A permutation group (not necessarily transitive) is called \emph{semiregular}
if the stabilizer of any point is the identity. Thus a transitive semiregular
group is regular. It is easy to show that the centralizer of a transitive
group is semiregular.

Suppose that $N_1$ and $N_2$ are minimal normal subgroups of the primitive
group $G$. Then each of $N_1$ and $N_2$ is transitive; but they centralize
each other, and so each is semiregular, and so regular. Clearly it is not
possible for there to be a third minimal normal subgroup.

The centralizer of a regular permutation group (in the symmetric group) is
regular; indeed, the centralizer of the \emph{right regular representation}
of a group (acting on itself by right multiplication) is the \emph{left
regular representation}. These two regular groups coincide if and only if they
are abelian. So, in our situation, $N_1$ and $N_2$ must be non-abelian.\qed
\end{proof}

\begin{example}
Here is an example of a primitive group with two minimal normal subgroups.
Let $S$ be any finite group, and let $G=S\times S$. Then there is an action
of $G$ on itself, where the first factor acts on the left and the second on
the right, as follows:
\[(g,h): x \mapsto g^{-1}xh.\]
It is easy to show that the action is faithful if and only if $S$ has trivial
centre. Moreover, the action is primitive if and only if $S$ is a non-abelian
simple group. For
any congruence for the second factor is the relation ``same coset of $T$'' for
some subgroup $T$ of $S$; and this congruence is preserved by the first factor
if and only if $T$ is a normal subgroup. So, if $S$ is simple, then $G$ has
only the trivial congruences; and conversely.
\end{example}

\subsubsection{Other definitions}

In this section, we give two further properties equivalent to primitivity,
one due to Higman, the other to Rystsov.

Let $G$ be a transitive permutation group on $\Omega$. Then the set $\Omega^2$
of ordered pairs of elements of $\Omega$ is partitioned into orbits under
the componentwise action of $G$. These orbits are called \emph{orbitals} of
$G$. One orbital consists of all the pairs $(a,a)$ for $a\in\Omega$ (by the
assumption of transitivity); this is the \emph{diagonal orbital}. Any
non-diagonal orbital can be regarded as the set of edges of a digraph on
the vertex set $\Omega$, called an \emph{orbital digraph} of $G$. If
the orbital $O$ is \emph{symmetric} (that is, $(a,b)\in O$ implies
$(b,a)\in O$), then we can regard the orbital digraph as an undirected graph.

\begin{theorem}
The transitive permutation group $G$ is primitive if and only if every
non-diagonal orbital digraph is connected.
\end{theorem}

\begin{proof}
If there is an orbital digraph which is not connected, then its connected
components form blocks for $G$. Conversely, suppose that there is a non-trivial
$G$-invariant partition $P$ of $\Omega$, and choose distinct points $a,b$ in
the same part of $P$. Then the orbital digraph corresponding to the orbital
$(a,b)G$ has the property that all its edges are contained within parts of $P$,
so it is not connected.\qed
\end{proof}

The next theorem was essentially proved, but not stated, by
Rystsov~\cite{rystsov}; the statement in this form appears in~\cite{ac}.

\begin{theorem}
Let $G$ be a transitive permutation group on $\Omega$, where $|\Omega|=n$. Then
$G$ is primitive if and only if, for any map $f\colon \Omega\to\Omega$ of rank $n-1$,
the monoid $\langle G,f\rangle$ contains an element of rank $1$.
\label{t:rystsov}
\end{theorem}

This theorem is not difficult, but it will be much easier when we have
developed a bit more technique.

\subsection{Imprimitive groups and wreath products}

A transitive but imprimitive permutation group $G$ preserves a partition $P$
of $\Omega$, and so is contained in the group of all permutations fixing this
partition. Since $G$ is transitive, the partition is uniform, with (say) $m$
parts each of size $k$. Then the group fixing the partition is the
\emph{wreath product} $S_k\wr S_m$ of symmetric groups of degrees $k$ and $m$.
This means that it is the product of two subgroups:
\begin{enumerate}\itemsep0pt
\item the \emph{base group}, the direct product of $m$ copies of $S_k$,
where the $i$th copy acts on the $i$th part of the partition $P$;
\item the \emph{top group}, isomorphic to $S_m$, which permutes the parts
of $P$.
\end{enumerate}
The base group is a normal subgroup, and the top group acts on it by
permuting the direct factors. Thus the group has order $(k!)^m\,m!$.

It is possible to define the wreath product of two arbitrary permutation
groups similarly, and to show that if $G$ is an imprimitive group, $H$ is the
group induced on a block by its setwise stabilizer, and $L$ is the group
of permutations of the blocks induced by $G$, then $G$ is embedded into
the wreath product $H\wr L$.

The set on which the wreath product acts can be identified with the
Cartesian product $K\times M$, where $K=\{1,\ldots,k\}$ and
$M=\{1,\ldots,m\}$: we think of this as a fibre space over $M$, where each
fibre is isomorphic to $K$ (Figure~\ref{f:fibre}):

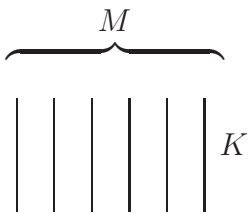
\begin{figure}[htbp]
\begin{eqnarray*}
&M&\\
&\overbrace{\hbox{\setlength{\unitlength}{0.5mm}
\begin{picture}(55,40)
\multiput(0,0)(10,0){6}{\line(0,1){30}}
\put(54,15){$K$}
\end{picture}}}&
\end{eqnarray*}
\caption{\label{f:fibre}A fibration}
\end{figure}

This action of the wreath product is called the \emph{imprimitive action},
as opposed to the \emph{power action}, which we will meet shortly.

\subsection{The O'Nan--Scott theorem}

The group-theoretic structure of primitive groups was further elucidated
independently by Michael O'Nan and Leonard Scott at a conference on finite
groups in Santa Cruz in 1979. Both papers appeared in preliminary proceedings
but only Scott's paper is in the final volume. We do not require the full
detail of the theorem (for which see~\cite{lps}), so we can make some
simplifications.

The \emph{socle} of a finite group is the product of its minimal normal
subgroups. (Any two minimal normal subgroups commute, and each has trivial
intersection with the product of the others, so we have a direct product.) As
we saw above, a primitive group has at most two minimal normal subgroups, and
if there are two then they are isomorphic; so the socle of a primitive group
is a product of isomorphic finite simple groups. The O'Nan--Scott theorem
allows us to apply the Classification of Finite Simple Groups to the study
of primitive groups.

\subsubsection{Non-basic groups}
A \emph{Cartesian structure} or \emph{power structure} on $\Omega$ is a
bijection between $\Omega$ and the set $K^M$  of functions from
$M$ to $K$, where $|M|,|K|>1$. This gives $\Omega$ the
structure of an $m$-dimensional hypercube (where $m=|M|$) whose sides
have size $|K|$. If $K=\{1,\ldots,k\}$ and $M=\{1,\ldots,m\}$,
then $\Omega$ is identified with the set of $m$-tuples over the alphabet
$K=\{1,\ldots,k\}$. The automorphism group of a power structure is the wreath
product $S_k\wr S_m$, but in a different action from the imprimitive action
we saw earlier: the \emph{power action}, or \emph{product action}, of the
wreath product.

Let $G$ act on $\Omega$. We say that $G$ is \emph{non-basic} if
it preserves a Cartesian structure on $\Omega$, and \emph{basic}
otherwise.

A transitive non-basic group is embeddable in the wreath product of permutation
groups on $K$ and $M$ in the power action. Elements of the base
group of the wreath product permute the symbols in each coordinate
independently, while elements of the top group permute the coordinates.
If we take a set of size $km$ partitioned into $m$ sets of size $k$ on which the
wreath product has its imprimitive action, then we can identify the elements
of the Cartesian structure with sets of points which are sections for the
partition (that is, contain one element from each part).

In other language, in terms of the fibre space $K\times M$ on which the wreath
product has its imprimitive action, the elements of the product action (which
are the functions $\phi:M\to K$) are the \emph{global sections} of the
fibration (Figure~\ref{f:global}):

\begin{figure}[htbp]
\begin{eqnarray*}
&M&\\
&\overbrace{\hbox{\setlength{\unitlength}{0.5mm}
\begin{picture}(55,40)
\multiput(0,0)(10,0){6}{\line(0,1){30}}
\put(-20,17){$\phi(d)$}
\put(-2,32){$d$}
\put(0,20){\circle*{2}}\put(0,20){\circle{4}}
\put(10,17){\circle*{2}}\put(10,17){\circle{4}}
\put(20,22){\circle*{2}}\put(20,22){\circle{4}}
\put(30,25){\circle*{2}}\put(30,25){\circle{4}}
\put(40,26){\circle*{2}}\put(40,26){\circle{4}}
\put(50,21){\circle*{2}}\put(50,21){\circle{4}}
\curve(0,20,10,17,20,22,30,25,40,26,50,21)
\put(54,15){$K$}
\end{picture}}}&
\end{eqnarray*}
\caption{\label{f:global}A global section}
\end{figure}
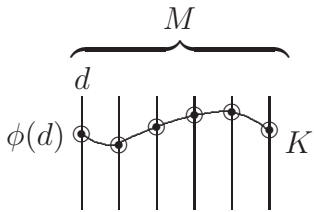

However, for primitive groups, we can make a stronger statement.

\begin{theorem}[O'Nan--Scott for non-basic groups]
Let $G$ be a primitive but non-basic permutation group with socle $N$. Then
$G$ is embeddable in the wreath product $G_0\wr K$, where $G_0$ is a basic
primitive permutation group. Moreover, if $K$ has degree $n$, then $N=N_0^n$,
where $N_0$ is either the socle or a minimal normal subgroup of $G_0$.
\end{theorem}

The case where $G_0$ has two minimal normal subgroups, of which $N_0$ is one
(the so-called \emph{twisted wreath product} case), was pointed out by Michael
Aschbacher. The smallest twisted wreath product has degree $60^6=46656000000$.
A discussion of these is given in~\cite{dixon_mortimer}.
It will turn out that non-basic groups are non-synchronizing, so we will not
be concerned with twisted wreath products.

\subsubsection{Basic groups}

In order to describe basic groups, we need to look at several special classes
of groups.

\paragraph{Affine groups}
Let $V$ be a $d$-dimensional vector space over the field $\mathbb{F}_p$, where
$p$ is prime, and let $H$ be a group of linear transformations of $V$. Then
there is a corresponding \emph{affine group}
\[G=\{x\mapsto x^h+v: h\in H, v\in V\}\]
of permutations of $V$, generated by the translations (which form a normal
subgroup) and elements of $H$.

\begin{theorem}
With the above notation,
\begin{enumerate}\itemsep0pt
\item $G$ is always transitive;
\item $G$ is primitive if and only if $H$ acts \emph{irreducibly} on $V$
(that is, leaves invariant no non-zero proper subspace of $V$);
\item $G$ is basic if and only if $H$ acts \emph{primitively} on $V$ (that
is, preserves no non-trivial direct sum decomposition of $V$).
\end{enumerate}
\end{theorem}

A primitive group is affine if and only if its socle (which is its unique
minimal normal subgroup) is an elementary abelian $p$-group.

\paragraph{Diagonal groups}
Let $S$ be a non-abelian finite simple group. A \emph{diagonal group} is one
whose socle is $S^n$, acting on the cosets of a diagonal subgroup
\[\{(s,s,\ldots,s):s\in S\}\]
of $S^n$.

For $n=2$ we have the example of $S\times S$ acting by left and right
multiplication we saw earlier.

A diagonal group may also contain
\begin{enumerate}\itemsep0pt
\item automorphisms of $S$, acting in the same way on all factors;
\item permutations of the factors.
\end{enumerate}
If $n>2$, we must have at least a transitive group of permutations of the
factors in order for the diagonal group to be primitive.

\paragraph{Almost simple groups}
A group $G$ is \emph{almost simple} if its socle is simple. Such a group is
an extension of a simple group by a subgroup of its automorphism group; in
other words, there is a simple group $S$ such that $S\le G\le\Aut(S)$.

For example, the symmetric group $S_n$ is almost simple for $n\ge5$. (It is
affine for $n\le 4$.)

The almost simple primitive groups are the largest and least understood class.
Note that, unlike the other two cases, the action of the group is not specified.

\begin{theorem}[O'Nan--Scott for basic groups]
Let $G$ be a basic primitive permutation group. Then $G$ is affine, or diagonal,
or almost simple.
\end{theorem}

The O'Nan--Scott Theorem opened the way to the application of the
Classification of Finite Simple Groups to permutation group theory, which has
been done very successfully since the Classification was first announced in
1980. (These results were conditional on the Classification until its proof
was completed in 2005.)

\subsection{The Classification of Finite Simple Groups}

This major theorem has a proof which is currently over 10000 pages long.
We will not specify the groups too precisely, since there are good sources
for this: we recommend~\cite{wilson}.

\begin{theorem}
Any finite simple group is one of the following:
\begin{enumerate}\itemsep0pt
\item a cyclic group of prime order;
\item an alternating group $A_n$, $n\ge5$;
\item a group of Lie type;
\item one of the $26$ sporadic finite simple groups.
\end{enumerate}
\end{theorem}

We refer to this theorem as CFSG.

The groups of Lie type are quotients of matrix groups over finite fields. There
are finitely many families; some (the \emph{classical groups}, which we will
discuss in more detail later) are parametrized by a dimension and a field
order; the rest (the \emph{exceptional groups}) just by a field order.

\subsection{$2$-transitive and $2$-homogeneous groups}

A permutation group $G$ on $\Omega$ is said to be \emph{$2$-transitive} if
it acts transitively on the set of ordered pairs of distinct elements of
$\Omega$: in other words, given two ordered pairs $(a_1,a_2)$ and $(b_1,b_2)$,
with $a_1\ne a_2$ and $b_1\ne b_2$, there exists $g\in G$ with $a_ig=b_i$
for $i=1,2$.

A permutation group $G$ on $\Omega$ is said to be \emph{$2$-homogeneous} if
it acts transitively on the set of $2$-element subsets of $\Omega$. This is
weaker than $2$-transitivity, since we do not require that we can interchange
two points. Indeed, a $2$-homogeneous group is $2$-transitive if and only if
its order is even. For, if $G$ is $2$-transitive, then an element which
interchanges two points has even order. Conversely, a group of even order
contains an involution, so some pair of points can be interchanged; if the
group is $2$-homogeneous, then any pair can be interchanged.

Using this, the classification of $2$-homogeneous but not $2$-transitive groups
was achieved by Kantor and Berggren independently in the late 1960s, using
the Feit--Thompson theorem on solvability of groups of odd order
\cite{kantor,berggren}.

\begin{theorem}
Let $G$ be a permutation group on $\Omega$ which is $2$-homogeneous but not
$2$-transitive. Then we can identify $\Omega$ with the finite field
$\mathbb{F}_q$
where $q\equiv3\pmod{4}$, so that $G$ is a subgroup of the semi-affine group
\[\{x\mapsto a^2x^\sigma+b:a,b\in\mathbb{F}_q,a\ne0
\sigma\in\Aut(\mathbb{F}_q\}.\]
\end{theorem}

\begin{proof} The group $G$ has odd order, so by the Feit--Thompson
theorem it is solvable. Hence it has an elementary abelian regular normal
subgroup $N$ which is the additive group of a vector space. Now consider
the group $G^+=\langle G,-1\rangle$. This group is $2$-transitive, and also
solvable. The $2$-transitive solvable groups were determined by Huppert
\cite{huppert}; by examining the list we can complete the proof.\qed
\end{proof}

The classification of $2$-transitive groups is a consequence of CFSG. We do
not discuss the result here; the list of $2$-transitive groups can be found
in the books by Cameron~\cite{c:permgps} and by Dixon and Mortimer
\cite{dixon_mortimer}.

We conclude this section with a simple observation:

\begin{prop}
A $2$-homogeneous group is primitive.
\end{prop}

\begin{proof} This follows easily from Higman's Theorem, since any two points
are adjacent in a non-trivial orbital graph for $G$, which is thus 
connected.\qed
\end{proof}

So we have the following properties of permutation groups:
\begin{eqnarray*}
&&\hbox{transitive}\Leftarrow\hbox{primitive}\Leftarrow\hbox{basic}\\
&&\qquad\Leftarrow\hbox{$2$-homogeneous}\Leftarrow\hbox{$2$-transitive}.
\end{eqnarray*}

Note that each of these properties is closed upwards: an overgroup
of a permutation group with the property also has the property.

We will extend this hierarchy by inserting some new classes of permutation
groups between $2$-homogeneous and basic.

\clearpage

\section{Synchronization}

Our automata (illustrated by examples below) are very simple gadgets. An
automaton is in one of a finite set of internal \emph{states}; on reading
an \emph{input}, a symbol from a given alphabet, it undergoes a change of state
according to a given transformation. There is no prescribed initial or terminal
state; these automata do not accept languages.

\subsection{Two examples}

The first example was suggested by Olof Sisask.

\begin{example}
A certain calculator has an `On' button but no `Off' button. To switch it off,
you hold down the `Shift' key and press the `On' button. The `Shift' key has
no effect if the calculator is switched off. Assuming that you can't see the
screen, how can you ensure that the calculator is switched off?

Obviously, pressing the `On' button leaves the calculator switched on, no
matter what its former state; and then `Shift-On' will switch it off.
\end{example}

Note that if, instead, there is a single `On-Off' button which toggles the
states, then the problem would have no answer.

\begin{example}
You are in a dungeon consisting of a number of rooms. Passages are
marked with coloured arrows. Each room contains a special door; in
one room, the door leads to freedom, but in all the others, to
instant death. You have a schematic map of the dungeon (Figure~\ref{f:aut}),
but you do not know where you are.

\begin{figure}[htbp]
\begin{center}
\setlength{\unitlength}{1mm}
\begin{picture}(40,40)
%\thicklines
\color{red}
\multiput(0,0)(0,40){2}{\line(1,0){40}}
\multiput(35,-3)(0,44){2}{$\gets$}
\put(0,0){\line(1,1){40}}
\put(1,4){$\nearrow$}
\curve(40,0,15,15,0,40)
\put(11,18){$\searrow$}
\color{blue}
\multiput(0,0)(40,0){2}{\line(0,1){40}}
\put(41,35){$\downarrow$}
\put(-3,4){$\uparrow$}
\curve(40,0,25,25,0,40)
\put(22,28){$\nwarrow$}
\put(28,22){$\searrow$}
\color{black}
\multiput(0,0)(40,0){2}{\circle*{2}}
\multiput(0,40)(40,0){2}{\circle*{2}}
\put(-4,40){$1$} \put(43,40){$2$}
\put(43,-2){$3$} \put(-4,-2){$4$}
\end{picture}
\end{center}
\caption{\label{f:aut}An automaton}
\end{figure}
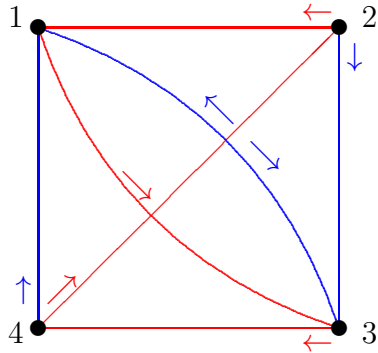

You can check that ({\color{blue}Blue}, {\color{red}Red},
{\color{blue}Blue})  takes you to room $1$ from any starting point.
Then you can use the map to navigate to the exit door.
\end{example}

\subsection{Automata and synchronization}

A (finite, deterministic) \emph{automaton} consists of a finite set $\Omega$
of \emph{states} and a finite set of \emph{transitions}, each transition
being a function from $\Omega$ to itself.

Combinatorially, an automaton can be regarded as an edge-coloured digraph with
one edge of each colour out of each vertex. Vertices are states, colours
are transitions. On reading a given symbol in a given state, the automaton
moves along the edge of the appropriate colour from that state. This is the
representation used in Figure~\ref{f:aut} above.

Algebraically, we will be interested in the composition of transitions of the
automaton. The set of all transformations of the states which are produced
by applying a (possibly empty) sequence of transitions is a
\emph{transformation monoid} on $\Omega$; that is, a set of transformations
which contains the identity map and is closed under composition. So we can
regard an automaton as a transformation monoid (acting on the set of states)
with a prescribed set of generators (the transitions of the automaton).

The \emph{rank} of a transformation of $\Omega$ is the cardinality of its
image.

A \emph{reset word} is a sequence of transitions such that the composition
of the transitions in the sequence, applied to any starting
vertex, brings you to the same state. An automaton which
possesses a reset word is called \emph{synchronizing}. Thus, from the algebraic
point of view, an automaton is synchronizing if the corresponding
transformation monoid contains a map of rank~$1${, that is, a constant map}.

Not every finite automaton has a reset word. For example, if every transition
is a permutation, then every word in the transitions is a permutation, and
has rank equal to $|\Omega|$.

\subsection{The \v{C}ern\'y conjecture}

Here is a simplified example of the application of synchronization in
industrial robotics (cf.~\cite{Eppstein}). The general situation is as follows.

Pieces arrive to be assembled by a robot. The
orientation is critical. You could equip the robot with vision sensors and
manipulators so that it can rotate the pieces into the correct orientation.
But it is much cheaper and less error-prone to regard the possible orientations
of the pieces as states of an automaton on which transitions can be performed
by simple machinery, and apply a reset word before the pieces arrive at the
robot.

\begin{example}
Suppose that the component is square with a projecting tab on one side.

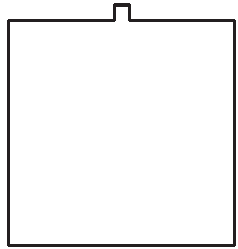
\begin{center}
\setlength{\unitlength}{1mm}
\begin{picture}(30,32)
\thicklines
\multiput(0,0)(30,0){2}{\line(0,1){30}}
\put(0,0){\line(1,0){30}}
\multiput(0,30)(16,0){2}{\line(1,0){14}}
\multiput(14,30)(2,0){2}{\line(0,1){2}}
\put(14,32){\line(1,0){2}}
\end{picture}
\end{center}
It can sit in a tray on the conveyor belt in any one of four orientations.

The following transitions are easy to implement:
\begin{itemize}\itemsep0pt
\item[{\color{red}R}:] rotate through $90^\circ$ in the positive direction;
\item[{\color{blue}B}:] rotate through $90^\circ$ if the projection points up,
otherwise do nothing.
\end{itemize}

Figure~\ref{f:aut2} is a diagram of the automaton. Each state represents the
position of the component with the projection on that side.

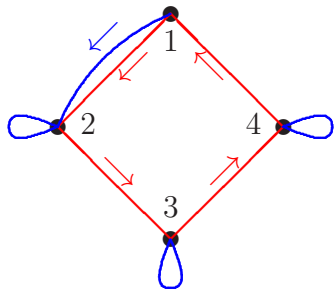
\begin{figure}[htbp]
\begin{center}
\setlength{\unitlength}{1mm}
\begin{picture}(50,40)
\thicklines
\multiput(25,10)(0,30){2}{\circle*{2}}
\multiput(10,25)(30,0){2}{\circle*{2}}
\put(24,35){$1$}
\put(13,24){$2$}
\put(24,13){$3$}
\put(35,24){$4$}
\color{red}
\multiput(25,10)(-15,15){2}{\line(1,1){15}}
\multiput(25,10)(15,15){2}{\line(-1,1){15}}
\put(18,32){$\swarrow$}
\put(16,18){$\searrow$}
\put(30,18){$\nearrow$}
\put(28,32){$\nwarrow$}
\color{blue}
\curve(10,25,16,34,25,40)
\put(14,36){$\swarrow$}
\curve(10,25,5,23.5,3.5,25,5,26.5,10,25)
\curve(25,10,23.5,5,25,3.5,26.5,5,25,10)
\curve(40,25,45,23.5,46.5,25,45,26.5,40,25)
\end{picture}
\end{center}
\caption{\label{f:aut2}Another automaton}
\end{figure}

Now the following table is easily checked.

\newcommand{\R}{\hbox{\color{red}R}} \newcommand{\B}{\hbox{\color{blue}B}}
\[\begin{array}{|c|ccccccccc|}
\hline
  & \B & \R & \R & \R & \B & \R & \R & \R & \B \\\hline
1 & 2  & 3  & 4  & 1  & 2  & 3  & 4  & 1  & 2  \\
2 & 2  & 3  & 4  & 1  & 2  & 3  & 4  & 1  & 2  \\
3 & 3  & 4  & 1  & 2  & 2  & 3  & 4  & 1  & 2  \\
4 & 4  & 1  & 2  & 3  & 3  & 4  & 1  & 2  & 2  \\\hline
\end{array}
\]
So \B\R\R\R\B\R\R\R\B\ is a reset word.
\end{example}

It can be shown that there is no shorter reset word for this automaton.

Moreover, the example extends to any number $n$ of states, replacing the square
by a regular $n$-gon. The corresponding shortest reset word has length
$(n-1)^2$.

In 1969, \v{C}ern\'y made the following conjecture (see~\cite{volkov}):
\begin{conj}
Suppose that an automaton with $n$ states is synchronizing. Then it has a
reset word of length at most $(n-1)^2$.
\end{conj}

This conjecture is still open after close to fifty years! The example above
shows that, if true, it would be best possible. The best current upper bound, despite years of intensive effort, is $\frac{n^3-n}{6}$, due to the combined work of Frankl~\cite{frankl} and Pin~\cite{twocomb} from 1983. The remainder of the literature consists of special cases (e.g.~\cite{mortality,AMSV,volkovc2,volkovc3,volkovc1,arnold_steinberg,bealperrinnew,strongtrans,strongtrans2,cerny,Eppstein,Karicounter,Kari,Pincerny,rystrank,rystsov:regular,rystcom,rystsov,Salomcerny,steinberg,averagingjournal,primecycle,traht2,trahtmanap}). The strongest result is Dubuc's theorem~\cite{dubuc}, which proves the \v{C}ern\'y conjecture under the assumption that some transition cyclically permutes the states, as is the case for the \v{C}ern\'y examples~\cite{cerny}.  See~\cite{Jungers} for a recent approach involving linear programming.

One of the difficulties of the \v{C}ern\'y problem is that there are few known families of slowly synchronizing automata (cf.~\cite{volkprim} for the connection with exponents of primitive digraphs) and so we don't have a very good understanding of what makes an automaton slow to synchronize. In fact, the \v{C}ern\'y sequence is still the only infinite sequence of examples of $n$-state automata that have minimal length reset word of length $(n-1)^2$.

The other issue is that random automata are synchronizing and synchronize quickly.  More precisely, Berlinkov~\cite{berlinkov} showed that a random $n$-state automaton is synchronizing with high probability as $n$ approaches infinity.  Nicaud~\cite{nicaud} proved that if $\varepsilon>0$, then with high probability an $n$-state automaton has a reset word of length at most $n^{1+\varepsilon}$.  Thus one is unlikely to find a counterexample by searching at random; also the search space is too large for a brute force attempt to find a counterexample.

We will not prove the \v{C}ern\'y conjecture in this paper, but it provided
motivation for our approach, and we will return to it later.

For an accessible discussion of the \v{C}ern\'y conjecture, we recommend
the survey by Volkov~\cite{volkov}.

\subsection{The Road-Colouring Conjecture}

The underlying digraph of an automaton with $n$ transitions is a digraph with
the property that every vertex has exactly $n$ arcs leaving it.

Conversely, and trivially, given any digraph with this property, it is clear
that it can be edge-coloured so as to represent an automaton.

The resulting automaton may or may not be synchronizing. What are necessary
and sufficient conditions for there to be an edge-colouring representing a
synchronizing automaton?

We will assume that the automaton can be synchronized in any given state by a
suitable reset word. A necessary and sufficient condition for this is that it
is strongly connected. (If so then, as in our dungeon, if we can synchronize at
some state, we can move from there to any other state.)

It is also necessary that the greatest common divisor of the lengths of
cycles in the digraph is~$1$. For suppose the g.c.d. of cycle lengths is
$d$. Choose any vertex $v$, and let $\Omega_i$ be the set of vertices
reachable from $v$ in a number of steps congruent to $i$ mod~$d$, for
$i=0,1,\ldots,d-1$. The sets $\Omega_i$ are pairwise disjoint, and so no
automaton based on the digraph can be synchronizing.

The conjecture that these two necessary conditions are also sufficient was
made in 1970 by Weiss and Adler~\cite{AW,AGW} in connection with symbolic
dynamics, and became known as the \emph{Road-Colouring Conjecture}. It was
proved by Avraham Trahtman in 2007~\cite{trahtman}:

\begin{theorem}
Let $D$ be a digraph which is strongly connected and has constant out-degree,
and suppose that the greatest common divisor of the cycle lengths in $D$
is~$1$. Then $D$ can be edge-coloured so as to produce a synchronizing
automaton.
\end{theorem}

\subsection{Synchronizing groups}

Looking at the extreme examples above for the \v{C}ern\'y conjecture, we
see that, of the two transitions, the first is a cyclic permutation, which
generates a transitive group on the set of states; the second is a
non-permutation.

This observation is the basis of the next definition~\cite{ar,arnold_steinberg}.
A permutation group $G$
on $\Omega$ is said to be \emph{synchronizing} if, whenever $f$ is a map on
$\Omega$ which is not a permutation, the monoid $\langle G,f\rangle$ is
synchronizing (that is, there is a word in $f$ and the elements of $G$ which
has rank~$1$).

We have abused language here since $G$ itself (regarded as a monoid) is not
a synchronizing monoid; but a permutation group cannot be a synchronizing
monoid, so hopefully the confusion will not be too great.

For example, the automorphism group of the Petersen graph is synchronizing.
This fact can be proved by considering all possible non-permutations on the
vertex set; but in the next section we will develop a technique to make it much
easier to check assertions like this.

More generally, we say that a permutation group $G$ \emph{synchronizes} a
non-permutation $f$ if $\langle G,f\rangle$ contains a map of rank $1$. Thus,
$G$ is synchronizing if it synchronizes every non-permutation.

The next theorem shows how synchronizing groups relate to the classical
notions of primitive and $2$-homogeneous groups.

\begin{theorem}
\begin{enumerate}\itemsep0pt
\item A synchronizing group is primitive.
\item A $2$-homogeneous group is synchronizing.
\end{enumerate}
\label{prim_synch}
\end{theorem}

\begin{proof}
(a) The simplest argument here is to recall the characterization of primitivity
based on Theorem~\ref{t:rystsov}: a permutation group $G$ of degree~$n$ is
primitive if and only if the monoid $\langle G,f\rangle$ is synchronizing for
any map $f$ of rank $n-1$.

Since we haven't proved this yet, we give a different proof. Suppose that $G$
is imprimitive. Let { $P$} be a non-trivial $G$-invariant partition of $\Omega$,
and let $A$ be a subset of $\Omega$ containing one point from each part of $P$
(so $A$ is a \emph{section}, or \emph{transversal}, of $P$.) Now the map $f$
that takes each point of $\Omega$ to the unique point of $A$ in the same part
of $P$ is not synchronized by $G$. For any word in $f$ and the elements of $G$
which contains $f$ at least once has the property that its image is a section
for $P$, so no such word can have rank $1$.

(b) Suppose that $G$ is $2$-homogeneous, and let $f$ be any non-permutation.
Let $r$ be the minimal rank of an element $h\in\langle G,f\rangle$, and suppose
for a contradiction that $r>1$. Choose two distinct points $x,y$ in the image
of $h$, and two points $u,v$ which are mapped to the same place by $f$. Then
choose $g\in G$ mapping $\{x,y\}$ to $\{u,v\}$. Then $hgf$ has smaller rank
than $h$, a contradiction. So $r=1$, as required.\qed
\end{proof}

The first part of this theorem can be improved:

\begin{prop}
A synchronizing group is basic.
\end{prop}

\begin{proof}
Let $G$ be non-basic, and suppose that $\Omega$ has been identified with the
set of $m$-tuples over a set $A$ of size $k$, in such a way that $G$ preserves
the identification (and so is embedded in $S_k\wr S_m$).

Let $f$ be the map which takes the $m$-tuple $(a_1,a_2,\ldots,a_m)$ to the
$m$-tuple $(a_1,a_1,\ldots,a_1)$ with all entries equal. Let $B$ be the image
of $f$. Then applying any element of $G$ to $B$ gives a set of $k$ elements
whose projections onto any coordinate form the whole of $A$; so following this
by $f$ gives us the set $B$ again. So no word in $f$ and $G$ can have rank
smaller than $k$, and $G$ fails to synchronize $f$.\qed
\end{proof}

Figure~\ref{f:grid} shows how the map works for $S_3\wr S_2$, the automorphism
group of the $3\times3$ grid.

\begin{figure}[htbp]
\begin{center}
\setlength{\unitlength}{1.5mm}
\begin{picture}(20,20)
\multiput(10,0)(10,0){2}{\circle{1}}
\multiput(0,10)(20,0){2}{\circle{1}}
\multiput(0,20)(10,0){2}{\circle{1}}
\multiput(0,0)(10,10){3}{\circle*{1.5}}
\multiput(0,0)(0,10){3}{\line(1,0){20}}
\multiput(0,0)(10,0){3}{\line(0,1){20}}
\multiput(0,8)(0,10){2}{\vector(0,-1){4}}
\put(10,2){\vector(0,1){4}}\put(10,18){\vector(0,-1){4}}
\multiput(20,2)(0,10){2}{\vector(0,1){4}}
\end{picture}
\end{center}
\caption{\label{f:grid}Failure to synchronize a square grid}
\end{figure}
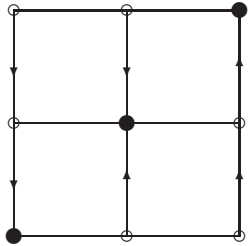

We conclude this section with examples to show that the inclusions just proved
do not reverse: we give examples of basic primitive groups which are not
synchronizing, and synchronizing groups which are not $2$-homogeneous.

Our examples are given by the symmetric groups $S_m$ for $m\ge5$, in their
action on the set of $2$-element subsets of $\{1,\ldots,m\}$.
\begin{enumerate}\itemsep0pt
\item
The group is primitive and basic for all $m\ge5$. For it is easy to see that
there are just two complementary orbital graphs: the vertex set is the set of
$2$-subsets of $\{1,\ldots,m\}$: two vertices are joined in the first graph
if they have non-empty intersection (so this is the \emph{line graph} of the
complete graph $K_m$), and in the second if they have empty intersection. Now
both of these graphs are connected.
\item
The group is not $2$-homogeneous. For the edges in the two orbital graphs are
not equivalent.
\item
The group is synchronizing if and only if $m$ is odd. We will defer the proof
of this assertion to the next section, when we will have another technique
available.
\end{enumerate}

\subsection{Section-regular partitions}

In the next section we will develop a very convenient combinatorial
characterization of synchronization. In the meantime we give another
characterization which was introduced in~\cite{ar} and developed by Peter
Neumann~\cite{neumann:sectionregular}.

Let $P$ be a partition of $\Omega$. A \emph{section}, or \emph{transversal},
for $P$ is a subset $A$ of $\Omega$ which meets every part of $P$ in a single
point. Recall that a partition $P$ is \emph{non-trivial} if it is not the
partition into singletons and not the partition with a single part.

Now let $G$ be a permutation group on $\Omega$. We say that the partition $P$
is \emph{section-regular} for $G$, with section $A$, if $Ag$ is a section for
$P$ for every $g\in G$.

In Figure~\ref{f:grid}, the partition into vertical lines is section-regular
for the group $S_3\wr S_2$ of automorphisms, with the diagonal as a section.

\begin{theorem}\label{t:sec.reg.char}
A permutation group $G$ is synchronizing if and only if it has no non-trivial
section-regular partition.
\end{theorem}

\begin{proof}
Suppose first that $G$ is non-synchronizing. Let $f$ be a map such that
$\langle G,f\rangle$ contains no map of rank~$1$. We may suppose, without
loss of generality, that $f$ is an element of minimal rank (say $r$) in
$\langle G,f\rangle$. Let $P$ be the kernel of $f$, and $A$ the image of $f$.
If $Ag$ is not a section for $P$, then $Ag$ meets fewer than $r$ kernel
classes of $f$, and so $fgf$ has rank smaller than $r$, a contradiction. So
$P$ is section-regular with section $A$.

Conversely, suppose that $P$ is a section-regular partition for $G$, with
section $A$. Then for any $x\in\Omega$, there is a unique $y\in A$ which lies
in the same class of $P$ as does $x$. Define a map $f$ by the rule that $xf=y$
when the above holds. Now $Ag$ is a section for $P$, so $Agf=A$, for any
$g\in G$. An easy induction shows that no element of $\langle G,f\rangle$ has
rank smaller than $|A|$.\qed
\end{proof}

The next two results are due to Peter Neumann~\cite{neumann:sectionregular}.

\begin{theorem}
A section-regular partition for a transitive permutation group $G$ is
uniform.
\label{t:secreg_unif}
\end{theorem}

\begin{proof}
Let $P$ be section-regular for $G$, with section $A$. Suppose that $B$ is a
part of $P$. Count triples $(a,b,g)$, where $a\in A$, $b\in B$ and $g\in G$
satisfies $ag=b$; there are $|A|\cdot|B|$ choices of $a$ and $b$ and then
$|G|/n$ choices of $g$ (since the set of elements of $G$ mapping $a$ to $b$
is a coset of the stabilizer of $a$). On the other hand, for every $g\in G$,
$|Ag\cap B|=1$, so there is a unique pair $(a,b)$ satisfying the condition.
Thus $|A|\cdot|B|=n$; in particular, $|B|$ is independent of the part $B$ of
$P$ chosen.\qed
\end{proof}

\begin{cor}\label{c:min.rank.unif}
A map $f$ of minimal rank subject to being not synchronized by the transitive
group $G$ has uniform kernel.
\end{cor}

\begin{proof}
This follows from the two preceding theorems.\qed
\end{proof}

It follows that any transitive group of prime degree is synchronizing, a result due originally to Pin~\cite{Pincerny} that can be considered the first result in the theory of synchronizing groups.

Theorem~\ref{t:secreg_unif} and the proof of Theorem~\ref{t:sec.reg.char} in fact yield the following corollary.

\begin{cor}\label{c:uniform.min.ideal}
Let $M$ be a transformation monoid on a finite set $\Omega$ with a transitive group of units.  Then each element of $M$ of minimal rank has a uniform kernel.
\end{cor}

\subsection{The \v{C}ern\'y conjecture revisited}

Can we prove at least some instances of the \v{C}ern\'y bound for transformation
monoids of the form $\langle G,f\rangle$, where $G$ is a synchronizing
permutation group?

Since the permutations in this monoid are precisely the elements in $G$, and
these by themselves will not synchronize, it seems reasonable to build a
word of the form $fg_1fg_2\cdots fg_rf$, where we use $f$ to reduce the rank
of the partial product and $g_i$ to ensure that the next application of $f$
does so. Note that the rank of $hf$ is strictly less than the rank of $f$
if and only if two points of the image of $h$ lie in the same kernel class
of $f$. So, if the rank of a product of the above form is at most $k$,
choose $g_{r+1}$ to map two points in the image of the product into a kernel
class of $f$, and then the rank of the product $fg_1f\cdots g_{r+1}f$ will
be at most $k-1$.

If this strategy succeeds, we will have a reset word with at most $n-1$
occurrences of $f$. The task now is to bound the lengths of the expressions
for $g_1,\ldots$ in terms of the given generators of $G$, for which
hopefully group theory will help.

We want to avoid the case where there is a set $A$ (the image of a subword
of the product), all of whose $G$-images are partial sections for the kernel
of $P$. This is where conditions on $G$ like ``synchronizing'' are relevant.

\clearpage

\section{Graph endomorphisms}

In this section, we define endomorphisms of graphs, and use them to give
characterizations of synchronizing monoids and groups. This result gives
us the simplest available test for the synchronizing property of permutation
groups. We will illustrate by returning to the example of the symmetric group
$S_m$ acting on $2$-sets, and showing that it is synchronizing if and only if
$m$ is odd.

\subsection{Cliques, colourings and endomorphisms}

A \emph{relational structure} consists of a set carrying a number of relations
of specified arities. The most important example for us is a graph, a set with
a single binary relation.

A \emph{homomorphism} $f:A\to B$ between relational structures $A$ and $B$ is
a map between the underlying sets which preserves all instances of the relation.
Thus, a graph homomorphism maps edges to edges, but its action on non-edges
is not specified; it could map a non-edge to an edge, or to a non-edge, or to
a single vertex. (If the graph has no loops, then edges cannot be collapsed
to single vertices.)

Some important graph parameters can be expressed in terms of homomorphisms.
The \emph{clique number} $\omega(\Gamma)$ of a graph $\Gamma$ is the number of
vertices in the largest complete subgraph of $\Gamma$, that is, the largest
number of vertices such that any two are adjacent. A \emph{(proper) colouring}
of $\Gamma$ is an assignment of colours from a set $C$ to the vertices in
such a way that the ends of any edge have different colours. The
\emph{chromatic number} $\chi(\Gamma)$ of $\Gamma$ is the smallest number of
colours required for a proper colouring of $\Gamma$.

It is clear that $\chi(\Gamma)\ge\omega(\Gamma)$ for any graph $\Gamma$, since
the vertices of a clique must all have different colours.

Recall that $K_r$ is the complete graph with $r$ vertices.

\begin{theorem}
\begin{enumerate}\itemsep0pt
\item There is a homomorphism from $K_r$ to $\Gamma$ if and only if
$\omega(\Gamma)\ge r$.
\item There is a homomorphism from $\Gamma$ to $K_r$ if and only if
$\chi(\Gamma)\le r$.
\end{enumerate}
\end{theorem}

\begin{proof}
(a) The images of the vertices of $K_r$ under a homomorphism must all be
distinct.

(b) Let $C$ be the vertex set of $K_r$. We think of $C$ as a set of colours,
and the homomorphism $f$ assigns to $v$ the colour $f(v)$. Now the definition
of a homomorphism shows that this is a proper colouring.\qed
\end{proof}

\begin{cor}
For a graph $\Gamma$, the following are equivalent:
\begin{enumerate}\itemsep0pt
\item $\omega(\Gamma)=\chi(\Gamma)$;
\item there are homomorphisms in both directions between $\Gamma$ and a
complete graph.
\end{enumerate}
\end{cor}

For a detailed study of graph homomorphisms, we recommend~\cite{hell_nesetril}.

An \emph{endomorphism} of $\Gamma$ is a homomorphism from $\Gamma$ to itself.
The composition of endomorphisms is an endomorphism, and the identity map is
an endomorphism; so the set of endomorphisms of $\Gamma$ is a transformation
monoid on the vertex set of $\Gamma$, denoted by $\End(\Gamma)$.

In line with our previous practice, we write endomorphisms on the right.

\subsection{Graphs and endomorphism monoids}

The map $\Gamma\to\End(\Gamma)$ is a mapping from graphs to transformation
monoids. Unfortunately, it is not a functor in any reasonable sense; this is
also the case for the next map we define, which goes in the other direction.

Let $M$ be a transformation monoid on a set $\Omega$. We define a graph
$\Gamma=\Gr(M)$ on the vertex set $\Omega$ by the following rule for
adjacencies:
\begin{quote}
$v$ and $w$ are adjacent in $\Gr(M)$ if and only if there does not exist
$f\in M$ with $vf=wf$.
\end{quote}

This correspondence has various nice properties:

\begin{theorem}
\begin{enumerate}\itemsep0pt
\item For any monoid $M$, the graph $\Gr(M)$ has the properties
\begin{enumerate}\itemsep0pt
\item $M\le\End(\Gr(M))$;
\item $\omega(\Gr(M))=\chi(\Gr(M))$.
\end{enumerate}
\item If $M_1\le M_2$, then $\Gr(M_2)$ is a spanning subgraph of
$\Gr(M_1)$.
\end{enumerate}
\end{theorem}

\begin{proof}
(a) (i) Let $f\in M$; we have to show that $f$ is an endomorphism of $\Gr(M)$,
so suppose not. Then there exists an edge $\{v,w\}$ of $\Gr(M)$ which is not
preserved by $M$. By definition, $vf\ne wf$; so this can only happen if
$\{vf,wf\}$ is a non-edge of $\Gr(M)$. But then, by definition, there exists
$h\in M$ such that $(vf)h=(wf)h$. Then $fh\in M$ and $v(fh)=w(fh)$,
contradicting the fact that $\{v,w\}$ is an edge of $\Gr(M)$.

(a) (ii) Now let $f$ be an element of $M$ of smallest possible rank. Let
$A=\Im(f)$. No element of $M$ can map two points of $A$ to the same
place, since if $h$ did so then $fh$ would have smaller rank than $f$. So
by definition, $A$ is a clique in $\Gr(M)$. Since $f\in\End(\Gr(M))$, we see
that $f$ induces a proper colouring of $\Gr(M)$ with $|A|$ colours.

(b) Clearly adding extra {endo}morphisms cannot produce new edges which were
not there before.\qed
\end{proof}

\subsection{Characterization of synchronizing monoids}

Now we can give our characterization of synchronizing monoids.

\begin{theorem}\label{th4.4}
Let $M$ be a transformation monoid on $\Omega$. Then $M$ is non-synchronizing
if and only if there exists a non-null graph $\Gamma$ on the vertex set
$\Omega$ with $M\le\End(\Gamma)$. If such a graph $\Gamma$ exists, then we may
choose it so that $\omega(\Gamma)=\chi(\Gamma)$.
\end{theorem}

\begin{proof}
If $M\le\End(\Gamma)$ for any non-null graph $\Gamma$, then $M$ is not
synchronizing, since edges of $\Gamma$ cannot be collapsed by elements of $M$.

Conversely, if $M$ is non-synchronizing, let $\Gamma=\Gr(M)$. Suppose that
$\Gamma$ is the null graph. Then any pair of points of $\Omega$ are mapped to
the same place by some element of $M$. Let $f$ be an element of least possible
rank in $M$. If the rank of $f$ is greater than $1$, choose $x,y\in\Im(f)$,
and $h\in M$ with $xh=yh$; then $fh$ has smaller rank than $f$. So $f$ has
rank~$1$, and $M$ is synchronizing, a contradiction. So $\Gamma$ is non-null.

Now the remaining assertions of the theorem come from the properties of
$\Gr(M)$ from the preceding subsection.\qed
\end{proof}

\begin{cor}\label{cor4.5}
Let $G$ be a transitive permutation group on $\Omega$. Then $G$ is
synchronizing if and only if every non-trivial $G$-invariant graph $\Gamma$
has $\omega(\Gamma)\ne\chi(\Gamma)$.
\end{cor}

\begin{proof}
If $\Gamma$ is a non-trivial $G$-invariant graph with $\omega(\Gamma)=\chi(\Gamma)=r$, then there are graph homomorphisms $f\colon \Gamma\to K_r$ and $g\colon K_r\to \Gamma$.  Composition of these homomorphisms provides a singular endomorphism $h\colon \Gamma\to \Gamma$ and so $\langle G\cup h\rangle$ is not synchronizing by the previous theorem. \qed
\end{proof}

For example, the automorphism group of the Petersen graph is edge-transitive
and nonedge-transitive; so we only have to check the Petersen graph and its
complement. It is not hard to show that the Petersen graph has clique number
$2$ and chromatic number $3$, while its complement has clique number $4$ and
chromatic number $5$. So the automorphism group is synchronizing.

This corollary is the basis for the best computational test for synchronization.
The test runs as follows. Given a transitive permutation group $G$ on $\Omega$,
do the following:
\begin{enumerate}\itemsep0pt
\item Find all the non-trivial $G$-invariant graphs. There are $2^r-2$ such
graphs, where $r$ is the number of orbits of $G$ on $2$-element subsets of
$\Omega$, since the edge set of a $G$-invariant graph is a union of orbits of
$G$.
\item Test each graph $\Gamma$ to see whether $\omega(\Gamma)=\chi(\Gamma)$.
If one does, then $G$ is not synchronizing; otherwise it is synchronizing.
\end{enumerate}

This algorithm looks extremely inefficient. The first stage generates
exponentially many graphs to be checked; and computing the clique number and
chromatic number of a graph are both \textsf{NP}-complete problems.

However, in practice, ``interesting'' permutation groups often have
comparatively few orbits on $2$-sets, so $r$ is small; and the graphs which
have to be tested have large automorphism groups, which can be exploited to
reduce the computational burden in the second step.

In the next section, we will see how the algorithm can be slightly improved.

\begin{example}
Here is an example promised earlier. Let $G$ be the symmetric group of degree
$m\geq 5$, in its action on $2$-element subsets of $\{1,\ldots,m\}$. There
are only two orbits on pairs of $2$-element subsets: the subsets may intersect
in a point, or they may be disjoint. So we have two $G$-invariant graphs to
consider: the line graph of $K_m$ and its complement.

Let $\Gamma$ be the line graph of $K_m$. The clique number of $\Gamma$ is
$m-1$; a typical maximal clique is $\{\{1,i\}:i=2,\ldots,m\}$. When can the
chromatic number be $m-1$? Pairs with the same colour must be disjoint,
so there are at most $\lfloor m/2\rfloor$ pairs in a colour class. If $m$ is
odd, this number is $(m-1)/2$, so at least ${m\choose 2}/((m-1)/2)=m$ colours
are required. If $m$ is even, we can have $m/2$ edges in a colour class, and
$m-1$ colours. This can be realized as follows. Take a regular $(m-1)$-gon
in the plane. The edges and diagonals fall into $m-1$ parallel classes, with
$(m/2)-1$ pairs in each class, and one point omitted from each class. Assign
one colour to all the edges in each class. Now add an extra point $\infty$,
and give the colour of a class $C$ to the pair consisting of $\infty$ and
the point omitted by $C$.
For example, if $m=6$ and we label the vertices of the regular pentagon by $1,2,3,4,5$ in counterclockwise ordering, then the five colour classes are $\{\{1,2\},\{3,5\},\{4,\infty\}\}$, $\{\{2,3\},\{1,4\},\{5,\infty \}\}$, $\{\{3,4\}, \{5,2\}, \{1,\infty\}\}$, $\{\{4,5\},\{1,3\},\{2,\infty\}\}$, and
$\{\{1,5\},\{2,4\},\{3,\infty\}$.
So this graph $\Gamma$ has
$\omega(\Gamma)=\chi(\Gamma)$ if and only if $m$ is odd.

Now let $\Gamma$ be the complement of the line graph of $K_m$. Now a clique
consists of disjoint pairs, so the largest clique has size $\lfloor m/2\rfloor$.
But $\Gamma$ cannot be coloured with this many colours. For the colour classes
must be cliques in $L(K_m)$; we saw that such cliques have size at most $m-1$
and so we would need at least $m/2$ in a partition. So we could only achieve
the bound if $m$ were even and the cliques were pairwise disjoint. But the
cliques of size $m-1$ consist of all pairs containing a given point; and the
cliques defined by points $a$ and $b$ have the pair $\{a,b\}$ in common. So
this graph never has clique number and chromatic number equal.

We conclude that, for $m\ge5$, $S_m$ acting on $2$-sets is synchronizing if
and only if $m$ is even.

We remark that, in fact, the chromatic number of the complement of $L(K_m)$
is known to be $m-2$; this is a special case of a theorem of Lov\'asz
\cite{lovasz}.
\end{example}

Theorem \ref{th4.4} and Corollary \ref{cor4.5} 
have been used in a number of places, for example,
\cite{abcrs,ac,randomsynch,SchSil}, for investigating synchronizing groups.

\subsection{Rystsov's Theorem}

We illustrate these concepts by proving Theorem~\ref{t:rystsov}.
With the terminology we have introduced, the theorem states:

\begin{theorem}
A transitive permutation group of degree $n$ is primitive if and only if it
synchronizes every map of rank $n-1$.
\end{theorem}

\begin{proof}
Suppose first that $G$ fails to synchronize the map $f$ of rank $n-1$. Then
there exist $a$ and $b$ such that $af=bf$, but $f$ is injective
on any subset not containing both $a$ and $b$. Suppose that $\Gamma$ is a
non-trivial graph with $\langle G,f\rangle\le\End(\Gamma)$. Since $G$ is
transitive, $\Gamma$ is regular; suppose that every vertex has degree $d$.
Since $af=bf$, we see that $\{a,b\}$ is a non-edge of $\Gamma$; so
$f$ maps the neighbours of $a$ bijectively to the neighbours of $af$.
Similarly, $f$ maps the neighbours of $b$ bijectively to the neighbours of
$bf=af$. Hence $a$ and $b$ have the same neighbours. Now the relation
$\equiv$, defined by $x\equiv y$ if and only if $x$ and $y$ have the same
neighbours, is a $G$-invariant equivalence relation; so $G$ is imprimitive.

Conversely, suppose that $G$ is imprimitive; let $P$ be a non-trivial
$G$-invariant partition. Let $a$ and $b$ be two points in the same part of $P$.
Define a map $f$ by
\[xf=\cases{x & if $x\ne b$;\cr a & if $x=b$.\cr}\]
It is easy to see that $f$ has rank $n-1$ and is not synchronized by $G$
(it is an endomorphism of the complete multipartite graph with multipartition
$P$).\qed
\end{proof}

In the paper~\cite{abcrs}, the authors extend this result to show that a
primitive group of degree $n$ synchronizes any map of rank $n-4$ or greater.
The non-basic group $S_3\wr S_2$, the automorphism group of the $3\times3$
square grid, has degree~$9$ and fails to synchronize a map of rank~$3$
(the grid graph has clique number and chromatic number~$3$); so this result
is within one of best possible.

\subsection{Cores and hulls}

The \emph{core} of a graph $\Gamma$ is the smallest graph $\Delta$ with the
property that there are homomorphisms from $\Gamma$ to $\Delta$ and from
$\Delta$ to $\Gamma$. It is known that every graph has a core, which is
unique up to isomorphism; moreover, the core is an induced subgraph of
$\Gamma$, and there is a \emph{retraction} from $\Gamma$ to its core (an
endomorphism which acts as the identity on its image).

Cores play an important role in the theory of graph homomorphisms, see
\cite{hell_nesetril}. We remark that the graphs which we used in our
characterization of synchronizing monoids can be defined as the graphs whose
cores are complete. The following well-known result in graph theory~\cite[Theorem~3.9]{ht}, can
be viewed as part of the theory of synchronizing groups.

\begin{theorem}
Let $\Gamma$ be a vertex-transitive graph. Then the retraction 
from $\Gamma$ to its core is uniform; in particular, the number of vertices in
$\mathrm{Core}(\Gamma)$ divides the number of vertices of $\Gamma$.
\end{theorem}

\begin{proof}
The retraction from $\Gamma$ to its core is a minimal rank element
of $\End(\Gamma)$ and so the theorem follows from Corollary~\ref{c:uniform.min.ideal}.\qed
\end{proof}

A ``dual'' concept is that of the \emph{hull} of a graph, introduced in
\cite{ck}. It is defined by
\[\Hull(\Gamma)=\Gr(\End(\Gamma));\]
in other words, two vertices are adjacent in the hull of $\Gamma$ if and only
if no endomorphism of $\Gamma$ collapses them to the same point.

Some of its properties are given by the following result:

\begin{theorem}
\begin{enumerate}\itemsep0pt
\item $\Gamma$ is a spanning subgraph of $\Hull(\Gamma)$.
\item The core of $\Hull(\Gamma)$ is a complete graph on the vertices of
the core of~$\Gamma$.
\item $\End(\Gamma)\le\End(\Hull(\Gamma))$ and $\Aut(\Gamma)\le
\Aut(\Hull(\Gamma))$.
\end{enumerate}
\end{theorem}

\begin{proof}
(a) This just says that every edge of $\Gamma$ is an edge of its hull, which
is clear since endomorphisms do not collapse edges.

(b) The core of $\Gamma$ is the image of an endomorphism of $\Gamma$ of
minimal rank. Thus endomorphisms of $\Gamma$ cannot identify two vertices of
the core, so it induces a clique in $\Hull(\Gamma)$. This clique is the image
of an endomorphism of $\Hull(\Gamma)$, and it is clear that no endomorphism
can have smaller image; so it is the core of $\Hull(\Gamma)$.

(c) Putting $M=\End(\Gamma)$, we know that $M\le\End(\Gr(M))$, which gives
the first inequality; the second follows.\qed
\end{proof}

Thus, passing from a graph to its hull cannot decrease the symmetry, but
might increase it in some cases.

\begin{example}
Let $\Gamma$ be the path of length~$3$, shown in Figure~\ref{f:hull}.

\begin{figure}[htbp]
\begin{center}
\setlength{\unitlength}{1mm}
\begin{picture}(30,30)
\multiput(0,0)(30,0){2}{\circle*{2}}
\multiput(0,30)(30,0){2}{\circle*{2}}
\multiput(0,0)(30,0){2}{\line(0,1){30}} \put(0,30){\line(1,0){30}}
\put(-5,0){$x$} \put(33,0){$y$}
\multiput(0,0)(4,0){8}{\line(1,0){2}}
\end{picture}
\end{center}
\caption{\label{f:hull}A graph and its hull}
\end{figure}
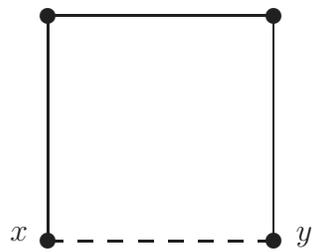

No homomorphism can identify $x$ and $y$, so they are
joined in the hull.

Note the increase in symmetry: $|\Aut(\Gamma)|=2$ but
$|\Aut(\Hull(\Gamma))|=8$.
\end{example}

\clearpage

\section{Related concepts}

The definition of synchronization can be varied in several ways, giving rise
to several closely-related concepts. We consider some of these in this section
of the paper. Perhaps the most interesting is the property of being almost
synchronizing. The first examples showing that this is not equivalent to
primitivity were found very recently.

We also define some measures of how far a given group is from being
synchronizing.

\subsection{Almost synchronizing groups}

We saw in the preceding section that the symmetric group $S_m$ acting on
$2$-sets is primitive but not synchronizing if $m$ is even and $m\ge6$.

However, the only maps that it fails to synchronize are the proper
endomorphisms of $L(K_m)$, which all have rank $m-1$.

To take another example, consider the (non-basic primitive) group $S_k\wr S_m$
for $k\ge3$, $m\ge2$. This group fails to be synchronizing, but the situation
is more complicated than the preceding one. Consider the \emph{Hamming graph}
$H(m,k)$, whose vertices are all $m$-tuples over an alphabet of size $k$,
and two $m$-tuples are adjacent if they agree in all but one position. The
automorphism group of this graph is the wreath product $S_k\wr S_m$.

Now, for any $d$ with $1\le d\le m$, we can construct an endomorphism $f_d$
with rank $k^d$, as follows. Choose the alphabet to be a group of order $k$,
for example, the additive group of the integers mod~$k$. Now set
\[(a_1,\ldots,a_m)f_d=(a_1,\ldots,a_{d-1},a_d+a_{d+1}+\cdots+a_m,0,\ldots,0).\]
It is easily verified that changing one coordinate in an $m$-tuple changes one
coordinate in its image, so $f_d$ is an endomorphism; its image is the
Hamming graph $H(d,k)$ with $k^d$ vertices, as required.

These examples suggested the following definition. Let us
call a permutation group $G$ \emph{almost synchronizing} if every map
which is not uniform (i.e. not all its kernel classes have the same size)
is synchronized by $G$. We note that an almost synchronizing group is primitive.
For suppose that $G$ is imprimitive, and preserves the non-trivial partition
$P$. Form the multipartite graph in which two vertices are adjacent if they
lie in different parts of $P$. This graph is preserved by $G$, but we can
collapse vertices within each part of $P$ arbitrarily by endomorphisms.

On the strength of this and examples like those above, it was conjectured
that every primitive group is almost synchronizing.

This was very recently shown to be false~\cite{abcrs}. We describe here a
general construction from that paper.

Consider the following two graph products. Let $\Gamma$ and $\Delta$ be graphs
with vertex sets $A$ and $B$ respectively.
\begin{enumerate}\itemsep0pt
\item The \emph{cartesian product} $\Gamma\cart\Delta$ has vertex set
$A\times B$; there is an edge from $(a_1,b_1)$ to $(a_2,b_2)$ if either
$a_1=a_2$ and $b_1$ is adjacent to $b_2$ in $\Delta$, or $a_1$ is adjacent to
$a_2$ in $\Delta$ and $b_1=b_2$.
\item The \emph{categorical product} $\Gamma\times\Delta$ also has vertex set
$A\times B$; but there is an edge from $(a_1,b_1)$ to $(a_2,b_2)$ if there
are edges from $a_1$ to $a_2$ in $\Gamma$ and from $b_1$ to $b_2$ in $\Delta$.
\end{enumerate}
The notation for these products is chosen so that the product symbol represents
the corresponding product of two edges.

For example, the Cartesian product of two copies of $K_r$ is the Hamming
graph $H(2,r)$, while the categorical product is the complement of $H(2,r)$.

Now here is a flexible construction of primitive graphs with non-uniform
endomorphisms.

Let $\Gamma$ be a graph whose automorphism group acts primitively on its
vertices. Then the Cartesian product $\Gamma\cart\Gamma$ is also
vertex-primitive, with automorphism group $\Aut(\Gamma) \wr S_2$. In addition,
if the chromatic and clique number of $\Gamma$ are both equal to $k$, then
$V(\Gamma)$ can be partitioned into $k$ colour classes of equal size, say
$V_1$, $V_2$, $\ldots$, $V_k$, and there is a surjective homomorphism
$\Gamma\cart\Gamma\rightarrow K_k \cart K_k$ with  kernel classes
$V_{i} \times V_{j}$ for  $1 \leq i,j \leq k\}$. Moreover, there is a
homomorphism from $\Gamma$ to $\Gamma\cart\Gamma$; simply take
the second coordinate to be fixed.

Therefore if there is a homomorphism $f: K_k \cart K_k \rightarrow \Gamma$,
then by composing homomorphisms
\[
\Gamma\cart\Gamma \rightarrow K_k \cart K_k \ra^f \Gamma\rightarrow\Gamma\cart\Gamma,
\]
there is an endomorphism of $\Gamma \cart \Gamma$. Moreover, if the homomorphism $f$ is
non-uniform, then the endomorphism is also non-uniform; and its rank is equal
to the rank of $f$.

We can obtain examples by taking $\Gamma$ to be the complement of
$K_k\cart K_k$, that is, $\Gamma=K_k\times K_k$. Now a homomorphism $f$
from $K_k\cart K_k$ to $K_k\times K_k$ is given by
$(u,v)\mapsto(g(u,v),h(u,v))$, where the
two coordinate functions $g(u,v)$ and $h(u,v)$ satisfy the homomorphism
requirement that if $(u,v)$ and $(u',v')$ agree in one position but not the
other, then $g(u,v)\ne g(u',v')$ and $h(u,v)\ne h(u',v')$.

In other words, $g$ and $h$ are \emph{Latin squares} of order $k$: that is,
they define $k\times k$ arrays with entries from a set of size $k$ such that no
entry is repeated in a row or a column. But note that there is no connection
between the two Latin squares!

The rank of the homomorphism is the number of ordered pairs of symbols which
arise when the two Latin squares are superimposed. The possibilities have
been determined by Colbourn, Zhu and Hang~\cite{cz,zz}:

\begin{theorem}
There are two Latin squares of order $k$ whose superposition gives $r$ ordered
pairs of symbols if and only if $r=k$, or $r=k^2$, or $k+2\le r\le k^2-2$,
with the following exceptions:
\begin{enumerate}\itemsep0pt
\item $k=2$ and $r=4$;
\item $k=3$ and $r \in \{5,6,7\}$;
\item $k=4$ and $r \in \{7,10,11,13,14\}$;
\item $k=5$ and $r \in \{8,9,20,22,23\}$;
\item $k=6$ and $r \in \{33,36\}$. \qed
\end{enumerate}
\end{theorem}

Note that the case $r=k$ corresponds to using the same Latin square twice,
while $r=k^2$ corresponds to a pair of orthogonal Latin squares. In these
cases, the endomorphism constructed is uniform; but in general it is not.

Here is another example from~\cite{abcrs}. Though this construction is not as
flexible as the previous one, it was the first one found, produces a primitive
group of smallest possible degree (namely $45$) which is not almost
synchronizing, and also produces a non-uniform map of smallest possible
rank (namely $5$) which fails to be synchronized by a primitive group.

We start with three particular graphs.
Two of these examples are two of the three ``remarkable graphs'' discussed
by Biggs~\cite{biggs}, namely the Petersen graph (see Figure~\ref{f:peter})
and the Biggs--Smith graph; the third is the \emph{Tutte--Coxeter graph} on
$30$ vertices~\cite{tutte,coxeter}. All are trivalent graphs without
triangles, and have proper $3$-edge colourings; and their automorphism groups
act primitively on the edges. We consider their line graphs. These are
$4$-valent vertex-primitive graphs with chromatic number~$3$; the closed
neighbourhood of a vertex is the \emph{butterfly graph} shown in
Figure~\ref{f:butter}.

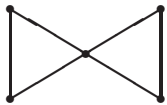
\begin{figure}[htbp]
\begin{center}
\setlength{\unitlength}{1mm}
\begin{picture}(20,12)
\thicklines
\multiput(0,0)(20,0){2}{\circle*{1}}
\multiput(0,12)(20,0){2}{\circle*{1}}
\put(10,6){\circle*{1}}
\multiput(0,0)(20,0){2}{\line(0,1){12}}
\put(0,0){\line(5,3){20}}
\put(20,0){\line(-5,3){20}}
\end{picture}
\end{center}
\caption{\label{f:butter}The butterfly}
\end{figure}

Take a $3$-colouring of one of these line graphs, with colour classes
$C_1$, $C_2$, $C_3$. If $v$ is a vertex in $C_3$, then two neighbours of $v$
lie in each of the other two classes; so the induced subgraph on
$C_2\cup C_3$ has valency~$2$, so is a union of cycles. If it is disconnected,
then taking alternate vertices in the cycles we get partitions
$C_2=C'_2\cup C''_2$ and $C_3=C'_3\cup C''_3$, with edges only between
$C'_i$ and $C''_i$ for $i=2,3$. In other words, we have a homomorphism from
$L(\Gamma)$ onto the butterfly, where $C_1$ maps to the ``body'' and $C'_i$ and
$C''_i$ to the two vertices of a ``wing'' for $i=1,2$.

The line graph of the Petersen graph has a unique $3$-colouring up to
isomorphism, and $C_2\cup C_3$ turns out to be connected. But in each of
the other cases, the required homomorphism exists. For the line graph of
the Tutte--Coxeter graph (with $45$ vertices), we obtain a homomorphism
with kernel classes of sizes $5$, $5$, $10$, $10$, and $15$. Since the
butterfly is a subgraph, we can realise this map as an endomorphism of
$L(\Gamma)$; so the automorphism group of $L(\Gamma)$ is not almost
synchronizing.

Further analysis shows that this graph has an endomorphism onto the ``double
butterfly'' (a triangle with triangles attached at two of its vertices)
with rank $7$.

The line graph of the Biggs--Smith graph is even more prolific. It has
non-uniform endomorphisms of ranks $5$, $7$ and $9$ (the last onto the
``triple butterfly'', a triangle with triangles attached at each vertex).

An almost synchronizing group is primitive (since an imprimitive group
preserves a complete multipartite graph, which has non-uniform endomorphisms).
However, it need not be basic: the automorphism group of the square grid
(Figure~\ref{f:grid}) fails to synchronize only uniform maps of rank~$3$, but
is clearly not basic. In the other direction, the butterfly examples in this
section show that primitive groups which fail to be almost synchronizing may or
may not be basic. So there is no implication between these concepts.

\subsection{Separating groups}

The next variant is based on the following theorem. A set of vertices of a
graph $\Gamma$ is \emph{independent} if it contains no edge (so the induced
subgraph on the set is a null graph). Let $\alpha(\Gamma)$ be the
\emph{independence number} of $\Gamma$, the size of the largest independent
set. Note that $\alpha(\Gamma)=\omega(\overline{\Gamma})$.

\begin{theorem}
Let $\Gamma$ be a graph on $n$ vertices, whose automorphism group acts
transitively on vertices. Then $\alpha(\Gamma)\cdot\omega(\Gamma)\le n$.
Equality holds if and only if any clique and any independent set of maximum
size have non-empty intersection.
\end{theorem}

\begin{proof}
Let $A$ be a clique and $B$ an independent set. Count triples $a,b,g$
with $a\in A$, $b\in B$, and $ag=b$. There are $|A|$ choices for $a$,
$|B|$ choices for $B$, and $|G|/|\Omega|$ choices for $g$ (since the set
of such $g$ is a coset of the stabilizer of a point, and all such cosets
have the same cardinality). On the other hand, there are $|G|$ choices
for $g$, and for each choice there is at most one element in $Ag\cap B$,
since a clique and an independent set cannot have more than one point in
common. So
\[|A|\cdot|B|\cdot|G|/n \le |G|,\]
giving $|A|\cdot|B|\le n$. If equality holds, then $Ag\cap B\ne\emptyset$
for all $g\in G$, and in particular $A\cap B\ne\emptyset$; this holds for
any clique and independent set.\qed
\end{proof}

Let $G$ be a transitive permutation group on $\Omega$. If $A$ and $B$ are
subsets of $\Omega$ which satisfy $|A|\cdot|B|<|\Omega|$, then $A$ and $B$
can be \emph{separated} by $G$: that is, there exists $g\in G$ such that
$Ag\cap B=\emptyset$. (This was proved by Peter Neumann
\cite{neumann:finitary}; an elementary proof appears in~\cite{bbmn}.) What
happens if $|A|\cdot|B|=|\Omega|$?

We say that $G$ is \emph{non-separating} if, given two subsets $A,B$ of
$\Omega$ (other than singleton sets and $\Omega$) satisfying
$|A|\cdot|B|=|\Omega|$, we have {$|Ag\cap B|=1$} for all $g\in G$. If
this is not the case, then since the average size of this intersection is $1$
(see~Theorem~\ref{t:average.value} below for a  proof of a generalization of
this), there must be some $g\in G$ for which $Ag\cap B=\emptyset$;
in this case, we say that $G$ is \emph{separating}.

\begin{theorem}
Let $G$ be a transitive permutation group on $\Omega$. Then $G$ is
non-separating if and only if there is a non-trivial graph $\Gamma$ on the
vertex set $\Omega$ satisfying $\omega(\Gamma)\alpha(\Gamma)=|\Omega|$ and
$G\le\Aut(\Gamma)$.
\end{theorem}

\begin{proof}
If a graph $\Gamma$ with these properties exists, then taking $A$ and $B$ to
be a clique and an independent set of maximum size, we see that $G$ cannot
separate them; so it is not separating.

Conversely, suppose that $G$ is not separating, and let $A$ and $B$ be the
sets specified in the definition, so that $|A|\cdot|B|=|\Omega|$. Now let
$\Gamma$ be the graph on $\Omega$ whose edges are all the images under $G$ of
pairs of vertices in $A$. Then $G\le\Aut(\Gamma)$; $A$ is a clique in $\Gamma$;
and, since no element of $G$ maps $A$ to a set intersecting $B$ in two points,
$B$ is an independent set in $\Gamma$.\qed
\end{proof}

\begin{cor}
A separating group is synchronizing.
\end{cor}

\begin{proof}
Suppose that $G$ is not synchronizing, and let $\Gamma$ be a graph with
clique number equal to chromatic number which witnesses this. Then we have
$G\le\Aut(\Gamma)$. Each colour class is an independent set, and the average
size of the colour classes is $|\Omega|/|\chi(\Gamma)|=|\Omega|/|\omega(\Gamma)|$; but no independent set can exceed this size, so
$\alpha(\Gamma)=|\Omega|/|\omega(\Gamma)|$.

Alternatively, if $P$ is a section-regular partition with section $A$ and $B$ is any part of $P$, then $|A|\cdot |B|=|\Omega|$ and $|Ag\cap B|=1$ for all $g\in G$.\qed
\end{proof}

The converse of this corollary is false; we will see an example later.

A transitive group of prime degree is obviously separating since there are no non-trivial subsets $A,B$ with $|A|\cdot |B|=|\Omega|$.

\begin{cor}
A transitive group of prime degree is separating and hence synchronizing.
\end{cor}

We can now make a small improvement in the algorithm for testing synchronization. As a comment on this, both clique number and chromatic number are
\textsf{NP}-hard, but in practice finding clique number is very much easier
than finding chromatic number.
(There are results in the theory of parameterized complexity which support
this assertion. The $k$-clique problem lies in the complexity class W[1] and
is complete for this class -- see Downey and Fellows~\cite{df} --
but $k$-colouring is NP-complete even for $k=3$.)

We modify the previous algorithm as follows.
\begin{enumerate}\itemsep0pt
\item The $2^r-2$ $G$-invariant graphs (where $r$ is the number of $G$-orbits
on $2$-element subsets of $\Omega$) fall into $2^{r-1}-1$ complementary pairs.
\item For each pair, find the clique numbers of the two graphs in the pair. If
their product is $|\Omega|$, then $G$ is not separating; remember this pair of
graphs. If this never happens, then $G$ is separating (and hence synchronizing).
\item Now we just have to look at the graphs produced in the second stage of
the algorithm, and check whether they have clique number equal to chromatic
number (noting that we now know the clique number). If this never happens,
then $G$ is synchronizing; otherwise, not.
\end{enumerate}

\begin{example}
Consider the symmetric group $S_m$ acting on $2$-sets. We have just one
complementary pair of graphs to consider: the line graph of $K_m$ (which
has clique number $m-1$) and its complement (which has clique number
$\lfloor m/2\rfloor$). We see immediately that this group is separating if
(and only if) $m$ is odd. In the case $m$ even, we have to work out the
chromatic number of these two graphs, as we did earlier, and we find that this
group is not synchronizing.

So, for these groups, the properties ``synchronizing'' and ``separating''
are equivalent.
\end{example}

\subsection{Partition separation}

We have seen that separation is a strengthening of synchronization. There is a
dual notion, which is a weakening of synchronization. We say that a permutation
group $G$ on $\Omega$ is \emph{not partition-separating} if there are two
non-trivial partitions $P$ and $Q$ of $\Omega$ such that, for any part $A$ of
$P$ and any part $B$ of $Q$, and any $g\in G$, $Ag\cap B\ne\emptyset$. It is
\emph{partition-separating} if no such pair of partitions exists.

Note that, if $G$ is not partition-separating, then each of the two
witnessing partitions $P$ and $Q$ is section-regular for $G$, with any
part of the other partition as a section.

By very similar arguments to those we have seen, we obtain the following
characterization:

\begin{theorem}
The transitive permutation group $G$ on $\Omega$ is not partition-separating if
and only if there exists a non-trivial graph $\Gamma$ on the vertex set
$\Omega$ with $G\le\Aut(\Gamma)$ and
$\chi(\Gamma)\cdot\chi(\overline{\Gamma})=|\Omega|$.
\end{theorem}

A partition-separating group $G$ is primitive. For if $G$ is imprimitive,
the above theorem applies to the complete multipartite graph whose parts
are those of the non-trivial partition fixed by $G$.
Furthermore, we have:

\begin{prop}
A partition separating group is basic.
\end{prop}

\begin{proof}
Consider the Cartesian structure with automorphism group $S_k\wr S_m$, and
identify the domain of $S_k$ with the group of integers mod~$k$. Now consider
the following two partitions:
\begin{itemize}\itemsep0pt
\item The parts of the first partition consist of all $m$-tuples where the
values of all coordinates except the last are constant: there are $k^{m-1}$
parts of size~$k$.
\item The parts of the second partition are the $m$-tuples with fixed sum:
there are $k$ parts of size $k^{m-1}$.
\end{itemize}
It is straightforward to show that any part of one is a section for the other.
\qed
\end{proof}

\begin{example}
We have seen that the symmetric group $S_m$ acting on $2$-sets is not
synchronizing if $m$ is even and greater than $4$. However, this group is
partition-separating. For the line graph of $K_n$ has clique number equal
to chromatic number in this case, but its complement does not.
\end{example}

\subsection{Multisets}

In order to describe the next class of permutation groups, the \emph{spreading}
groups, we need to introduce some notation for multisets.

A \emph{multiset} of $\Omega$ is a function from $\Omega$ to the natural
numbers (including zero). If $A$ is a multiset, we call $A(i)$ the
\emph{multiplicity} of $i$ in $A$. The set of elements of $\Omega$ with
non-zero multiplicity is the \emph{support} of $A$. We can regard a set
as a special multiset in which all multiplicities are zero and one
(identifying the set with its characteristic function).

The \emph{cardinality} of $A$ is
\[|A| = \sum_{i\in\Omega}A(i);\]
this agrees with the usual definition in the case of a set.

The \emph{product} of two multisets $A$ and $B$ is the multiset
$A*B$ defined by
\[(A*B)(i) = A(i)B(i).\]
This is a generalization of the usual definition of intersection of sets;
but the ``intersection'' of multisets is defined differently
in the literature.
\begin{enumerate}\itemsep0pt
\item The product of two sets is their intersection.
\item The product of a multiset $A$ and a set $B$ is the ``restriction of $A$
to $B$'', that is, points of $B$ have the same multiplicity as in $A$, while
points outside $B$ have multiplicity zero.
\item if we identify a multiset $A$ with a vector $v_A$ of non-negative
integers with coordinates indexed by $\Omega$, then we have
$|A*B|=v_A\cdot v_B$ for all multisets $A$ and $B$. In particular,
$|A|=v_A\cdot j$, where $j$ is the all-one vector.
\end{enumerate}

The image of a multiset $A$ under a permutation $g$ is defined by
\[Ag(i) = A(ig^{-1}).\]
This agrees with the usual image of a set under a permutation.

\begin{theorem}\label{t:average.value}
Let $G$ be a transitive permutation group on $\Omega$, and let
$A$ and $B$ be multisets of $\Omega$. Then the average cardinality of the
product of $A$ and $Bg$ is given by
\[\frac{1}{|G|}\sum_{g\in G}|A*Bg| = \frac{|A|\cdot|B|}{|\Omega|}.\]
\end{theorem}

\begin{proof}
We count triples $(a,g,b)$ with $a\in A$, $g\in G$, $b\in B$, and $bg=a$.
(Points of $A$ or $B$ are counted according to their multiplicity.) There
are $|A|$ choices for $a$ and $|B|$ choices for $b$. Then the set of
elements of $G$ mapping $b$ to $a$ is a right coset of the stabilizer $G_b$
since $G$ is transitive, so there are $|G|/|\Omega|$ such elements.

On the other hand, for each element $g\in G$, if $bg=a$, then this element
belongs to the support of $A*Bg$. The number of choices of $a$ is equal to the sum of
multiplicities in $A$, and for each one, the number of choices of $b$ is
the multiplicity of $ag^{-1}$ in $B$, that is, of $a$ in $Bg$. So the
product counts the multiplicities correctly.

Equating the two sides gives the result.\qed
\end{proof}

\subsection{Spreading}

Let $G$ be a transitive permutation group on $\Omega$, and $A$ and $B$
multisets of $\Omega$. Consider the following four conditions, where $\lambda$ is a positive integer:
\begin{itemize}\itemsep0pt
\item[$(1)_\lambda$:] $|A*Bg|=\lambda$ for all $g\in G$.
\item[$(2)$:] $A$ is a set.
\item[$(3)$:] $B$ is a set.
\item[$(4)$:] $|A|$ divides $|\Omega|$.
\end{itemize}

Note that
\begin{enumerate}
\item $(1)_\lambda$ is symmetric in $A$ and $B$.
\item $(1)_\lambda$ with $\lambda=1$ implies $(2)$, $(3)$ and $(4)$. For, if
$A(i)>1$, the choosing $g$ to map a point in the support of $B$ to $i$, we
would have $|A\cap Bg|>1$; so $(2)$ holds, and $(3)$ is similar. Finally,
if $(1)_\lambda$ holds with $\lambda=1$ then
$|A|\cdot|B|=|\Omega|$ {by Theorem~\ref{t:average.value}}.
\item If $(2)$ and $(3)$ hold, then we can replace product by intersection
in $(1)_\lambda$.
\end{enumerate}

We will call a multiset \emph{trivial} if either it is constant or its
support is a singleton. (This is a slight departure from our previous
convention on non-triviality!)

The transitive permutation group $G$ on $\Omega$ is \emph{non-spreading}
if there exist non-trivial multisets $A$ and $B$ and a positive integer
$\lambda$ such that $(1)_\lambda$, $(3)$ and $(4)$ hold, and is
\emph{spreading} otherwise. {Note that if $(1)_\lambda$ holds, then
\begin{equation}\label{eq:ben.star}
\lambda = \frac{|A|\cdot |B|}{|\Omega|}
\end{equation}
by Theorem~\ref{t:average.value}.
}

\begin{theorem}
The permutation group $G$ on $\Omega$ is spreading if and only if, for any
function $t\colon \Omega\to\Omega$ which is not a permutation and any non-trivial
subset $S$ of $\Omega$, there exists $g\in G$ such that
$|Sgt^{-1}|>|S|$.
\end{theorem}

\begin{proof}
Suppose that $G$ is non-spreading, and let the multiset $A$ and set $B$ be
witnesses. Since $|A|$ divides $|\Omega|$, there is a function $t$ from
$\Omega$ to $\Omega$ so that $|at^{-1}|$ is proportional to the multiplicity
of $a$ in $A$ (the constant of proportionality being $|\Omega|/|A|$). Let
$S=B$. Then for any $g\in G$, we have
\[|Sgt^{-1}| = |A*Sg|\cdot|\Omega|/|A| = |S|,\]
by the definition of non-spreading {and equation (\ref{eq:ben.star})}.

Conversely, suppose that there is a function $t$ and subset $S$ for which
the condition in the theorem is false. Let $A$ be the multiset in which the
multiplicity of $a$ is equal to $|at^{-1}|$. Then we have $|A|=|\Omega|$
and it is false that $|A*Sg|>|S|$ for any $g\in G$; thus we have
$|A*Sg|=|S|$ for all $g\in G$ (since the average value of $|A*Sg|$ is $|S|$
by Theorem~\ref{t:average.value}). We conclude that $(1)_{|S|}$, $(3)$ and
$(4)$ hold, so that $G$ is non-spreading.\qed
\end{proof}

\begin{theorem}
\begin{enumerate}\itemsep0pt
\item A spreading permutation group is separating.
\item A $2$-{homogeneous} group is spreading.
\end{enumerate}
\end{theorem}

\begin{proof}
(a) Witnesses to non-separation are also witnesses to non-spreading (with
$\lambda=1$).

(b) The arguments are similar to those we have seen before.\qed
\end{proof}

We will see that neither implication reverses.
In fact, Pin proved that transitive groups of prime degree are spreading~\cite{Pincerny}.  We shall obtain this
 as a special case of a stronger result later.

\begin{example}
We saw that the automorphism group of the Petersen graph is synchronizing.
However, this group is not spreading. Take $A$ to be the outer pentagon,
and $B$ an independent set of size~$4$: then $|Ag\cap B|=2$ for any
automorphism $g$.

More generally,
we have seen that $S_n$, acting on the set of $2$-subsets of $\{1,\ldots,n\}$,
is separating if $n$ is odd and $n\ge5$. We now show that it is not spreading.

Let $A$ be a set of $n$ pairs forming an $n$-cycle: $A=\{\{1,2\},\{2,3\},\ldots,
\{n-1,n\},\{n,1\}\}$.
Let $B$ be the set of $n-1$ pairs containing the fixed element $1$. Then
\begin{enumerate}\itemsep0pt
\item $|Ag\cap B|=2$ for all $g\in G$;
\item $A$ and $B$ are sets;
\item $|A|=n$ divides $|\Omega|=n(n-1)/2$ if $n$ is odd.
\end{enumerate}
\end{example}

\subsection{Spreading groups and the \v{C}ern\'y conjecture}

\begin{theorem}\label{t:spreading}
Let $G$ be a spreading permutation group on $\Omega$, and $f$ a function from
$\Omega$ to $\Omega$ which is not a permutation. Then $\langle G,f\rangle$
contains a rank~$1$ mapping that can be expressed as a product 
which has at most $n-1$ occurrences of $f$, where $n=|\Omega|$.

In particular, if $A$ is a generating set for $G$ and each element
of $G$ can be written as a word in $A$ of length at most $t$, then 
there is a reset word over $A\cup \{f\}$ of length at most $1+(t+1)(n-2)$.
\end{theorem}

In other words, the property of being spreading not only implies
synchronization, but also realizes the first part of our programme for
bounding the length of the reset word.

\begin{proof}
Suppose that we have a set $S_k$ with $|S_k|\ge k$, such that there is a word
$w$ in $\langle G,f\rangle$ with at most $k-1$ occurrences of $f$ which maps
$S_k$ to a singleton.

By the preceding theorem, there exists $g\in G$ such that
$S_{k+1}=S_kgf^{-1}$ satisfies $|S_{k+1}|\ge k+1$. We have $S_k=S_{k+1}fg^{-1}$,
so the word $fg^{-1}w$ with at most $k$ occurrences of $f$ maps $S_{k+1}$
to a singleton.

By induction on $k$, the result is proved.

The final statement follows because there is a rank $1$ mapping of the form
$fg_1fg_2\cdots fg_{n-2}f$ and each mapping $fg_i$ can be represented
by a word of length at most $t+1$.  \qed
\end{proof}

\subsection{Measuring non-synchronization}

Given a permutation group $G$ on $\Omega$, we want to give a quantitative
measure of how far (if at all) $G$ is from being synchronizing.

There are several methods for doing this. For example, we could consider the
largest and smallest rank of a map not synchronized by $G$.
Theorem~\ref{t:rystsov}
shows that $G$ is primitive if and only if it synchronizes every map of rank
$n-1$ (where $n=|\Omega|$); we also noted that it is shown in~\cite{abcrs}
that primitive groups synchronize maps of rank at least $n-4$. In the other
direction, we have the following:

\begin{theorem}
A primitive group synchronizes every map of rank~$2$.
\end{theorem}

\begin{proof}
Suppose that $G$ fails to synchronize a map $f$ of rank~$2$. Then there is a
graph $\Gamma$ with clique number and chromatic number $2$ (that is, a
non-trivial bipartite graph) with $G\le\Aut(\Gamma)$. If $\Gamma$ is
disconnected, then the partition into connected components is preserved by $G$;
if it is connected, then it has a unique bipartition, which is preserved by $G$.
\qed
\end{proof}

The example of the $3\times3$ grid (Figure~\ref{f:grid}) shows that this
result does not extend to maps of rank~$3$.

More generally, given a group $G$, we define the set
\[\NS(G)=\{r:\hbox{there exists a map of rank $r$ not synchronized by $G$}\}\]
of \emph{non-synchronizing ranks} for $G$.

\begin{theorem}
If $G$ is transitive but imprimitive, of degree $n$, then
\[|\NS(G)|\ge(\textstyle{\frac{3}{4}}+o(1))n.\]
\end{theorem}

\begin{proof}
Suppose that $G$ has $m$ blocks of imprimitivity, each of size $k$. Then,
among the non-trivial $G$-invariant graphs, we find:
\begin{enumerate}\itemsep0pt
\item The disjoint union of $m$ complete graphs of size $k$. This graph
can be mapped onto any non-empty subset of its components; so
\[\{k,2k,\ldots,(m-1)k\}\subseteq\NS(G).\]
\item The complete multipartite graph with $m$ parts of size $k$. Each part
can be collapsed onto any non-empty subset of itself; so
\[\{m, m+1, m+2,\ldots, mk-2, mk-1\}\subseteq\NS(G).\]
\end{enumerate}
It is easy to see that the union of these two subsets has size
$(\frac{3}{4}-o(1))n$.\qed
\end{proof}

\begin{conj}
If $G$ is primitive of degree $n$, then $|\NS(G)|=o(n)$.
\end{conj}

If true, this would show that, as far as synchronization is concerned, there
is a big divide between primitive and imprimitive groups, with primitive
groups being close to synchronizing, and imprimitive groups more distant.
The most extreme primitive groups known (the examples of primitive groups
which are not almost synchronizing) have $|\NS(G)|=O(\sqrt{n})$
(\cite{abcrs}).

\clearpage

\section{Examples}

In this section, we treat some general classes of examples. These will
yield examples of groups which are synchronizing but not separating.
We will see that
\begin{enumerate}\itemsep0pt
\item the techniques are combinatorial and geometric rather than
group-theoretic;
\item we reach very hard problems very quickly.
\end{enumerate}

\subsection{The symmetric group on subsets}

Let $G=S_n$, and let $\Omega$ be the set of all $k$-subsets of
$\{1,\ldots,n\}$.

We may assume that $n\ge2k$, since the actions of $S_n$ on $k$-sets and on
$(n-k)$-sets are isomorphic.

In fact we may assume that $n\ge2k+1$, since the action of $S_n$ on $k$-sets
is imprimitive if $n=2k$: the relation ``equal or disjoint'' is a congruence.

Now $G$ has $k$ orbits on the $2$-element subsets of $\Omega$, namely,
\[O_l=\{\{S_1,S_2\} : |S_1\cap S_2|=l\}\]
for $l=0,1,\ldots,k-1$. These $k$ graphs together with the relation of equality
form a combinatorial structure known as an \emph{association scheme},
specifically the \emph{Johnson scheme} $J(n,k)$. (Association schemes will
be discussed further in Section~\ref{s:as}.)

All these graphs are connected (this is an exercise), so $G$ is primitive
on $\Omega$. Since its socle is simple, it is basic.

If $k=1$, then $G$ is $2$-transitive. We ignore this case. Also, we dealt
with the case $k=2$ earlier. So we assume that $k\ge3$.

\subsubsection{Baranyai's Theorem}

Let $\mathcal{F}$ be a set of $k$-subsets of $\{1,\ldots,n\}$, where $k$
divides $n$. A \emph{$1$-factorization} of $\mathcal{F}$ is a partition of
$\mathcal{F}$ such that each part is a partition of $\{1,\ldots,n\}$ (that
is, a set of $n/k$ pairwise disjoint subsets).

\begin{theorem}
If $k$ divides $n$, then there is a $1$-factorization of the set of all
$k$-subsets of $\{1,\ldots,n\}$.
\end{theorem}

The theorem was proved by Baranyai in 1973 (\cite{baranyai}). The proof is a beautiful
application of the Max-Cut Min-Flow Theorem for networks.

As a corollary we have:

\begin{theorem}
If $k$ divides $n$, then $S_n$ acting on $k$-sets is not synchronizing.
\end{theorem}

For the set of all $k$-sets containing a fixed element (say $1$) is a
section of the Baranyai partition, which is thus section-regular.

\subsubsection{The case $k=3$}

We now consider the case $k=3$, and resolve completely the question of
synchronization and separation. We will see that further combinatorial tools
are required.

\begin{theorem}
Let $G=S_n$ acting on the set of $3$-subsets of $\{1,\ldots,n\}$, with $n\ge7$.
Then the following are equivalent:
\begin{enumerate}\itemsep0pt
\item $G$ is synchronizing;
\item $G$ is separating;
\item $n$ is congruent to $2$, $4$ or $5$ (mod~$6$), and $n\ne8$.
\end{enumerate}
\end{theorem}

Note that synchronization and separation are equivalent for this class of
groups.

A \emph{Steiner triple system} is a collection $\mathcal{S}$ of $3$-subsets of
$\{1,\ldots,n\}$ with the property that every pair of points of
$\{1,\ldots,n\}$ is contained in a unique member of $\mathcal{S}$.

Kirkman proved in 1847 that a Steiner triple system on $n$ points exists if
and only if $n$ is congruent to $1$ or $3$ mod~$6$.

A \emph{large set} of Steiner triple systems is a partition of the set of
all $3$-subsets of $\{1,\ldots,n\}$ into Steiner triple systems. (Counting
shows that there must be $n-2$ such systems.)

For $n=7$, there is a unique Steiner triple system, the \emph{Fano plane}
(see Figure~\ref{f:fano}.)

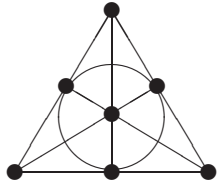
\begin{figure}[htbp]
\begin{center}
\setlength{\unitlength}{0.0505mm}
\begin{picture}(510,425)
\put(0,0){\circle*{40}}
\put(255,0){\circle*{40}}
\put(510,0){\circle*{40}}
\put(135,225){\circle*{40}}
\put(375,225){\circle*{40}}
\put(255,153){\circle*{40}}
\put(255,425){\circle*{40}}
\put(0,0){\line(1,0){510}}
\put(0,0){\line(5,3){375}}
\put(0,0){\line(3,5){255}}
\put(510,0){\line(-5,3){375}}
\put(510,0){\line(-3,5){255}}
\put(255,0){\line(0,1){425}}
\put(255,144.5){\circle{289}}
\end{picture}
\end{center}
\caption{\label{f:fano}The Fano plane}
\end{figure}

We cannot find more than two disjoint copies of the Fano plane. This fact goes
back to Cayley. However, Teirlinck~\cite{teirlinck} showed:

\begin{theorem}
If $n$ is congruent to $1$ or $3$ (mod~$6$) and $n>7$, then there exists a
large set of Steiner triple systems on $n$ points.
\end{theorem}

Now let $G$ be $S_n$ acting on $3$-sets, for $n\ge7$.

Baranyai's theorem shows that $G$ is non-synchronizing if $n$ is divisible
by $3$, that is, if $n$ is congruent to $0$ or $3$ (mod~$6$).

Teirlinck's theorem shows that $G$ is non-synchronizing if $n$ is congruent
to $1$ or $3$ (mod~$6$) and $n\ne7$. (The set of triples through two given
points is a section for all images of the large set.)

The cases $n=7$ and $n=8$ require special treatment.

\paragraph{The case $n=7$}

For each line $L$ of the Fano plane, let $S(L)$ be the set of $3$-sets
equal to or disjoint from $L$. Then $|S(L)|=5$.

Since no two lines of the Fano plane are disjoint, and no $3$-set is
disjoint from more than one line, we see that the sets $S(L)$ are pairwise
disjoint. Since $5\cdot7=35={7\choose3}$, they form a partition of $\Omega$.

Now $3$-sets in the same $S(L)$ meet in $0$ or $2$ points. So any image of
the Fano plane meets each $S(L)$ in at most (and hence exactly) one set.
Thus the partition is section-regular, the Fano plane being the section.

So $S_7$ acting on $3$-sets is not synchronizing.

\paragraph{The case $n=8$}

Take a Fano
plane on $\{1,\ldots,7\}$. For each line $L$ of the Fano plane, partition
the eight points into $L\cup\{8\}$ and the rest, and take the set $T(L)$ of
eight triples contained in a part of this partition. This gives a partition
of all the ${8\choose 3}=56=7\cdot8$ $3$-sets into seven subsets of size $8$.

Once again we find that this partition is section-regular, with the Fano plane
as a section.

\subsubsection{The separating cases}

We have now shown that, in the cases not stated in the theorem, $G$ is
non-synchronizing and hence non-separating. We have to show that, in the
remaining cases, $G$ is separating, and hence synchronizing.

There are $2^3-2$ graphs to consider. We denote them by $\Gamma_I$, for
$\emptyset\subset I\subset\{0,1,2\}$; the vertices are the $3$-sets, and
two vertices are adjacent if and only if the cardinality of their intersection
belongs to $I$.

We have to find the clique number of each of these graphs, and check whether
$\omega(\Gamma_I)\omega(\Gamma_{I^*})={n\choose 3}$, where $I^*=\{0,1,2\}\setminus I$.

The following theorem, the \emph{Erd\H{o}s--Ko--Rado theorem},
finds the clique number of some of these graphs.
A family $\mathcal{F}$ of $k$-subsets of $\{1,\ldots,n\}$ is
\emph{$t$-intersecting} if $|A\cap B|\ge t$ for all $A,B\in\mathcal{F}$.

\begin{theorem}
For $n\ge n_0(k,t)$, the maximum size of a $t$-intersecting family of $k$-sets
of $\{1,\ldots,n\}$ is $n-t\choose k-t$, with equality realized only by the
family of all $k$-sets containing a fixed $t$-set.
\end{theorem}

The correct value of $n_0(k,t)$ is known. We need only that the assertion of
the theorem is true for $k=3$, $n\ge7$, and $t=1$ or $t=2$.

\paragraph{The cases $I=\{0\}$ and $I=\{1,2\}$}

Clearly $\omega(\Gamma_{\{0\}})=\lfloor n/3\rfloor$.

By Erd\H{o}s--Ko--Rado,
$\omega(\Gamma_{\{1,2\}})={n-1\choose 2}$. The product of these numbers is
$n\choose 3$ if and only if $n$ is a multiple of $3$; but this case is excluded.

\paragraph{The cases $I=\{0,1\}$ and $I=\{2\}$}

By Erd\H{o}s--Ko--Rado, $\omega(\Gamma_{\{2\}})=n-2$.

A clique in $\Gamma_{\{0,1\}}$ has the property that two points lie in at most
one set in the clique; so $\omega(\Gamma_{\{0,1\}})\le n(n-1)/6$, with equality
if and only if there is a Steiner triple system of order $n$, that is, $n$
is congruent to $1$ or $3$ (mod~$6$). But these cases are excluded.

\paragraph{The cases $I=\{1\}$ and $I=\{0,2\}$}

It is easy to show that a maximum clique in $\Gamma_{\{0,2\}}$ is obtained by
dividing most of $\{1,\ldots,n\}$ into disjoint $4$-sets and taking all
the $3$-subsets of these $4$-sets. In particular, $\omega(\Gamma_{\{0,2\}})\le n$.

A maximum clique in $\Gamma_{\{1\}}$ is obtained by taking $\lfloor n/2\rfloor$
triples through a fixed point but having no further point in common, provided
that $n\ge17$. For smaller values, a Fano plane may be better.

A little calculation shows that the product of these bounds is strictly smaller
than $n\choose 3$ except for $n=7$ and $n=8$; but these cases are excluded.

\subsubsection{Spreading}

The recent remarkable result of
Keevash~\cite{keevash} on the existence of Steiner systems shows, as above,
the existence of infinitely many more values of $n$ and $k$ for which the
symmetric group $S_n$ acting on $k$-sets is non-separating.

However, for spreading, things are much easier. The following argument is
due to Peter Neumann.

\begin{theorem}
The symmetric group $S_n$ acting on $k$-sets is always non-spreading.
\end{theorem}

\begin{proof}
Let $d$ be the greatest common divisor of $n$ and $k$. Let $H$ be a cyclic
group of order~$n$ permuting the elements of $\{1,\ldots,n\}$ in the natural
way. Now choose a $k$-subset of $\{1,\ldots,n\}$ which is a union of $k/d$
orbits of the subgroup of order $d$ of $H$, and let $A$ be the $H$-orbit
(in $\Omega$) containing this set; so $|A|=n/d$. Let $B$ consist of all
$k$-sets containing the element $1$. Since $A$ is invariant under a transitive
group, $|A\cap Bg|$ is constant for $g\in G$. Also, clearly $A$ and $B$ are
sets.

It remains only to show that $|A|=n/d$ divides $|\Omega|={n\choose k}$. The
stabilizer in $H$ of any $k$-set has order dividing $k$ and also dividing $n$,
hence dividing $d$; so the size of any $H$-orbit in $\Omega$ is a multiple
of $n/d$. The assertion follows.\qed
\end{proof}

\subsubsection{Linear groups acting on subspaces}

The action of $\mathrm{PGL}(n,q)$ on the set of $k$-dimensional subspaces of
the $n$-dimensional vector space gives a linear analogue of the action of
$S_n$ on $k$-subsets of $\{1,\ldots,n\}$. But much less is known in this
case, since the linear analogues of combinatorial results such as those of
Baranyai and Erd\H{o}s--Ko--Rado are not known except in special cases
(for example,~\cite{beutelspacher}). Even the existence of the analogues of
Steiner systems is a major unsolved problem; the first examples have been
given very recently~\cite{beovw}.

\subsection{Classical groups and polar spaces}

Now we turn to the other family of examples discussed here: classical
(symplectic, unitary and orthogonal) groups acting on the associated polar
spaces.

We give a brief introduction to these groups and geometries; more detail
is available in several places, including~\cite{c:pps,taylor}.

We are only interested in finite classical groups; this makes the theory
simpler in several respects.

\subsubsection{Finite classical groups}

A classical group acts on a vector space and preserves a form of some type:
\begin{enumerate}\itemsep0pt
\item for \emph{symplectic groups}, an alternating bilinear form;
\item for \emph{unitary groups}, a Hermitian sesquilinear form;
\item for \emph{orthogonal groups}, a quadratic form, and the symmetric
bilinear form obtained from it by \emph{polarization}.
\end{enumerate}

The basic form should be \emph{non-degenerate} or \emph{non-singular}.
The reason for separating cases is that strange things happen with quadratic
forms in characteristic~$2$. But we can ignore this complication!

There are three parameters associated with a classical group:
\begin{itemize}\itemsep0pt
\item[] $q$, the order of the field over which the matrices are defined;
\item[] $r$, the \emph{Witt index}, the dimension of the largest subspace
on which the form vanishes identically;
\item[] $\epsilon$, a parameter defined shortly.
\end{itemize}
We denote the dimension of the underlying vector space by $n$.

We divide the classical groups into six families:
\begin{itemize}\itemsep0pt
\item[] symplectic: $\PSp(2r,q)$, $n=2r$
\item[] unitary: $\PSU(2r,q_0)$, $n=2r$, and $PSU(2r+1,q_0)$, $n=2r+1$;
\item[] orthogonal: $\POm^+(2r,q)$, $n=2r$; $\POm(2r+1,q)$, $n=2r+1$; and
$\POm^-(2r+2,1)$, $n=2r+2$.
\end{itemize}

Note that for the unitary groups, the field order must be a square, say
$q=q_0^2$, and there is a field automorphism $x\mapsto x^{q_0}$ of order~$2$.
We use the group-theorists' notation $\PSU(n,q_0)$, but the field of
definition is $\mathbb{F}_q$.

We need not consider orthogonal groups of odd dimension over fields of
characteristic~$2$, since they turn out to be isomorphic to symplectic groups
of one dimension less.

The values of the parameter $\epsilon$ are given in the table:
\[\begin{array}{|c|c|}
\hline
\hbox{Type} & \epsilon \\\hline
\PSp(2r,q) & 0 \\
\PSU(2r,q_0) & -\frac{1}{2} \\
\PSU(2r+1,q_0) & \frac{1}{2} \\
\POm^+(2r,q) & -1 \\
\POm(2r+1,q) & 0 \\
\POm^-(2r+2,q) & 1 \\
\hline
\end{array}\]

\subsubsection{Polar spaces}

The \emph{polar space} associated with a classical group acting on a vector
space $V$ is the geometry of \emph{totally isotropic} subspaces of $V$, those
on which the form vanishes identically. We abbreviate this to \emph{t.i.}

In the case of orthogonal groups, we should really use the term
\emph{totally singular} or \emph{t.s.} instead; but we will ignore this
distinction.

Subspaces of (vector space) dimension $1$ or $2$ are called \emph{points}
and \emph{lines}, as usual in projective geometry. Subspaces of maximum
dimension $r$ are called \emph{maximal subspaces}.

\subsubsection{Numerical information}

Numerical information about polar spaces can be expressed in terms of the
parameters $q,r,\epsilon$:

\begin{theorem}
\begin{enumerate}
\item The number of points of the polar space is
$(q^r-1)(q^{r+\epsilon}+1)/(q-1)$; each maximal subspace contains
$(q^r-1)/(q-1)$ points.
\item The number of points not collinear with a given point is
$q^{2r+\epsilon-1}$.
\item The number of maximal subspaces is
\[\prod_{i=1}^r (1+q^{i+\epsilon}).\]
\end{enumerate}
\end{theorem}

\subsubsection{Witt's Lemma}

\emph{Witt's Lemma} asserts that the action of the classical group on a
polar space is ``homogeneous'', in the sense that any linear isometry between
subspaces of the vector space is induced by an element of the group.

In particular, the group acts transitively on points, on collinear pairs of
points, and on non-collinear pairs of points.

So the \emph{graph} of the polar space (whose vertices are the points, two
vertices joined if they are collinear) is a rank~$3$ graph.

In the case $r=1$, there are no lines, so the graph of the polar space is
null; Witt's lemma implies that the action of the group is $2$-transitive.
We will ignore this case.

\subsubsection{An example}

The polar space of type $\POm^+(4,q)$ is the familiar \emph{ruled quadric},
see Figure~\ref{f:quadric}.

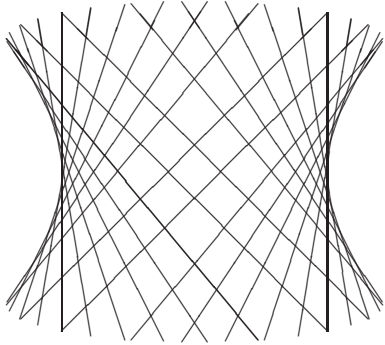
\begin{figure}[htbp]
\begin{center}
\setlength{\unitlength}{0.025mm}
\begin{picture}(1200,1800)(-600,-200)
\put(707,-141){\line(0,1){1697}}
\put(-707,-141){\line(0,1){1697}}
\put(-834,-110){\line(1,6){282}}
\put(834,-110){\line(-1,6){282}}
\put(-552,-167){\line(-1,6){282}}
\put(552,-167){\line(1,6){282}}
\put(-929,-74){\line(1,3){558}}
\put(929,-74){\line(-1,3){558}}
\put(-371,-186){\line(-1,3){558}}
\put(371,-186){\line(1,3){558}}
\put(-986,-33){\line(1,2){822}}
\put(986,-33){\line(-1,2){822}}
\put(-164,-197){\line(-1,2){822}}
\put(164,-197){\line(1,2){822}}
\put(69,-200){\line(-2,3){1067}}
\put(-69,-200){\line(2,3){1067}}
\put(-998,14){\line(2,3){1067}}
\put(998,14){\line(-2,3){1067}}
\put(1000,0){\line(-5,6){1338}}
\put(-1000,0){\line(5,6){1338}}
\put(338,-188){\line(-5,6){1338}}
\put(338,-188){\line(-5,6){1338}}
\put(-338,-188){\line(5,6){1338}}
\put(924,-76){\line(-1,1){1631}}
\put(-924,-76){\line(1,1){1631}}
\put(707,-141){\line(-1,1){1631}}
\put(-707,-141){\line(1,1){1631}}
\end{picture}
\end{center}
\caption{\label{f:quadric}A ruled quadric}
\end{figure}

Combinatorially this structure is just a grid, so the classical group is
non-basic. We will also ignore this case.

\subsubsection{Cliques and cocliques}

We must now look at cliques and cocliques in the graph $\Gamma$ of a polar space.

A clique is a set of $1$-dimensional subspaces on which the form vanishes and
which are pairwise orthogonal; so its span is also a clique. Thus the cliques
of maximal size are just the maximal subspaces, of size $(q^r-1)/(q-1)$.

Hence a coclique contains at most $q^{r+\epsilon}+1$ points, with equality
if and only if it meets every maximal in exactly one point.

A coclique meeting this bound is called an \emph{ovoid}.

We need one further definition: a \emph{spread} is a family of maximal
subspaces which partitions the set of points.

\begin{theorem}
\begin{enumerate}\itemsep0pt
\item A classical group is non-separating if and only if its polar space
possesses an ovoid.
\item A classical group is non-synchronizing if and only if its polar
space possesses either
  \begin{enumerate}\itemsep0pt
  \item an ovoid and a spread; or
  \item a partition into ovoids.
  \end{enumerate}
\item A classical group is not partition-separating if and only if its polar
space possesses both a spread and a partition into ovoids.
\end{enumerate}
\end{theorem}

\subsubsection{Ovoids, spreads and partitions}

You might expect at this point to be told that the question of which polar
spaces contain ovoids, spreads, or partitions into ovoids has been completely
solved by finite geometers.

Unfortunately, despite a lot of effort, this is not the case.

We summarize some of the results which have been obtained. See
\cite{hirschfeld_thas,thas} for details.

\paragraph{Ovoids} $\phantom{.}$

\begin{center}
\begin{tabular}{|r|p{3in}|}
\hline
$\PSp(2r,q)$ & Yes for $r=2$ and $q$ even; no in all other cases \\\hline
$\PSU(2r,q_0)$ & Yes for $r=2$ \\\hline
$\PSU(2r+1,q_0)$ & No \\\hline
$\POm^+(2r,q)$ & Yes for $r=2,3$; yes for $r=4$ and $q$ even, or $q$ prime, or
$q\equiv3$ or $5$ mod $6$; no for $r\ge4$ and $q=2$ or $q=3$ \\\hline
$\POm(2r+1,q)$ & Yes for $r=2$; yes for $r=3$ and $q=3^h$ \\\hline
$\POm^-(2r+2,q)$ & No \\\hline
\end{tabular}
\end{center}

\paragraph{Spreads} $\phantom{.}$

\begin{center}
\begin{tabular}{|r|p{3in}|}
\hline
$\PSp(2r,q)$ & Yes \\\hline
$\PSU(2r,q_0)$ & No \\\hline
$\PSU(2r+1,q_0)$ & No for $r=2$, $q_0=2$ \\\hline
$\POm^+(2r,q)$ & No if $r$ is odd; yes if $r=2$, or $r=4$ with $q$ prime or
$q\equiv3$ or $5$ mod $6$; yes if $r$ and $q$ are even \\\hline
$\POm(2r+1,q)$ & No if $r$ is even (and $q$ odd); yes if $r=3$ with $q$ prime
or $q\equiv3$ or $5$ mod $6$ \\\hline
$\POm^-(2r+2,q)$ & Yes if $r=2$, or if $q$ is even \\\hline
\end{tabular}
\end{center}

\subsubsection{Some conclusions}
We conclude that
$\PSp(2r,q)$, $\PSU(2r+1,q_0)$, and $\POm^-(2r+2,q)$ are separating for all
$r\ge2$, except for $\PSp(4,q)$ with $q$ even.
Cases where the group is not separating can also be read off from the first
table. However, less is known about partitions into ovoids, so results
about synchronization and partition separation are less clear.

\begin{example}
The polar space of the group $\POm(5,q)$, for $q$ odd, possesses ovoids but
no spreads. For $q=3,5,7$, these ovoids are all \emph{classical}; that is,
they consist of the set of points lying in a non-singular $4$-dimensional
space of type $\POm^-(4,q)$ (this polar space has Witt index~$1$, so contains
no lines). Any two such spaces meet in a $3$-dimensional space, so two
such ovoids meet in a \emph{conic}. In particular, there are no partitions
into ovoids.

So the group $\POm(4,q)$, for $q=3,5,7$, is synchronizing but not separating.
These are our first examples of such groups, and show that the implication from
separating to synchronizing does not reverse.
\end{example}

\subsubsection{Spreading}

We give a necessary condition for a classical group to be non-spreading, which
applies to three of the six types.

\begin{theorem}
Let $G$ be a classical group of Witt index at least~$2$, acting on the
points of its polar space. Suppose that there exists a non-degenerate
hyperplane which has Witt index smaller than that of the whole space.
Then $G$ is non-spreading.
\label{t:classical-ns}
\end{theorem}

\begin{proof}
We take $A$
to be a maximal subspace, and $B$ to be the set of points lying in the
assumed hyperplane. Then $|A\cap B^g|=(q^{r-1}-1)/(q-1)$ for all $g\in G$,
and $A$ and $B$ are both sets with $|A|$ dividing $|\Omega|$.\qed
\end{proof}

This theorem covers the classical groups $\PSU(2r,q_0)$, $\POm^+(2r,q)$, and
$\POm(2r+1,q)$, but not $\PSp(2r,q)$, $\PSU(2r+1,q_0)$, or $\POm^-(2r+2,q)$.

\subsection{$S_{2m}$ on $(m,m)$ partitions}

As noted earlier, it seems that testing any class of primitive groups for
synchronization will produce difficult combinatorial problems. We are going
to prove one more result in this section, concerning the (primitive) action
of the symmetric group of even degree $2m$ on partitions of the domain into
two sets of size $m$. This is partly because the fact that this group is
non-spreading would follow from the truth of the \emph{Hadamard conjecture},
and also because of an unexpected appearance of the \emph{Catalan numbers}
in the proof. The Catalan numbers $(C_n)$ form one of the most ubiquitous
integer sequences in all mathematics~\cite{stanley}, but we only need two
simple properties of them:
\begin{itemize}\itemsep0pt
\item the formula: $\displaystyle{C_n=\frac{1}{n+1}{2n\choose n}}$;
\item the recurrence relation: $\displaystyle{C_m=\sum_{i=1}^{n-1}C_iC_{n-i}}$
for $i>1$.
\end{itemize}

The Catalan numbers arise in a technical result we need. Note that the number
of $(m,m)$ partitions of a $2m$-set is $\frac{1}{2}{2m\choose m}$.

\begin{lemma}
For any positive integer $m$,
\begin{enumerate}\itemsep0pt
\item $2m-1$ divides $\frac{1}{2}{2m\choose m}$;
\item if $m$ is odd then $2(2m-1)$ divides $\frac{1}{2}{2m\choose m}$.
\end{enumerate}
\end{lemma}

\begin{proof}
We have
\[\frac{1}{2}{2m\choose m}={2m-1\choose m-1}=(2m-1)\frac{(2m-2)!}{m!(m-1)!}
=(2m-1)C_{m-1}.\]
If $m$ is odd, then $m-1$ is even and the terms in the recurrence for $C_{m-1}$
come in equal pairs.\qed
\end{proof}

A \emph{Hadamard matrix} of order~$n$ is an $n\times n$ matrix $H$ with entries
$\pm1$ satisfying $HH^\top=nI$. These matrices are so-called because they
attain equality in Hadamard's bound for the determinant of a square matrix
$A=(a_{ij})$ with $|a_{ij}|\le1$ for all $i,j$.

The defining condition shows that any two rows of $H$ are orthogonal. But it
follows that $H^\top H=nI$, and so any two columns are orthogonal.

It is known that the order of a Hadamard matrix must be $1$, $2$ or a multiple
of $4$; the \emph{Hadamard conjecture} asserts that they exist for all such
orders. This is known to be true for $n<668$ (the last value to be resolved
was $n=428$ in 2005,~\cite{kt-r}).

\begin{theorem}
Suppose there exists a Hadamard matrix of order $n=4k$. Then
\begin{enumerate}\itemsep0pt
\item $S_{4k}$, acting on $(2k,2k)$ partitions, is non-spreading;
\item if $k$ is odd, then $S_{2k}$, acting on $(k,k)$ partitions, is
non-spreading.
\end{enumerate}
\end{theorem}

\begin{proof}
Let $H$ be a Hadamard matrix of order $2k$.

(a) We can normalize by changing signs of columns so that the first row of
$H$ consists entirely of $+1$ entries. Then any further row has $2k$ $+1$s
and $2k$ $-1$s, and so defines a $(2k,2k)$ partition. Let $A$ be the set of
these partitions. Note that $|A|=4k-1$. Let $B$ be the set of all $(2k,2k)$
partitions such that the elements $1$ and $2$ belong to the same part.
Since the any columns of $H$ are orthogonal, $|Ag\cap B|=2k-1$ for any
permutation $g$ of the columns. Finally, the lemma shows that $|A|$ divides
the number of partitions. So $S_{4k}$ on $(2k,2k)$ partitions is non-spreading.

(b) It is a well-known fact about Hadamard matrices that any three rows of
a Hadamard matrix of order $4k$ agree in $k$ positions. (This can be found
in the final part of~\cite{wsw}.) Normalize the first row as above, and
consider the set of $2k$ positions where the second row has entries $+1$;
then any further row has $+1$s in $k$ of these positions and $-1$ in $k$
positions. This gives us a set $A$ of $2(k-1)$ partitions of a $2k$-set of type
$(k,k)$. Exactly as above, with $B$ the set of all $(k,k)$ partitions where
$1$ and $2$ lie in the same part, we find that $|Ag\cap B|=2k-2$ for any
permutation $g$. The second part of the lemma shows that $|A|$ divides the
number of partitions if $k$ is odd.\qed
\end{proof}

\begin{cor}
If the Hadamard conjecture is true, then $S_{2m}$ acting on the set of all
$(m,m)$ partitions is non-spreading for all $m>1$.
\end{cor}

\subsection{Factorizations of simple groups}

For a final example, we turn to the simplest diagonal primitive groups, those
of the form $S\times S$, where $S$ is a simple group, acting on $S$ by the
rule
\[(g,h):x\mapsto g^{-1}xh\]
for $(g,h)\in S\times S$, $x\in S$. We cannot prove much here: the results are
mostly descriptive.

Suppose that $G$ is a group of this form. Then questions about synchronization
and separation in $G$ reduce to questions about subsets and partitions of the
simple group $S$. Consider the case when $S$ is not separating, so that there
exist subsets $A,B$ of $S$ such that $|g^{-1}Ah\cap B|=1$ for all $g,h\in S$.

Consider first the case where $A$ and $B$ are subgroups of $S$. Then
$A\cap B=1$ and $AB=S$, so we have a \emph{perfect factorization} of $S$.
Conversely, suppose that we have a perfect factorization of $S$, and let
$g=a_1b_1$ and $h=a_2b_2$, where $a_1,a_2\in A$ and $b_1,b_2\in B$. Then
\begin{eqnarray*}
|g^{-1}Ah\cap B| &=& |b_1^{-1}a_1Aa_2b_2 \cap B| \\
&=& |b_1^{-1}Ab_2 \cap B| \\
&=& |A \cap b_1Bb_2^{-1}| \\
&=& |A\cap B|\\
&=& 1,
\end{eqnarray*}
so $G$ is not separating.

Moreover, in this case, every right coset of $A$ intersects every left coset
of $B$ in a single element; so the partitions of $S$ into right cosets of $A$
and left cosets of $B$ demonstrate that $S$ is not partition-separating, and
so also not synchronizing.

In a perfect factorization of $S$, if we take the action of $S$ on the set of
right cosets of $B$, then $A$ is a regular subgroup, and \emph{vice versa}. So
finding all perfect factorizations with one factor maximal is equivalent to
finding all regular subgroups of primitive groups; this problem has been solved
by Liebeck, Praeger and Saxl (see~\cite{lps2}).

In the case where one of $A$ and $B$ is a subgroup (say $A$) and the other is
not, the condition that $A$ and $B$ witness the non-separating property of $G$
is equivalent to saying that $B$ is a \emph{loop transversal} for $A$ in $S$,
so that in the action of $S$ on the right cosets of $A$, the set $B$ is
sharply transitive: see~\cite{johnson,kuznetsov}, for example.

We do not know of any examples where neither $A$ nor $B$ is a subgroup, though
no doubt they exist.

\clearpage

\section{Representation theory}

The concept of ``spreading'' defined earlier turns out to be expressible in
terms of representation theory. In this section we outline the permutation
representation of a permutation group, and show how its properties over
different fields are related to some of the concepts we are considering.

\subsection{$2$-closure}

We have seen that synchronization and related properties are closed upwards
(i.e. preserved on passing to overgroups). They also have a limited form
of downward closure, as we will now see.

Let $G$ be a permutation group on $\Omega$.
\begin{enumerate}\itemsep0pt
\item The \emph{$2$-closure} of $G$ is the set of all permutations
of $\Omega$ which preserve the $G$-orbits on $\Omega^2$ (the set of ordered
pairs of elements of $\Omega$). The group $G$ is \emph{$2$-closed} if it is
equal to its $2$-closure.
\item The \emph{strong $2$-closure} of $G$ is the set of all permutations
of $\Omega$ which preserve the $G$-orbits on the set of $2$-element subsets
of $\Omega$. The group $G$ is \emph{strongly $2$-closed} if it is equal to its
strong $2$-closure.
\end{enumerate}

Note that
\begin{enumerate}\itemsep0pt
\item the $2$-closure of $G$ is contained (possibly strictly) in its strong
$2$-closure;
\item the $2$-closure of $G$ is the symmetric group if and only if $G$ is
$2$-transitive;
\item the strong $2$-closure of $G$ is the symmetric group if and only if
$G$ is $2$-homogeneous.
\end{enumerate}

\begin{theorem}
Let P denote one of the conditions ``primitive'', ``synchronizing'',
``separating'', ``$2$-homogeneous''. Then the following are equivalent:
\begin{enumerate}\itemsep0pt
\item $G$ satisfies P;
\item the $2$-closure of $G$ satisfies P;
\item the strong $2$-closure of $G$ satisfies P.
\end{enumerate}
\end{theorem}

\begin{proof}
In view of our earlier remarks, (a) implies (b) implies (c); so it suffices
to show that (c) implies (a). But each property can be defined in terms of
$G$-invariant graphs, and $G$ and its strong $2$-closure clearly
preserve the same graphs.\qed
\end{proof}

\subsection{Representation theory}

We now turn to an algebraic approach to these and related closure properties.
Let $\mathbb{F}$ be a field. We only consider the case $\mathbb{F}=\mathbb{C}$,
$\mathbb{R}$ or $\mathbb{Q}$. Certainly there is an interesting theory waiting
to be worked out in the case where $\mathbb{F}$ is, say, a finite field, a
$p$-adic field, or even a ring!

Let $G$ be a permutation group on $\Omega$.
The \emph{permutation module} is the $\mathbb{F}G$-module $\mathbb{F}\Omega$
which has the elements of $\Omega$ as a basis, where $G$ acts by permuting
the basis vectors.

Now the \emph{$\mathbb{F}$-closure} of $G$ consists of all permutations which
preserve all $\mathbb{F}G$-submodules of $\mathbb{F}\Omega$; and $G$ is
\emph{$\mathbb{F}$-closed} if it is equal to its $\mathbb{F}$-closure.

Consider the case where $G$ is the symmetric group $\Sym(\Omega)$. The
permutation module has just two non-trivial submodules:
\begin{enumerate}
\item the $1$-dimensional module $\underline{\Omega}$ spanned by the sum of
the elements of $\Omega$;
\item the $n-1$-dimensional \emph{augmentation submodule} consisting of the
vectors with coordinate sum zero.
\end{enumerate}

For, if $W$ is a submodule containing a vector $x$ with $x_v\ne x_w$, and $g$
is the transposition $(v,w)$, then $W$ contains $x-xg=\lambda(v-w)$. By
$2$-transitivity, $W$ contains all differences between basis vectors; but
these span the augmentation module.

\begin{theorem}\label{7.2}
The $\mathbb{C}$-closure of a permutation group $G$ is equal to its
$2$-closure.
\end{theorem}

The proof requires a little character theory; a brief sketch follows.

\subsection{Character theory}

Any representation of a group by matrices over the complex numbers is
determined up to isomorphism by its \emph{character}, the function
$\phi$ which maps each group element to the trace of the matrix representing
it. A character is a \emph{class function} (constant on conjugacy classes).

Any representation can be decomposed uniquely (up to isomorphism) into
\emph{irreducible} representations. An \emph{irreducible character} is
the character of an irreducible representation.

The irreducible characters form an orthonormal basis for the space of complex
class functions, under the inner product
\[\langle \phi,\psi\rangle=
\frac{1}{|G|}\sum_{g\in G}\phi(g)\overline{\psi(g)}.\]

The \emph{trivial character} $1_G$ is the function mapping every group element
to~$1$.

\subsection{The permutation character}

Let $G$ be a permutation group on $\Omega$, where $|\Omega|=n$. Then we
have an action of $G$ on $\mathbb{C}\Omega$ by permutation matrices. Its
character is the \emph{permutation character} $\pi$, where $\pi(g)$ is
the number of fixed points of $g$.

The \emph{Orbit-Counting Lemma} states that
\[\frac{1}{|G|}\sum_{g\in G}\pi(g) = \hbox{\# orbits of $G$}.\]
The sum on the left is just $\langle 1_G,\pi\rangle$; so the multiplicity of the
trivial character in $\pi$ is equal to the number of orbits of $G$.

Applying the preceding result to the action of $G$ on $\Omega\times\Omega$
(whose permutation character is $\pi^2$), we see that
\[\langle \pi,\pi\rangle=\langle \pi^2,1_G\rangle = \hbox{\# orbits of $G$ on
$\Omega^2$}.\]
This number is called the \emph{rank} of $G$.

The rank is equal to the
sum of squares of the multiplicities of the irreducible characters in $\pi$,
since if $\pi=\sum a_i\phi_i$, with $\phi_i$ irreducible, then orthonormality
gives
\[\langle \pi,\pi\rangle = \sum a_i^2.\]

In particular, $G$ is $2$-transitive if and only if $\pi=1_G+\phi$ for some
irreducible character $\phi$. (The character $\phi$ is afforded by the action
of $G$ on the augmentation submodule of the permutation module: so $G$ is
$2$-transitive if and only if the augmentation submodule is irreducible.)

We recall that
orbits of $G$ on $\Omega\times\Omega$ are called \emph{orbitals} of $G$.
Given an orbital $O$, there is a \emph{paired} orbital
\[O^*=\{(y,x):(x,y)\in O\}.\]
An orbital $O$ is \emph{self-paired} if $O=O^*$. We see from the above that
if $G$ has permutation character $\pi=\sum a_\phi\phi$, where $\phi$ are
irreducible characters of $G$, then the number of orbitals is $\sum a_\phi^2$.

We will discuss further the combinatorial structure of the orbitals in
Section~\ref{s:as}.

The decomposition of the permutation character also tells us about the
number of self-paired orbitals. This involves the \emph{Frobenius--Schur index}
$\epsilon_\phi$ of an irreducible character $\phi$, defined as follows:
\[\epsilon_\phi=\cases{1 & if $\phi$ is the character of a real representation
of $G$,\cr
-1 & if $\phi$ is real-valued but not the character of a real representation\cr
0 & if $\phi$ is not real-valued.\cr}\]
A character $\phi$ is called \emph{real}, \emph{quaternionic} or \emph{complex}
according as $\epsilon_\phi=+1$, $-1$ or $0$. (The term refers to the
centralizer algebra of the corresponding real representation affording the
character $\phi$, $2\phi$, or $\phi+\overline{\phi}$ respectively.)

\begin{theorem}
Let the permutation character of $G$ be
\[\pi=\sum_\phi a_\phi\phi,\]
where $\phi$ are irreducible characters of $G$. Then the number of self-paired
orbitals of $G$ is $\sum \epsilon_\phi a_\phi$.
\end{theorem}

An elementary account of this theorem appears in~\cite{c:rank}.

\subsection{$2$-closure $=$ $\mathbb{C}$-closure}
We are now ready to prove Theorem~\ref{7.2}.
Let $\oG$ be the $2$-closure of $G$. Then $\oG$ has the same sum of squares
of multiplicities of irreducibles as $G$, which implies that the decomposition
of the permutation character is the same for $\oG$ as for $G$. Hence $\oG$
is contained in the $\mathbb{C}$-closure of $G$.

Conversely, let $\hG$ be the $\mathbb{C}$-closure of $G$. Then $\hG$
preserves the isotypic components of the permutation module (one of these
consists of the sum of all copies of a particular isomorphism type of
irreducible module). The lattice of submodules of the sum of $r$ isomorphic
irreducible modules is isomorphic to the $(r-1)$-dimensional complex
projective space; all these submodules are preserved by $\hG$. So the
isomorphic $G$-modules remain isomorphic as $\hG$-modules. Thus the
multiplicities are the same for $\hG$ as for $G$, and so the ranks of these
groups are equal. Since $G\le\hG$, it follows that $\hG$ preserves the
$G$-orbits on $\Omega^2$, and so is contained in the $2$-closure $\oG$.

Hence $\hG=\oG$.

\begin{conj}
The $\mathbb{R}$-closure of a permutation group coincides with its strong
$2$-closure.
\end{conj}

This is not known in general, but it is true for groups whose permutation
character is multiplicity-free.

\subsection{\ffi{F} groups}

We say that the permutation group $G$ on $\Omega$ is \ffi{F} if
its $\mathbb{F}$-closure is the symmetric group; that is, if the only
$G$-submodules of $\mathbb{F}\Omega$ are $\underline{\Omega}$ and the
augmentation module.
Note that if $\mathbb F\subseteq \mathbb K$, then \ffi{K} implies \ffi{F} because extension of scalars commute with direct sums.

\begin{theorem}
Let $G$ be a permutation group on $\Omega$.
\begin{enumerate}
\item $G$ is \ffi{C} if and only if it is $2$-transitive.
\item $G$ is \ffi{R} if and only if it is {$2$-homogeneous}.
\end{enumerate}
\end{theorem}

\begin{proof}
(a) $G$ is \ffi{C} if its permutation character has the form $1_G+\phi$,
where $\phi$ is irreducible over $\mathbb{C}$. As noted above, this is
equivalent to the assertion that $G$ is $2$-transitive.

(b) $G$ is \ffi{R} if its permutation character has the form $1_G+\theta$,
where $\theta$ is irreducible over $\mathbb{R}$. Now there are three
possibilities for the decomposition of $\theta$ over $\mathbb{C}$:
\begin{itemize}\itemsep0pt
\item $\theta$ is irreducible over $\mathbb{C}$: then $G$ is $2$-transitive by
the preceding argument.
\item $\theta=2\phi$, where $\phi$ is irreducible over $\mathbb{C}$. Then
$\epsilon_\phi=-1$, and so the number of self-paired orbitals of $G$ is
$1-2=-1$, which is impossible.
\item $\theta=\phi+\overline{\phi}$ for some non-real character $\phi$. Then
$\epsilon_\phi=0$, and so the number of self-paired orbitals is $1$, this one
being the diagonal orbital. Thus the two non-diagonal orbitals are paired
with each other, and $G$ is $2$-homogeneous.
\end{itemize}
The argument clearly reverses.

The non-existence in the second case also follows from an old result of
Jordan (see Serre~\cite{serre}), according to which a finite transitive
permutation group of degree greater than~$1$ contains a fixed-point-free
element. Now, if $\pi=1_G+2\phi$ and $\pi(g)=0$, then $\phi(g)=-\frac{1}{2}$,
contradicting the fact that character values must be algebraic integers.\qed
\end{proof}

This naturally suggests looking at \ffi{Q} groups, to which we now turn.

\begin{theorem}
Let $G$ be a transitive permutation group on $\Omega$, and $\mathbb{F}$ a
field of characteristic zero. Then $G$ is primitive
(resp.\ synchronizing, separating, spreading, or \ffi{Q}) if and only
if its $\mathbb{F}$-closure is.
\end{theorem}

Three of these results are immediate from the next lemma.

\begin{lemma}
Let $G$ be a transitive permutation group on $\Omega$, and $\mathbb{F}$ a
field of characteristic zero. Let $A$ and $B$ be
multisets such that $|A*Bg|=\lambda$ for all $g\in G$. Then
$|A*Bg|=\lambda$ for all $g$ in the $\mathbb{F}$-closure { $\hG$} of $G$.
\end{lemma}

\begin{proof}
Let $v_1$ and $v_2$ be the characteristic functions of $A$ and $B$
respectively.
Setting $w_i=v_i-(v_i\cdot j)j/n$ for $i=1,2$,
where $j$ is the all-$1$ vector, we find that $j$, $w_1$ and $w_2g$ are
pairwise orthogonal for any $g\in G$  using that $\lambda=\frac{(v_1\cdot j)(v_2\cdot j)}{n}$ by Theorem~\ref{t:average.value}. So the $G$-submodules generated by
$j$, $w_1$ and $w_2$ are pairwise orthogonal. These modules are invariant
under $\hG$; reversing the argument gives the result.\qed
\end{proof}

This immediately proves the earlier theorem for separating, spreading and
\ffi{Q} groups.

Suppose that $G$ is imprimitive, and let $P$ be a $G$-invariant partition.
Then the characteristic functions of the parts of $P$ form an orthogonal
basis for a submodule of $\mathbb{F}\Omega$, which is {preserved} by $\hG$. The
partition can be recovered from the submodule, since it is the coarsest
partition on the parts of which the elements of the submodule are constant.
So $\hG$ is imprimitive.

Finally, suppose that $G$ is not synchronizing, and let the partition $P$
and section $S$ witness this
(that is, each $G$-translate of $S$ is a section of $P$).
Then $|A\cap Sg|=1$ for any part $A$ of $P$ and $g\in G$ and so, by the lemma, $|A\cap Sg|=1$ for all $g\in \hG$.  Thus every $\hG$-translate of $S$ is a section for $P$ and so $\hG$
is non-synchronizing.

If $G$ is \ffi{Q}, then its $\mathbb{Q}$-closure is the symmetric group,
which is spreading; so $G$ is spreading.
This was first proved in~\cite{arnold_steinberg}.

So our hierarchy finally looks like this:\label{hierarchy}
\begin{eqnarray*}
&&\hbox{$2$-transitive}\Rightarrow\hbox{$2$-{homogeneous}}\Rightarrow\hbox{\ffi{Q}}\Rightarrow\\
&&\quad\Rightarrow\hbox{spreading}\Rightarrow\hbox{separating}
  \Rightarrow\hbox{synchronizing} \Rightarrow\\
&&\Rightarrow\hbox{basic \& almost synchronizing} \quad\Rightarrow \left.\begin{array}{c}\hbox{basic}\\ \hbox{or} \\ \hbox{almost synchronizing}\end{array}\right\} \Rightarrow\\ && \Rightarrow
\hbox{primitive}\Rightarrow\hbox{transitive}.
\end{eqnarray*}

We will see that there are groups which are \ffi{Q} but not $2$-{homogeneous}; indeed, these groups have recently been classified. But no
examples are currently known of groups which are spreading but not
\ffi{Q}.

(Theorem~\ref{t:classical-ns} suggested to us that the classical
groups $\mathrm{PSp}(2r,q)$, $\mathrm{PSU}(2r+1,q_0)$, and
$\mathrm{P}\Omega^-(2r+2,q)$ may be good candidates for groups which are
spreading but not \ffi{Q}. However, Pablo Spiga was able to show that
$\mathrm{PSp}(4,p)$ is non-spreading for $p=3,5,7$ by computational methods.
The issue is unresolved in general.)

We have seen examples of basic, but not almost synchronizing; and of almost synchronizing, but non-basic; and, of course, synchronizing groups are almost synchronizing and basic.

\subsection{Affine groups}

Recall that an \emph{affine group} is a permutation group $G$ on the
$d$-dimensional vector space over $\mathbb{F}_p$ (where $p$ is prime)
generated by the translation group $T$ and an irreducible linear group $H$.
Thus $G$ is the semidirect product of $T$ by $H$; and $H$ is the stabilizer
of the zero vector.

\begin{theorem}
Let $G$ be an affine permutation group on the $d$-dimensional vector space
over $\mathbb F_p$, with $H=G_0$ as above. Then the following are equivalent:
\begin{enumerate}\itemsep0pt
\item $G$ is spreading;
\item $G$ is \ffi{Q};
\item $H$ is transitive on the set of $1$-dimensional subspaces of $V$;
\item the group generated by $G$ and the scalars in $\mathbb{F}_p$ 
is $2$-transitive.
\end{enumerate}
\label{t:affine}
\end{theorem}

The affine groups described in the Theorem can be classified, using
the classification of affine $2$-transitive groups~\cite{hering,liebeck}.

\begin{proof}
It is clear that (c) and (d) are equivalent. Let us suppose
that they do not hold. Then $H$ is not transitive on $1$-dimensional spaces
of $V$, and hence not transitive on $(d-1)$-dimensional subspaces either
(by Brauer's lemma). Choose hyperplanes $A$ and $B$ in different orbits
of $H$. Then no image of $B$ under $G$ is parallel to $A$, so
$|A\cap Bg|=p^{d-2}$ for all $g\in G$. Thus $G$ is not spreading. So
(a) implies (c) and (d). It is clear that (b) implies (a); so it remains
to prove that (c) and (d) imply that $G$ is \ffi{Q}.

The scalars in $\mathbb{F}_p$ act on
$V$, and hence on the characters of $V$; and their action is precisely that
of the Galois group of the field of $p$th roots of unity. Now assuming that
(d) holds, this group permutes the non-principal irreducibles in the
permutation character transitively, and so $G$ is \ffi{Q}, as required.\qed
\end{proof}

The equivalence of (b)--(d) is due to Dixon~\cite{dixon}.

A transitive group $G$ of prime degree $p$ contains a $p$-cycle $a$ and hence is an affine group.  Clearly (c) of the theorem is satisfied and so $G$ is spreading. Let $C=\langle a\rangle$;
then $C$, itself, is spreading and each element of $C$ can be expressed by a word of length at most $p-1$ in $\{a\}$.  Therefore, if $f$ is any singular mapping, then
Theorem~\ref{t:spreading} yields a reset word over $\{a,f\}$ of length at most $1+p(p-2)=(p-1)^2$.  This result, due to Pin~\cite{Pincerny}, was the first positive result concerning the
 \v{C}ern\'y conjecture.

\subsection{$3/2$-transitive groups}

A permutation group $G$ on $\Omega$ is said to be \emph{$3/2$-transitive}
if it is transitive, and the stabilizer of a point $v$ has all its orbits
except $\{v\}$ of the same size. (If there is just one such orbit then $G$
is $2$-transitive.)

\begin{example}
Let $q$ be a power of $2$. The group $\PSL(2,q)$ has dihedral subgroups of
order $2(q+1)$; it acts transitively on the set of cosets of such a subgroup,
and the stabilizer has $q/2-1$ orbits each of size $q+1$ on the remaining
points.
\label{e:321}
\end{example}

\begin{example}
There is a ``sporadic'' example: the symmetric group $S_7$ acting on
$2$-subsets of $\{1,\ldots,7\}$. This works because $2\cdot5={5\choose2}$.
\label{e:322}
\end{example}

Using the Classification of Finite Simple Groups, John Bamberg, Michael
Giudici, Martin Liebeck, Cheryl Praeger and Jan Saxl~\cite{bglps} have
determined the almost simple $3/2$-transitive groups. Apart
from the affine groups of Theorem~\ref{t:affine},  the only primitive
$3/2$-transitive groups are those of Examples~\ref{e:321} and~\ref{e:322}.

Although the class of $3/2$-transitive groups is not closed upwards, this
classification gives us the \ffi{Q}-groups:

\begin{theorem}
Any \ffi{Q} group is $3/2$-transitive.
\end{theorem}

The reason is that the permutation character is the sum of the trivial
character and a family of algebraically conjugate characters; an old result
of Frame~\cite{frame1,frame2} (see~\cite[\S30]{wielandt}) now applies.
This was first observed by Dixon~\cite{dixon}.

Now the group $S_7$ acting on $2$-sets is not \ffi{Q}.
Careful analysis of the character values of $\PSL(2,q)$ show that the
$3/2$-transitive action of this group described earlier is \ffi{Q}
if and only if $q-1$ is a Mersenne prime.
So there are probably infinitely many examples of this form (the
\emph{Lenstra--Pomerance--Wagstaff conjecture}~\cite{wagstaff}), though nobody
knows for sure.

Any other \ffi{Q} group is affine, and as we have seen in
Theorem~\ref{t:affine}, these groups are classified.

\subsection{\ffi{Q} versus spreading}

We don't know any examples of groups which are spreading but not \ffi{Q}.
Moreover, there are very few \ffi{Q} groups, and there are plenty of
places to look for spreading groups.

We saw above that
\begin{itemize}\itemsep0pt
\item $G$ is not \ffi{Q} if and only if there are non-trivial multisets
$A$ and $B$ satisfying $(1)_\lambda$,
\end{itemize}
whereas, by definition,
\begin{itemize}
\item $G$ is not spreading if and only if there are non-trivial multisets
$A$ and $B$ satisfying $(1)_\lambda$, $(3)$ and $(4)$.
\end{itemize}
Condition $(3)$ says that $B$ is a set. In combinatorial problems of this
kind, there is usually a big difference between asking for a multiset with
a certain property and asking for a set. This reason and others suggest
that such groups will exist; but none have yet been found!

Further applications of representation theory to synchronization can be found in~\cite{mortality,steinbergbook}.

\clearpage

\section{Detecting properties with functions}

In this section we are going to expand on the detection of the primitive, synchronizing, spreading or separating properties using functions. The first,
motivating result is Rystsov's Theorem (Theorem~\ref{t:rystsov}), which we
reformulate here:

\begin{theorem}\label{t:ryst.ref} Let $G$ be a transitive permutation group on $\Omega$, with $|\Omega|=n$.
 The following are equivalent:
\begin{enumerate}
  \item $G$ is primitive;
  \item for
any function $f\colon \Omega\to\Omega$ whose image has cardinality $n-1$, the
semigroup generated by $G$ and $f$ contains a constant function;
  \item for any idempotent
 function $f\colon \Omega\to\Omega$ whose image has cardinality $n-1$, the
semigroup generated by $G$ and $f$ contains a constant function.
\end{enumerate}
\end{theorem}

Rystsov does not, in fact, explicitly state the above theorem.  But in~\cite{rystsov} he proved that if $a\neq b$ are elements of $\Omega$, then the orbital digraph corresponding to $(a,b)$ is connected if and only if $G$ synchronizes the rank $n-1$ idempotent $e$ defined by
\[xe=\cases{x & if $x\ne a$;\cr b & if $x=a$.\cr}\]  Theorem~\ref{t:ryst.ref} is an immediate consequence of this result,  Higman's characterization of primitivity in terms of the connectivity of orbital digraphs and the easy observation that if $G$ is a transitive group and $f$ is any rank $n-1$ mapping, then $\langle G\cup f\rangle$ contains a rank $n-1$ idempotent.

\begin{theorem} Let $G$ be a permutation group on $\Omega$, with $|\Omega|=n$.
 The following are equivalent:
\begin{enumerate}
  \item $G$ is $2$-homogeneous;
  \item for any function $f\colon \Omega\to\Omega$ whose image has size $n-1$, the
semigroup generated by $G$ and $f$ contains all transformations which are
not permutations;
  \item for any idempotent function $f\colon \Omega\to\Omega$ whose image has size $n-1$, the
semigroup generated by $G$ and $f$ contains all transformations which are
not permutations.
\end{enumerate}
\end{theorem}

\begin{proof}
It is obvious that (b) implies (c).

To prove that (a) implies (b), let $G$ be $2$-homogeneous and let $f\colon \Omega\to\Omega$ be a rank $n-1$ map, whose unique non-singleton kernel class is $\{a,b\}$ and $\{a_0\}=\Omega\setminus \Omega f$. We claim that $\langle G,f\rangle$ contains all idempotents of rank $n-1$. In fact, if $e$ is one such idempotent, with non-singleton kernel class $\{c,d\}$ and $\{c_0\}=\Omega\setminus \Omega e$, there exist $h,g\in G$ such that $\{c,d\}g=\{a,b\}$ and $a_0h=c_0$. Therefore, the unique non-singleton kernel class of $gfh$ is $\{c,d\}$ and $\Omega gfh=\Omega e$. Since $e$ is idempotent if follows that its image is a section of its kernel, and hence the same happens with $gfh$; thus $\rank(gfh)^k=\rank(e)$, for all natural numbers $k$, and hence there exists one $k_0$ such that $(gfh)^{k_0}$ is idempotent, having the same kernel and image as $e$. Since there is only one such idempotent  it follows that  $(gfh)^{k_0}=e$ and the claim is proved. It is well known (\cite{howie66}) that the rank $n-1$ idempotent maps generate all non-invertible maps and this concludes the proof of the implication.

To prove that $(b)$ or $(c)$ implies  $(a)$ suppose $f$ is a map with non-singleton kernel class $\{a,b\}$ such that $ \langle f,G\rangle$ generate all non-invertible maps. Let $f'\colon \Omega\to\Omega$ be a rank $n-1$ map with non-singleton kernel class $\{c,d\}$. Since, by hypothesis, $f'\in \langle f,G\rangle$, it follows that $f'=g_1fg_2\ldots fg_k$, and hence $(c,d)\in \ker(g_1fg_2\ldots fg_k)$. As $\rank(g_1fg_2\ldots fg_k)=\rank(f)$ it follows that $\ker(g_1fg_2\ldots fg_k)=\{(c,d)\}=\ker(g_1f)$; thus there exists $g_1\in G$ such that $\{c,d\}g_1=\{a,b\}$. As $\{c,d\}$ was an arbitrary $2$-set it follows that $G$ has only one orbit on $2$-sets. The two implications follow. \qed
\end{proof}

This theorem was first proved by McAlister~\cite{mcalister}.

Denote by $\Unif (\Omega)$ the set of functions $f\colon \Omega\rightarrow \Omega$ whose kernel is a uniform partition with at least two parts. The next result provides some characterizations of synchronizing groups.

\begin{theorem}\label{t:synch.char} Let $G$ be a transitive permutation group on $\Omega$, with $|\Omega|=n$.
 The following are equivalent:
\begin{enumerate}
  \item there is no non-trivial partition $P$  and set $A$ such that $Ag$ is a section for $P$, for all   $g\in G$;
  \item for
any function $f\colon \Omega\to\Omega$ which is not a permutation, the semigroup
generated by $G$ and $f$ contains a constant function;
  \item for
any idempotent function $f\colon \Omega\to\Omega$ which is not a permutation, the semigroup
generated by $G$ and $f$ contains a constant function;
  \item for all $t\in \Unif(\Omega )$, there exists a part $A$ of $\Ker(t)$ and $g\in G$ such that $|\Omega tg\cap A|> 1$.
\end{enumerate}
\end{theorem}

\begin{proof}
The equivalence (a) and (b) is the content of Theorem~\ref{t:sec.reg.char}; the equivalence of (b) and (c) is immediate since every mapping has an idempotent positive power.

The implication (d) implies (b) is essentially the content of Corollary~\ref{c:min.rank.unif}:  if $G$ is not synchronizing, then a minimal rank mapping $t$ not synchronized by $G$ belongs to $\Unif(\Omega)$. By (d), there exist $g\in G$ and a part $A$ of $\Ker(t)$ such that $|\Omega tg\cap A|>1$. Therefore $\Rank(tgt)<\Rank(t)$, a contradiction to the choice of $t$.

Conversely, suppose that $t\in \Unif(\Omega)$  is such that   $\langle G,t\rangle$ contains a constant mapping.  Then there exists $g\in G$ such that $\Rank(tgt)<\Rank(t)$, which is equivalent to saying that, for some part $A$ of $\Ker(t)$, we have $|\Omega tg \cap A|>1$. The result follows.
\qed
\end{proof}

The following result provides a characterization of separation that parallels the equivalence of (a) and (d) in the previous result.

\begin{theorem} Let $G$ be a transitive permutation group on $\Omega$, with $|\Omega|=n$.
 The following are equivalent:
\begin{enumerate}
  \item $G$ is separating;
  \item for all $t\in \Unif(\Omega )$ and all parts $A$ of $\Ker(t)$, there exists  $g\in G$ such that $|A\cap \Omega tg|>1$.
\end{enumerate}
\end{theorem}

\begin{proof}
Suppose that $G$ is separating, and let $t\in \Unif(\Omega)$ be singular. Put $B=\Omega t$ and let $A$ be an arbitrary $\Ker (t) $-class. Then $|A|\cdot |B|=|\Omega|$ (because $t$ is uniform) and there exists $g\in G$ such that $|A\cap Bg|> 1$, as $G$ is separating and the average value of $|A\cap Bg|$ is $1$ by Theorem~\ref{t:average.value}.

Conversely, suppose that $G$ is not separating.  Then we have two non-trivial subsets $A,B$ of $\Omega$ such that $|A|\cdot |B|=|\Omega|$ and $|A\cap Bg|=1$ for all $g\in G$. Let $t\in \Unif(\Omega)$ be any mapping such that $A$ is a part of $\Ker(t)$ and $\Omega t=B$. Then, by (b), there exists $g\in G$ such that $|A \cap Bg|> 1$, a contradiction.\qed
\end{proof}

The above theorem, in light of Theorem~\ref{t:synch.char}, provides another proof that separating groups are synchronizing.

We close this section with another characterization of spreading groups.

\begin{theorem}
Let $G\leq S_n$ be a transitive group acting on $\Omega$. The following are equivalent:
\begin{enumerate}
  \item for all proper subsets $A$ of $\Omega$ and singular mappings $t\in T(\Omega)$, there exists $g\in G$ such that $|Agt^{-1}|> |A|$;
  \item for all proper subsets $A$ of $\Omega$ and idempotent singular mappings $e\in T(\Omega)$, there exists $g\in G$ such that  $|A ge^{-1}|> |A|$.
\end{enumerate}
\end{theorem}

\begin{proof}
That (a) implies (b) is obvious. Conversely, let $t$ be a singular mapping on $\Omega$. Then there exists a singular idempotent mapping $e$  and $h\in S_n$ such that $h^{-1}e=t$.
By (b), there exists $g\in G$ such that $|Age^{-1}|> |A|$. Therefore,
$|Agt^{-1}| = |Age^{-1}h|=|Age^{-1}|>|A|$, as required.
\end{proof}

Observe that  synchronizing groups can be defined in terms of functions or in terms of section-regular partitions. The previous result, allowing a definition of spreading groups in terms of idempotents, leads to a parallel  definition in terms of partitions and sections. A group $G$ is spreading if and only if, for every $A \subsetneq \Omega$, and every partition $P$ of $\Omega$ with section $B$, there exists $g\in G$ such that $|Ag*B'|>|A|$, where $B'$ is the multiset with support $B$ that gives $x\in B$ multiplicity the size of its part in $P$.

\clearpage

\section{Applications to the \v{C}ern\'y Conjecture}

So far, with the exception of Theorem~\ref{t:spreading}, we have not provided any bounds on lengths of reset words.  In this section we prove a new result, generalizing previous results of Rystsov~\cite{rystsov:regular} and the third author~\cite{steinberg}, giving bounds for reset words in the case of a transitive permutation group and a collection of singular transformations that it synchronizes.

\subsection{Transitive permutation groups}
Let us set up some notation.  If $G$ is a finite group and $A$ is a generating set, then we write $d_A(G)$ for the smallest integer $d\geq 0$ such that $G=(A\cup \{1\})^d$.  One can think of $d_A(G)$ as the directed diameter of the Cayley digraph of $G$ with respect to $A$.  All our bounds are based on this parameter.  Trivially, $d_A(G)\leq |G|-1$ since a shortest directed path in the Cayley digraph of $G$ with respect to $A$ from $1$ to any vertex has length at most $|G|-1$.

If $M\subseteq T(\Omega)$ is a transformation monoid, then the $\mathbb Q$-vector space $\mathbb Q^{\Omega}$ of mappings $f\colon \Omega\to \mathbb Q$ is a left $\mathbb QM$-module via the action defined by $(mf)(x)=f(xm)$ for $m\in M$  and $x\in \Omega$.  Moreover, the subspace $V_1$ of constant mappings is a $\mathbb QM$-submodule isomorphic to the trivial $\mathbb QM$-module.  Notice that the character $\theta$ of $\mathbb Q^{\Omega}$ is given by \[\theta(m)=\left|\{x\in \Omega : xm=x\}\right|\]  and so, for a group, this is just the character of the permutation module $\mathbb Q\Omega$.  However, for monoids there is a significant difference between the transformation module $\mathbb Q\Omega$ and its vector space dual $\mathbb Q^{\Omega}$; for instance, if $M$ is synchronizing and transitive, then $\mathbb Q\Omega$ is a projective indecomposable right $\mathbb QM$-module with the trivial module as its simple top, whereas $\mathbb Q^{\Omega}$ is an injective indecomposable left $\mathbb QM$-module with the trivial module as its simple socle (see~\cite{steinbergbook} or~\cite{transformations} for details).

If $M$ is a monoid, $V$ is a left $\mathbb QM$-module,
$C\subseteq M$ and $W\subseteq V$ is a subspace, then we denote by $CW$ the linear span of all vectors of the form $cw$ with $c\in C$ and $w\in W$.

\begin{lemma}\label{spanning}
Let $G$ be a group generated by a set $A$ and let $V$ be a finite dimensional $\mathbb QG$-module.  Suppose that $W\subseteq V$ is a subspace.  Then the equality $(A\cup \{1\})^dW=GW$ holds where $d=\dim V-\dim W$.
\end{lemma}
\begin{proof}
Put $W_i=(A\cup \{1\})^iW$.  Then we have an increasing chain
\[W_0\subseteq W_1\subseteq\cdots\] of subspaces of $V$ whose union is $GW$. It follows by dimension considerations that $W_i=W_{i+1}$ for some $0\leq i\leq d$.  But this means $W_i=GW$ and hence $W_d=GW$, as required.\qed
\end{proof}

Now we can establish our desired synchronization bound.

\begin{theorem}\label{t:irred.bound.trans}
Let $G$ be a transitive permutation group on a set $\Omega$ of cardinality $n\geq 2$  and let $A$ be a generating set for $G$.  Suppose that $B\subseteq T(\Omega)$ is such that $\langle G\cup B\rangle$ is synchronizing. Then there is a reset word over  $A\cup B$ of length at most \[1+(n-m+d_A(G))(n-2)\] where $m$ is the maximum dimension of an irreducible $\mathbb QG$-submodule of $\mathbb Q^{\Omega}$.

In particular, in the case that $m\geq d_A(G)$, there is a reset word over $A\cup B$ of length at most $(n-1)^2$.
\end{theorem}
\begin{proof}
We claim that if $S\subsetneq \Omega$ is a proper subset with at least two elements, then there exists a word $v$ over $A\cup B$ of length at most $n-m+d_A(G)$ such that $|Sv^{-1}|>|S|$.  Let us see why this claim implies the proposition.  Since $A\cup B$ is synchronizing, it contains a singular mapping $b\in B$.  Let $x\in \Omega$ with $|xb^{-1}|\geq 2$.  The claim then finds us a sequence $v_1,\ldots, v_k$ of words over $A\cup B$ of length at most $n-m+d_A(G)$ such that $|xb^{-1}v_1^{-1}\cdots v_k^{-1}|=n$ with $1\leq k\leq n-2$.  Thus $v_kv_{k-1}\cdots v_1b$ is a reset word (with image $x$) of length at most $1+(n-m+d_A(G))(n-2)$.

To prove the claim, set $M=\langle G\cup B\rangle$ and let $V=\mathbb Q^{\Omega}$ with the left $\mathbb QM$-module structure described above.
 There is a direct sum decomposition of $V$ as a $\mathbb QG$-module $V=V_0\oplus V_1$ where
\[V_0=\left\{f\in V: \sum_{x\in \Omega}f(x)=0\right\}\] is a hyperplane and $V_1$ is the line consisting of constant mappings.  It is well known that the transitivity of $G$ implies that $V_1$ is the isotypic component of the trivial $\mathbb QG$-module and hence the operator $P=\sum_{g\in G}g$ annihilates $V_0$; indeed, $P$ is a scalar multiple of the primitive idempotent corresponding to the trivial representation.  Note that $V_1$ is a $\mathbb QM$-submodule, but $V_0$ is not.  We remark that $m$ is the dimension of an irreducible $\mathbb{Q}G$-submodule of $V_0$.

Denote by $\chi_A$ the characteristic function of a subset $A\subseteq \Omega$ and consider the vector \[\gamma_S = \chi_S-\frac{|S|}{n}\chi_{\Omega}\in V_0.\]   Let $W_0=G\gamma_S$ be the $\mathbb QG$-submodule generated by $\gamma_S$.  Then $W_0\subseteq V_0$.  On the other hand, if $w$ is a reset word over $A\cup B$ with $\Omega w\subseteq S$ (such exists by transitivity of $G$), then \[w\gamma_S=\chi_{Sw^{-1}}-\frac{|S|}{n}\chi_{\Omega}=\left(1-\frac{|S|}{n}\right)\chi_{\Omega}\notin V_0.\]  Thus $\mathbb QM\cdot W_0\nsubseteq V_0$.  It follows that if $W$ is any $\mathbb QG$-submodule of $V_0$ containing $W_0$, then there exists $b\in B$ with $\{1,b\}W\supsetneq W$.  Therefore, we can choose $b_1,\ldots, b_j\in B$ such that if \[W_i=G\{1,b_i\}G\{1,b_{i-1}\}G\cdots G\{1,b_1\}W_0,\] then we have \[W_0\subsetneq\cdots \subsetneq W_{j-1}\subseteq V_0\] and $W_j\nsubseteq V_0$.  For convenience, we put $W_{-1}=0$.

Note that $\{1,b_{i+1}\}W_i\supsetneq W_i$ for $0\leq i\leq j-1$ and so, by repeated application of Lemma~\ref{spanning}, 
\[W_{j-1} = (A\cup \{1\})^{d_{j-1}}\{1,b_{j-1}\}\cdots (A\cup \{1\})^{d_1}\{1,b_1\}(A\cup \{1\})^{d_0}\gamma_S\]
where $d_i=\dim W_i-\dim W_{i-1}-1$ for $0\leq i\leq j-1$. Thus
we can find words $w_0,\ldots, w_j$ over $A$ such that $b_jw_{j-1}b'_{j-1}\cdots w_1b'_1w_0\gamma_S\notin V_0$ where $b_i'\in \{1,b_i\}$ and \[|w_i|\leq \dim W_i-\dim W_{i-1}-1\] for all $i=0,\ldots, j-1$.  Therefore, we have
\begin{equation}\label{eq:ben3}
0\neq \sum_{x\in \Omega}\gamma_S(xb_jw_{j-1}b_{j-1}'\cdots w_1b_1'w_0)=|S(b_jw_{j-1}b_{j-1}'\cdots w_1b_1'w_0)^{-1}|-|S|.
\end{equation}

Let $U$ be the isotypic component of $V_0$ corresponding to an irreducible $\mathbb QG$-module of dimension $m$. We consider two cases.

First assume that $U\cap W_{j-1}=0$.  Then $V_0/W_{j-1}$ contains an irreducible constituent of dimension $m$ and so $\dim W_{j-1}\leq n-1-m$.  Putting $u=b_jw_{j-1}b_{j-1}'\cdots w_1b_1'w_0$, we have
\[|u| \leq j+\sum_{i=0}^{j-1}|w_i|\leq j+\sum_{i=0}^{j-1}(\dim W_i-\dim W_{i-1}-1)=\dim W_{j-1}\leq n-1-m.\]
On  the other hand, since $P\gamma_S=0$, it follows that
\begin{equation}\label{eq:ben4}
0=\sum_{x\in \Omega}uP\gamma_S(x)=\sum_{g\in G}(|S(ug)^{-1}|-|S|).
\end{equation}
Since $|Su^{-1}|-|S|\neq 0$ by equation~(\ref{eq:ben3}), we conclude by
equation~(\ref{eq:ben4}) that $|S(ug)^{-1}|>|S|$ for some $g\in G$. 
As $u$ has length at most $n-1-m$ and $g$ can be represented by some word
over $A$ of length at most $d_A(G)$, we deduce that $|Sv^{-1}|>|S|$ for some
word $v$ over $A\cup B$ of length at most $n-m+d_A(G)$, as required.

Suppose next that $U\cap W_{j-1}\neq 0$ and let $0\leq k\leq j-1$ be smallest with $U\cap W_k\neq 0$.  Then $\dim W_k-\dim W_{k-1}\geq m$.  Therefore, putting \[u'=b_jw_{j-1}b_{j-1}'\cdots w_{k+1}b_{k+1}'\quad \mbox{and}\quad  u''=b_k'w_{k-1}\cdots w_1b_1'w_0,\] we have that
\begin{eqnarray*}
|u'u''|\leq j+\sum_{i\in \{0,\ldots, j-1\}\setminus \{k\}}(\dim W_i-\dim W_{i-1}-1)\\
=1+\dim W_{j-1}-(\dim W_k- \dim W_{k-1}) \leq n-m.
\end{eqnarray*}
Using that $Pu''\gamma_S=0$, as $u''\gamma_S\in W_k\subseteq V_0$, we have that
\begin{equation}\label{eq:ben5}
0=\sum_{x\in \Omega}u'Pu''\gamma_S(x)=\sum_{g\in G}(|S(u'gu'')^{-1}|-|S|).
\end{equation}
Equation~(\ref{eq:ben3}) says that $|S(u'w_ku'')^{-1}|-|S|\neq 0$ and hence,
as $w_k\in G$,
we deduce that $|S(u'gu'')^{-1}|>|S|$ for some $g\in G$ by equation~(\ref{eq:ben5}).
As $|u'u''|\leq n-m$ and $g$ can be represented by some word over $A$ of length at most $d_A(G)$, it follows that $|Sv^{-1}|>|S|$ for some word $v$ over $A\cup B$ of length at most $n-m+d_A(G)$.  This completes the proof of the claim.

The theorem now follows, where the final statement is just the observation that $(n-1)^2=1+n(n-2)$.
\qed
\end{proof}

Of course, if $G$ is a \ffi{Q} group, then $m=n-1$ and so Theorem~\ref{t:irred.bound.trans} recovers the bound of $1+(d_A(G)+1)(n-2)$ obtained via the spreading property in Theorem~\ref{t:spreading}.  The weakening of the bound in Theorem~\ref{t:irred.bound.trans} that replaces $m$ by $1$ is essentially contained in the results of Rystsov~\cite{rystsov:regular}.

As an example, consider $S_k$ acting on the set $\Omega$ of $2$-sets of $\{1,\ldots, k\}$ with $k\geq 4$.  Then $n=|\Omega|={k\choose 2}$.   It is well known $\mathbb Q^{\Omega}$ is multiplicity-free with three irreducible submodules of dimensions $1$, $k-1$ and ${k\choose 2}-k=n-k$.  Theorem~\ref{t:irred.bound.trans} then shows that if $A$ is any generating set for $S_n$ and $B\subseteq T(\Omega)$ is such that $\langle S_n\cup B\rangle$ is synchronizing (e.g., if $k$ is odd, then any subset $B$ containing a singular map will do), then there is a reset word over $A\cup B$ of length at most $1+(k+d_A(S_k))(n-1)$.  In particular, if $d_A(S_k)\leq n-k$, then the \v{C}ern\'y bound is achieved.  For example, suppose $A$ is the set of Coxeter--Moore generators $(1,2)$, $(2,3)$, \dots, $(k-1,k)$.  Then $d_A(S_k)={k\choose 2}=n$ and so we obtain a bound of $1+(k+n)(n-2)\leq (n-1)^2+(\sqrt{2n}+1)(n-2)$, as $k\leq \sqrt{2n}+1$.

If $s$ is the number of irreducible constituents (with multiplicities) of $V_0$, then clearly $ms\geq \dim V_0=n-1$.  On the other hand, the number of irreducible constituents of the augmentation submodule of the permutation module over $\mathbb Q$ is less than the number of irreducible constituents of the augmentation submodule over $\mathbb C$ (with multiplicities).  If $r$ is the rank of the transitive permutation group $G$ acting on $\Omega$, then $r-1$ is the sum of the squares of the multiplicities of the complex irreducible constituents of the augmentation submodule.  Therefore, $r-1\geq s$ and so $m\geq \frac{n-1}{r-1}$.  Thus Theorem~\ref{t:irred.bound.trans} admits the following corollary, which avoids representation theoretic language.

\begin{cor}
Let $G$ be a transitive permutation group on a set $\Omega$ of cardinality $n\geq 2$  and let $A$ be a generating set for $G$.  Suppose that $B\subseteq T(\Omega)$ is such that $\langle G\cup B\rangle$ is synchronizing. Then there is a reset word over  $A\cup B$ of length at most \[1+\left(n-\frac{n-1}{r-1}+d_A(G)\right)(n-2)\] where $r$ is the rank of the permutation group $G$.
\end{cor}

\subsection{Regular permutation groups}
We next consider regular permutation groups, that is, transitive permutations groups with trivial point stabilizers.  Up to isomorphism, this means that we have a finite group $G$ acting on the right of itself, and so for the purpose of this discussion we shall take $\Omega=G$.  Notice that the $G$-invariant graphs in this case are precisely the left Cayley graphs of $G$ with respect to some subset $S\subseteq G$ (not necessarily a generating set).  Thus $G$ synchronizes $f\in T(G)$ if and only if $f$ is not an endomorphism of any non-trivial left Cayley graph of $G$. The only regular permutation groups which are primitive are of prime degree.

In the original paper of \v{C}ern\'y~\cite{cerny}, the worst case synchronizing automata were constructed by starting with a cyclic permutation of the state set and adjoining an idempotent of rank $n-1$.  A cyclic permutation of the states generates a regular permutation group  and it is therefore natural to consider in general how quickly regular permutation groups ``synchronize'' mappings.  Here, we are differing from our previous terminology a bit because some of the mappings we adjoin may be permutations, where before we were only adjoining singular mappings.  The first result in this subject is due to Rystsov, who proved a slightly more general statement than our formulation~\cite{rystsov:regular}.

\begin{prop}[Rystsov]\label{prop:rystsov.regular}
Let $G$ be a finite group of order $n$.  Let $A$ be a generating set for $G$ and $B\subseteq T(G)$ such that $\langle G\cup B\rangle$ is synchronizing.  Then there is a reset word over $A\cup B$ of length at most $2n^2-6n+5$.
\end{prop}

This can be obtained from Theorem~\ref{t:irred.bound.trans} by using that $m\geq 1$ and that $d_A(G)\leq n-1$ for a regular permutation group.
Proposition~\ref{prop:rystsov.regular} was improved upon by the third author to the bound in the next theorem, which is sharp in the case of a cyclic group of prime order~\cite{steinberg}.

\begin{theorem}\label{t:irred.bound}
Let $G$ be a finite group of order $n$ and $A$ a generating set for $G$.  Suppose that $B\subseteq T(G)$ is such that $\langle G\cup B\rangle$ is synchronizing. Then there is a reset word over  $A\cup B$ of length at most \[1+(n-m(G)+d_A(G))(n-2)\leq 1+(2n-1-m(G))(n-2)\] where $m(G)$ is the maximum dimension of an irreducible $\mathbb QG$-module.

In particular, in the case that $m(G)\geq d_A(G)$, there is a reset word over $A\cup B$ of length at most $(n-1)^2$.
\end{theorem}

Theorem~\ref{t:irred.bound} is immediate from Theorem~\ref{t:irred.bound.trans} for the case $\Omega=G$ once we make the following observation:  the module $V=\mathbb Q^G$ is isomorphic to the regular $\mathbb QG$-module and hence contains every irreducible $\mathbb QG$-module as a submodule.

For example, if $G$ is a cyclic group of prime order $p$, then \[\mathbb QG\cong \mathbb Q\times \mathbb Q[x]/(1+x+\cdots+x^{p-1})\] and so $m(G)=p-1$, whereas $d_A(G)\leq p-1$ for any generating set of $G$, with equality for a singleton generating set.  Thus Theorem~\ref{t:irred.bound} achieves the \v{C}ern\'y bound of $(p-1)^2$ in this case, recovering Pin's theorem~\cite{Pincerny}.
For a cyclic group of order $n$, in general, the bound obtained by Theorem~\ref{t:irred.bound} is $1+(2n-\phi(n)-1)(n-2)$ where $\phi$ is the Euler totient function.

It is well known that each irreducible representation of the symmetric group $S_k$ over $\mathbb Q$ is absolutely irreducible. It follows that $S_k$ has $p_k$ irreducible representations over $\mathbb Q$, where $p_k$ is the number of partitions of $k$, and that the sum of the degrees squared of these representations is $k!$.  Thus $p_km(S_k)^2\geq k!$ and so we obtain $m(S_k)\geq \sqrt{k!/p_k}$.  It is a well-known result of Hardy and Ramanujan that \[p_k\sim \frac{\exp\left(\pi\sqrt{2k/3}\right)}{4k\sqrt{3}}.\]  On the other hand, Stirling's formula says that $k!\sim \sqrt{2\pi k}\left(\frac{k}{e}\right)^k$.  Comparing these expressions, we see that $m(S_k)$ grows faster than any exponential function of $k$.  On the other hand,  $d_A(S_k)$ with respect to any of its usual generating sets grows polynomially with $k$. For instance, if one uses the Coxeter--Moore generators $(1,2)$, $(2,3)$, \dots, $(k-1,k)$ for $A$, then $d_A(S_k)={{k}\choose{2}}$, whereas if one uses the generators $(1,2),(1,2,\ldots, k)$ for $A$, then $d_A(S_k)\leq (k+1)k(k-1)/2$ since each Coxeter--Moore generator can be expressed as a product of length at most $k+1$ in these generators.  Thus Theorem~\ref{t:irred.bound} yields the \v{C}ern\'y bound for either of these generating sets for any $k$ sufficiently large.

Theorem~\ref{t:irred.bound} was used in~\cite{steinberg} to show that if $p\geq 17$ is prime and $B\subseteq T(\mathrm{SL}(2,p))$ is such that $\langle \mathrm{SL}(2,p),B\rangle$ is synchronizing, then there is a reset word over \[B\cup \left\{\left[\begin{array}{cc} 1 & 1\\ 0 & 1\end{array}\right], \left[\begin{array}{cc} 1 & 0\\ 1 & 1\end{array}\right]\right\}\] of length at most $(n-1)^2$ where $n$ is the order of $\mathrm{SL}(2,p)$. Further applications of Theorem~\ref{t:irred.bound} and its proof idea can be found in~\cite{steinberg}.

The most elegant result in the \v{C}ern\'y conjecture literature is Dubuc's theorem~\cite{dubuc}.

\begin{theorem}[Dubuc]\label{thm:dubuc}
Let $\Omega$ be a set of $n$ elements and suppose that $A\subseteq T(\Omega)$ contains a cyclic permutation of $\Omega$.  Then if $\langle A\rangle$ is synchronizing, there is a reset word over $A$ of length at most $(n-1)^2$.
\end{theorem}

In other words, a cyclic regular permutation group, together with a collection of mappings (some of which may be permutations), synchronizes within the \v{C}ern\'y bound whenever it synchronizes, provided that a generator of the cyclic group is one of the input letters for the automaton.  Since any generating set of a cyclic group of prime power order must contain a generator, we conclude that cyclic regular permutation groups of prime power degree, together with any collection of mappings (not necessarily singular) which generates a synchronizing monoid, always synchronizes within the \v{C}ern\'y bound.  To formalize this, let us say that a finite group $G$ of order $n$ is a \emph{\v{C}ern\'y group} if given any generating set $A$ of $G$ and any subset $B\subseteq T(G)$ such that $A\cup B$ is synchronizing, there is a reset word over $A\cup B$ of length at most $(n-1)^2$.

Theorem~\ref{thm:dubuc} implies that cyclic groups of prime power order are \v{C}ern\'y groups.  The third author proved in~\cite{steinberg} that elementary abelian $p$-groups are \v{C}ern\'y groups, as are dihedral groups of order $2p$ and $2p^2$ with $p$ an odd prime.  Conjecturally, all groups are \v{C}ern\'y groups, but this question is far from resolved. Note that since elementary abelian $p$-groups are \v{C}ern\'y groups, it follows that if one takes a synchronizing affine permutation group $G$, a singular mapping $f$ and a generating set $A$ for $G$ which contains a vector space basis for the subgroup of translations, then there is a reset word for $A\cup \{f\}$ of length at most $(n-1)^2$ where $n$ is the size of the vector space.

As a final comment on Dubuc's theorem, we note the following.

\begin{theorem}
Let $G$ be a permutation group of non-prime degree $n$ containing an $n$-cycle.
Then $G$ is synchronizing if and only if it is primitive.
\end{theorem}

\begin{proof}
The forward implication follows from Theorem~\ref{prim_synch}; the reverse
implication from a theorem of Burnside~\cite[Theorem 25.3]{wielandt},
according to which a primitive group containing a regular cyclic subgroup
of composite order is $2$-transitive.
\end{proof}

The primitive groups containing a cycle have been classified by
Gareth Jones~\cite{jones}.

Note also that there is a growing literature on the diameter of Cayley graphs
for certain groups, especially almost simple groups.
Babai (see~\cite{babai_seress}) conjectured that the diameter of any Cayley
graph for a simple group $G$ is bounded by a polynomial in $\log|G|$. Such
bounds have been found recently for several families of simple groups.
However, in these papers the diameter is always in the sense of
an undirected graph, whereas we are principally interested in the diameter as a directed graph.

\clearpage

\section{Other properties}

In this section, we survey briefly another class of permutation groups
which lie between the primitive and the $2$-homogeneous groups. Whether
there is a direct relationship is unknown.

\subsection{AS-friendly and AS-free groups}
\label{s:as}

The definition of these classes requires some background on coherent
configurations and association schemes. See~\cite{abc,c:cap} for more details.
The presentation here follows~\cite{c:cap}.

Coherent configurations were introduced independently by Donald Higman
\cite{higman,higman:cc} in the USA and by Weisfeiler and Leman~\cite{wl} in
the former Soviet Union to describe the orbits on pairs of a permutation
group.  Association schemes were introduced earlier by R.~C.~Bose and
collaborators~\cite{bs,bm} in connection with experimental design in statistics.

Let $\Omega$ be a finite set. A \emph{coherent configuration} (c.c.) on $\Omega$ is
a set $\mathcal{P}=\{R_1,\ldots,R_s\}$ of binary relations on $\Omega$
(subsets of $\Omega^2$) satisfying the following four conditions:
\begin{enumerate}\itemsep0pt
\item $\mathcal{P}$ is a partition of $\Omega^2$;
\item there is a subset $\mathcal{P}_0$ of $\mathcal{P}$ which is a
partition of the diagonal $\Delta=\{(a,a):a\in\Omega\}$;
\item for every relation $R_i\in\mathcal{P}$, its \emph{converse}
$R_i^\top = \{(b,a) : (a,b)\in R_i\}$ is in $\mathcal{P}$;
say $R_i^\top=R_{i^*}$.
\item there exist integers $p_{ij}^k$, for $1\le i,j,k\le s$, such
that, for any $(a,b)\in R_k$, the number of points $c\in\Omega$
such that $(a,c)\in R_i$ and $(c,b)\in R_j$ is equal to
$p_{ij}^k$ (and, in particular, is independent of the choice of
$(a,b)\in R_k$).
\end{enumerate}

A coherent configuration can be defined in terms of its \emph{basis matrices}
$A_1,\ldots,A_s$, where $A_i$ is the $\Omega\times\Omega$ matrix with
$(a,b)$ entry $1$ if $(a,b)\in R_i$, $0$ otherwise. In particular, condition
(d) asserts that $A_iA_j=\sum_{k=1}^sp_{ij}^kA_k$, so that the span of the
basis matrices is an algebra.

If $G$ is any permutation group on $\Omega$, then the partition of $\Omega^2$
into \emph{orbitals} (recall that these are the orbits of $G$ on $\Omega^2$)
is a coherent configuration, which we denote by
$\mathcal{K}(G)$. We refer to this as the \emph{group case}; a coherent
configuration of the form $\mathcal{K}(G)$ is called \emph{Schurian}. In
this case, the basis matrices span the centralizer algebra of the permutation
representation.

It is clear that a permutation group and its $2$-closure define the same
coherent configuration, so where necessary we can restrict our attention to
$2$-closed groups. Indeed, the $2$-closure of $G$ is just the automorphism
group of $\mathcal{K}(G)$ (the group of permutations preserving all the
relations in $\mathcal{K}(G)$).

Let $\mathcal{P}$ be a coherent configuration on $\Omega$. The sets $F$ such
that $\{(a,a):a\in F\}$ belong to~$\mathcal{P}$ are called the
\emph{fibres} of $\mathcal{P}$; they form a partition of $\Omega$. We say
that $\mathcal{P}$ is \emph{homogeneous} if there is only one fibre. If
$\mathcal{P}=\mathcal{K}(G)$, the fibres of~$\mathcal{P}$ are the orbits
of~$G$ on~$\Omega$; so $\mathcal{K}(G)$ is homogeneous if and only if $G$ is
transitive.

A coherent configuration is called \emph{commutative} if its
basis matrices commute with one another. It can be shown
that, if $\mathcal{P}=\mathcal{K}(G)$, then $\mathcal{P}$ is commutative if
and only if the permutation representation is
\emph{(complex)-multiplicity-free}.

A coherent configuration is called \emph{symmetric} if all the relations are
symmetric. A symmetric c.c.\ is homogeneous. (For, given any relation $R$ in
a c.c. with fibres $F_1, \ldots, F_t$, there are indices $i,j$ such that
$R\subseteq F_i\times F_j$.) If $\mathcal{P}=\mathcal{K}(G)$, then
$\mathcal{P}$ is symmetric if and only if $G$ is \emph{generously transitive},
that is, any two points of $\Omega$ are interchanged by some element of~$G$.
Symmetric coherent configurations are also known as \emph{association schemes},
although there is some disagreement over terminology here: each of the four
classes of coherent configurations appears with the name ``association scheme''
somewhere in the literature.

Let $\mathcal{P}$ be a c.c.\ on $\Omega$. The \emph{symmetrization}
$\mathcal{P}^{\mathrm{sym}}$ of $\mathcal{P}$ is the partition of $\Omega^2$
whose parts are all unions of the parts of $\mathcal{P}$ and their converses.
It may or may not be a c.c.; if it is, we say that $\mathcal{P}$ is
\emph{stratifiable}. The name arises in statistics~\cite{rab:strat}.
It can be shown that, if $\mathcal{P}=\mathcal{K}(G)$,
then $\mathcal{P}$ is stratifiable if and only if the permutation
representation of $G$ is \emph{real-multiplicity-free}, that is, if it is
decomposed into irreducibles over $\mathbb{R}$, they are pairwise
non-isomorphic. (Equivalently, the complex irreducibles have multiplicity at
most one except for those with Frobenius--Schur
index~$-1$, which may have multiplicity~$2$.)

Thus, the following implications hold:

\begin{prop}
A symmetric c.c.\ is commutative; a commutative c.c.\ is stratifiable; and a
stratifiable c.c.\ is homogeneous. None of these implications reverses.\qed
\end{prop}

We note also that, if $\mathcal{P}=\mathcal{K}(G)$, then $\mathcal{P}$ is
trivial if and only if $G$ is doubly transitive.

To motivate the next definition, we note that the join (in the lattice of
partitions) of c.c.s is a c.c.; the same holds for the subclasses defined
above. This allows us to define the meet of two c.c.s $\mathcal{C}_1$ and
$\mathcal{C}_2$ to be the join of all c.c.s below both of them in the
partition lattice; this class is non-empty since the configuration associated
with the trivial group (where all parts are singletons) is below any other
c.c. However, this does not apply to the subclasses; in particular, there is
no meet operation on association schemes.

Let $G$ be a transitive permutation group on the finite set $\Omega$.
\begin{enumerate}\itemsep0pt
\item We say that $G$ is \emph{AS-free} if the only $G$-invariant
association scheme on $\Omega$ is the trivial scheme.
\item We say that $G$ is \emph{AS-friendly} if there is a unique minimal
$G$-invariant association scheme on $\Omega$.
\end{enumerate}
Of course, if we replaced ``AS'' by ``CC'' in the above definitions, then
every group would be CC-friendly, and the CC-free groups would be precisely
the doubly transitive groups.

Note that a $2$-homogeneous group $G$ is AS-free, since the symmetrization of
$\mathcal{K}(G)$ is the trivial configuration.

It is easy to see that a uniform partition gives rise to an association scheme
(a \emph{group-divisible scheme}), while a Cartesian structure gives rise to an
association scheme (a \emph{Hamming scheme}). Thus,
\begin{itemize}\itemsep0pt
\item A transitive permutation group is primitive if and only if it preserves
no group-divisible association scheme;
\item A primitive permutation group is basic if and only if it preserves no
Hamming association scheme.
\end{itemize}
In a sense, then, the definition of AS-freeness simply carries this idea to
its logical conclusion!

\begin{example}
Here is an example of a group which is not AS-friendly. Let $G$ be the
symmetric group $S_n$ (for $n\ge5$), acting on the set $\Omega$ of ordered
pairs of distinct elements from the set $\{1, \ldots, n\}$: we write the
pair $(i,j)$ as $ij$ for brevity. The coherent
configuration $\mathcal{K}(G)$ consists of the following parts:
\begin{eqnarray*}
R_1&=&\{(ij,ij):i\ne j\},\\
R_2&=&\{(ij,ji):i\ne j\},\\
R_3&=&\{(ij,ik):i,j,k\hbox{ distinct}\},\\
R_4&=&\{(ij,kj):i,j,k\hbox{ distinct}\},\\
R_5&=&\{(ij,ki):i,j,k\hbox{ distinct}\},\\
R_6&=&\{(ij,jk):i,j,k\hbox{ distinct}\},\\
R_7&=&\{(ij,kl):i,j,k,l\hbox{ distinct}\}.
\end{eqnarray*}
We have $R_5^\top=R_6$; all other relations
are symmetric. The symmetrized partition is not an association scheme, but
we find three minimal association schemes as follows:
\begin{itemize}\itemsep0pt
\item the \emph{pair} scheme: $\{R_1,R_2,R_3\cup R_4,R_5\cup R_6,R_7\}$;
\item two ``divisible'' schemes
$\{R_1,R_3,R_2\cup R_4\cup R_5\cup R_6\cup R_7\}$ and
$\{R_1,R_4,R_2\cup R_3\cup R_5\cup R_6\cup R_7\}$.
\end{itemize}
\end{example}

\begin{theorem}\label{9.2}
The following implications hold between properties of a permutation group~$G$:

\begin{center}
\begin{tabular}{ccccccc}
$2$-transitive & $\Rightarrow$ & $2$-homogeneous & $\Rightarrow$ &
AS-free & $\Rightarrow$ & primitive \\
$\Downarrow$ & & $\Downarrow$ & & $\Downarrow$ & & $\Downarrow$ \\
gen.~trans. & $\Rightarrow$ & stratifiable & $\Rightarrow$ &
AS-friendly & $\Rightarrow$ & transitive
\end{tabular}
\end{center}

None of these implications reverses, and no further implications hold.
\end{theorem}

The smallest $2$-closed primitive group which is not AS-friendly is
$\mathrm{PSL}(2,11)$, with degree~$55$. The smallest $2$-closed primitive
groups which are AS-friendly but not stratifiable are $\mathrm{PSL}(2,13)$,
with degrees~$78$ and~$91$. These groups are numbers $(55,1)$, $(78,1)$,
$(91,1)$ and $(91,3)$ in the list of primitive groups available in
\textsf{GAP}. The smallest examples of AS-free
groups which are not stratifiable have degree~$234$, and are isomorphic to
$\mathrm{PSL}(3,3)$ and $\mathrm{PSL}(3,3):2$, numbers $(234,1)$
and $(234,2)$ in the list. (Further examples of such groups will be given
later.) $2$-homogeneous groups which are not generously transitive are well
known, as we have seen.

For the class of AS-free groups, we have:

\begin{theorem}
Let $G$ be a transitive AS-free group. Then $G$ is primitive and basic, and is
$2$-homogeneous, diagonal or almost simple.
\end{theorem}

Almost simple AS-free groups which are not $2$-homogeneous do
exist. This can be seen from the paper of Farad\v{z}ev
\emph{et~al.}~\cite{fetal}. These authors consider the following problem.
\emph{Let $G$ be a simple primitive permutation group
of order at most $10^6$ but not $\mathrm{PSL}(2,q)$. Describe
the coherent configurations above $\mathcal{K}(G)$.} Table 3.5.1 on p.~115
gives their results. In several cases, no non-trivial configuration consists
entirely of symmetric matrices: such groups are
of course AS-free. The smallest example is the group $\mathrm{PSL}(3,3)$,
acting on the right cosets of $\mathrm{PO}(3,3)$ (a subgroup isomorphic to
$S_4$), with degree $234$; as we have seen, this is the smallest AS-free group
which is not $2$-homogeneous. Other examples of AS-free groups in this list
are the sporadic simple groups $M_{12}$, degree~$1320$; $J_1$, degree~$1463$,
$1540$ or $1596$; and
$J_2$, degree~$1800$. The situation is not well understood!

No AS-free primitive diagonal group is known at present. It is known that
the socle of such a group must have at least four simple factors. (Groups
with two factors preserve a coarsening of the conjugacy class scheme of one
factor, while groups with three factors preserve a ``Latin square'' scheme
based on the multiplication table of the factor.)

\clearpage

\section{The infinite}

We cannot simply take the definition of a synchronizing finite permutation
group and extend it to the infinite: there would be no such groups!

Let $\Omega$ be an infinite set. Then both the injective maps, and the
surjective maps, on $\Omega$ form submonoids of the full transformation
monoid; they contain the symmetric group but no rank~$1$ map.

Since the essence of synchronization seems to involve mapping different states
to the same place, it is reasonable to require that the map we adjoin is not
injective.

Our first attempt at a suitable definition is based on the following fact
about the finite case:

\begin{theorem}
Let $M$ be a transformation monoid on a finite set $\Omega$. Suppose that,
for any $v,w\in\Omega$, there exists $f\in M$ with $vf=wf$. Then $M$ is
synchronizing.
\end{theorem}

\begin{proof}
The hypothesis is clearly equivalent to the statement that $\Gr(M)$ is null.
\qed
\end{proof}

Accordingly, we could try a definition along the following lines:
\begin{enumerate}\itemsep0pt
\item A transformation monoid $M$ on $\Omega$ is \emph{synchronizing} if, for
any $v,w\in\Omega$, there exists $f\in M$ with $vf=wf$;
equivalently, $\Gr(M)$ is the null graph on $\Omega$.
\item A permutation group $G$ on $\Omega$ is \emph{synchronizing} if, for any
map $f\colon \Omega\to\Omega$ which is not injective, the monoid $\langle G,f\rangle$
is synchronizing.
\end{enumerate}

Unfortunately this doesn't give anything interesting!

\subsection{Ramsey's Theorem}

Ramsey's Theorem is much more general than the form given here; but this is
all we need.

Where necessary, we assume the Axiom of Choice, one of whose consequences
is that an infinite set contains a countably infinite subset.

\begin{theorem}
An infinite graph contains either an infinite clique or an infinite
independent set.
\end{theorem}

By our remark, it suffices to prove this for a countably infinite graph.

\begin{proof}
Let $v_1,v_2,\ldots$ be the vertices. We construct inductively a sequence
of triples $(x_i,Y_i,\epsilon_i)$, where the $x_i$ are distinct vertices,
$Y_i$ are infinite decreasing subsets of vertices, $x_i\in Y_j$ if and only
if $j<i$, and $x_i$ is joined to all or no vertices of $Y_i$ according as
$\epsilon_j=1$ or $\epsilon_j=0$. We begin with $Y_0$ the whole vertex set.
Choose $x_i\in Y_{i-1}$. By the Pigeonhole Principle, either $x_i$ has
infinitely many neighbours, or it has infinitely many non-neighbours, in
$Y_{i-1}$; let $Y_i$ be the appropriate infinite set and choose $\epsilon_i$
appropriately.

Now the sequence $(\epsilon_1,\epsilon_2,\ldots)$ has a constant subsequence;
the points $x_i$ corresponding to this subsequence form a clique or independent
set, depending on the constant value of $\epsilon_i$.\qed
\end{proof}

We use Ramsey's Theorem to show that the notion of ``synchronizing'' we just
defined is not interesting, at least for permutation groups of countable degree.
\begin{theorem}
Let $G$ be a permutation group of countable degree. Then $G$ is synchronizing
if and only if it is $2$-{homogeneous}.
\end{theorem}

\begin{proof}
Suppose that $G$ is not {$2$-homogeneous}. Then there is a non-trivial
$G$-invariant graph $\Gamma$ (take a $G$-orbit on $2$-sets as edges). Replacing
$\Gamma$ by its complement if necessary, and using Ramsey's theorem, we may
assume that $\Gamma$ has a countable clique $K$.
Let $v$ and $w$ be non-adjacent vertices. Choose a bijection $f$ from
$\Omega\setminus\{w\}$ to $K$, and extend it by setting $wf=vf$. Clearly
$f$ is an endomorphism of $\Gamma$ collapsing $v$ and $w$, and $\langle G,f\rangle$
is not a synchronizing monoid.

Conversely, if $G$ is $2${-homogeneous} and $f$ a map satisfying $vf=wf$,
then $(vg)(g^{-1}f)=(wg)(g^{-1}f)$ for any $g\in G$; so $\langle G,f\rangle$
collapses all pairs, and $G$ is synchronizing.\qed
\end{proof}

\subsection{Weak synchronization}

We look at a couple of modifications. We say that $G$ is \emph{weakly
synchronizing} if, for any map $f\colon \Omega\to\Omega$ of finite rank (that is,
having finite image), the monoid $\langle G,f\rangle$ contains a rank~$1$ map.

Now imprimitive groups may be weakly synchronizing; but it is true that a
weakly synchronizing group cannot have a finite system of blocks of
imprimitivity. For if $S$ is a transversal for such a system, and $f$ is the
map taking any point of $\Omega$ to the representative point of $f$, then
$\langle G,f\rangle$ contains no rank~$1$ map.

Note also that, if $M$ is a transformation monoid containing an element of
finite rank, and $\Gr(M)$ is null, then $M$ contains a rank~$1$ map.

\subsection{Strong synchronization}

Another possible approach: since, in general, words in $\langle G,f\rangle$
will not be reset words, we should allow infinite words. This requires some
preliminary thought.

Let $M$ be a transformation monoid on $\Omega$, and let $\overline{M}$ be its
\emph{closure} in the topology of pointwise convergence: a sequence $(f_n)$
of element of $M$ converges to the limit $f$ if, for all $v\in\Omega$,
there exists $n_0$ such that $vf_n=vf$ for all $n\ge n_0$.

Now we say that a permutation group $G$ is \emph{strongly synchronizing} if,
for any map $f$ which is not injective, the closure of $M=\langle G,f\rangle$
contains an element of rank~$1$.

\begin{theorem}
\begin{enumerate}\itemsep0pt
\item
A strongly synchronizing group is synchronizing.
\item
A $2$-{homogeneous} group of countable degree is strongly synchronizing.
\end{enumerate}
\end{theorem}

As a consequence of this theorem and the previous one about synchronizing
groups, a permutation group of countable degree is strongly synchronizing
if and only if it is {$2$-homogeneous}.

\begin{proof}
(a) Let $f$ be a map which is not injective, and let $(f_n)$ be a sequence
of elements of $\langle G,f\rangle$ converging to a rank~$1$ function with
image $\{z\}$, and choose two distinct points $x$ and $y$. There exist $n_1$
and $n_2$ such that $xf_n=z$ for $n\ge n_1$ and $yf_n=z$ for $n\ge n_2$.
So, if $n=\max(n_1,n_2)$, then $f_n\in\langle G,f\rangle$ and $xf_n=yf_n$.
So $G$ is synchronizing.

(b) Let $G$ be {$2$-homogeneous}
and let $f$ be a function which is not
injective. Choose two points $x$ and $y$ with $xf=yf$. By
post-multiplication by an element of $G$, we can assume that $xf=x$.

Enumerate $\Omega$, as $\{x_1,x_2,\ldots\}$, with $x_1=x$,
and construct a sequence $(f_n)$
of elements of $\langle G,f\rangle$ as follows. Begin with $f_1=f$. Now
suppose that $f_n$ is defined, and satisfies $x_mf_n=x$ for $m\le n$.
If $x_{n+1}f_n=x$, then choose $f_{n+1}=f_n$.
Otherwise, choose $g\in G$ mapping $\{x,x_{n+1}\}$ to $\{x,y\}$, and let
$f_{n+1}=f_ngf$. Clearly $x_mf_{n+1}=x$ for all $m\le n+1$.
So the sequence converges to the constant function with value $x$.\qed
\end{proof}

\subsection{Larger infinities}

Nothing is known about synchronization for larger infinite sets. But the proof
that ``synchronizing'' is equivalent to ``$2$-{homogeneous}'' fails, because
of the failure of Ramsey's theorem to guarantee a clique or independent set
of the same cardinality as $\Omega$.

We do not know whether the two concepts are equivalent or not for sets of
larger cardinalities. The answer might depend on the choice of set-theoretic
axioms.

\begin{example}
The Axiom of Choice implies that there is a \emph{well-ordering} of
$\mathbb{R}$, a total ordering in which every non-empty subset has a least
element. Choose such a well-ordering $\prec$. Now form a graph by joining
$v$ and $w$ if $\prec$ and the usual order $<$ agree on $\{v,w\}$, and not
if they disagree.

We claim that there is no uncountable clique. Let $Y$ be a clique; then $Y$
is well-ordered by the usual order on $\mathbb{R}$. In a well-order, each
non-maximal element $v$ has an immediate successor $v'$; choose a rational
number $q(v)$ in the interval $(v,v')$. The chosen rationals are all distinct.

Reversing the usual order shows that the complementary graph has the same
form; so the graph we constructed has no uncountable independent set either.
\end{example}

\subsection{Hulls}

The definition of cores in the infinite case is problematic, since it is not
clear what ``minimal'' means. See Bauslaugh~\cite{bauslaugh} for some
possible definitions of cores of infinite structures.

Hulls can be defined as usual, but don't do what we want!

Let $\Gamma$ have vertex set $\Omega$. The \emph{hull} of $\Gamma$ is the graph
$\Gr(\End(\Gamma))$; that is, two vertices $v,w$ are joined in $\Hull(\Gamma)$
if and only if there is no endomorphism $f$ of $\Gamma$ satisfying $vf=wf$.

\begin{theorem}
Any countable graph containing an infinite clique is a hull.
\end{theorem}

This follows just as in our previous argument using Ramsey's theorem.

What happens for graphs with finite clique size?
Nick Gravin proved the following result (personal communication from Dima
Pasechnik):

\begin{theorem}
If $\Gamma$ is an infinite hull with finite clique number, then the clique
number and chromatic number of $\Gamma$ are equal.
\end{theorem}

\begin{proof}
Let $\Delta$ be a finite subset of the vertex set of $\Gamma$. If the induced
subgraph on $\Delta$ is not complete, then there is an endomorphism of $\Gamma$
collapsing a non-edge of $\Delta$. If $\Delta f_1$ is not complete, there is an
endomorphism $f_2$ collapsing a non-edge of $\Delta f_1$; and so on. We end up
with a homomorphism from $A$ to a complete graph, whose size is at most
$\omega(\Gamma)$. So $\chi(\Delta)\le\omega(\Gamma)$ for every finite subgraph
$\Delta$ of $\Gamma$. It follows from a compactness argument due to
de~Bruijn (see below) that $\chi(\Gamma)\le\omega(\Gamma)$. Hence equality
holds.\qed
\end{proof}

\begin{theorem}
Let $\Gamma$ be an infinite graph, and suppose that every finite subgraph
of $\Gamma$ has chromatic number at most $m$. Then $\chi(\Gamma)\le m$.
\end{theorem}

\begin{proof}
We may suppose $\Gamma$ countable; let the vertex set be $\{v_1,v_2,\ldots\}$.
Construct a graph as follows. Vertices at level $i$ are $m$-colourings of the
induced subgraph on $\{v_1,\ldots,v_i\}$; vertices at levels $i-1$ and $i$
are adjacent if the colouring of $\{v_1,\ldots,v_{i-1}\}$ is a restriction of
the colouring of $\{v_1,\ldots,v_i\}$. (The unique vertex at level $0$ is
the root.) The graph is a tree; each level is finite and non-empty, and there
is a path from the root to any vertex. By K\"onig's Infinity Lemma, there is
an infinite path in the tree, which describes an $m$-colouring of $\Gamma$.\qed
\end{proof}

\subsection{Strong primitivity}

For infinite groups, Wielandt~\cite{winf} pointed out a notion which lies
between primitivity and $2$-transitivity. A permutation group $G$ on $\Omega$
is \emph{strongly primitive} if every $G$-invariant reflexive and transitive
relation is trivial (that is, invariant under the symmetric group).
Said otherwise, a transitive permutation group $G$ is strongly primitive if and
only every non-diagonal orbital graph for $G$ is strongly connected.

By Theorem~\ref{t:chicago}, a finite primitive group is strongly primitive.
But, for example, the group $\Aut(\mathbb{Q})$ of order-automorphisms of
$\mathbb{Q}$ is primitive (even $2$-homogeneous) but not strongly primitive,
since the order relation is reflexive and transitive but not symmetric.

We can refine Wielandt's notion by saying that a permutation group $G$ on
$\Omega$ is \emph{strong} if every $G$-invariant reflexive and transitive
relation is symmetric (and so is an equivalence relation). For example, a
\emph{torsion group} (one in which all elements have finite order) is strong.

Now $G$ is strongly primitive if and only if it is strong and primitive.

\clearpage

\section{Problems}

In this section we propose a number of problems that are naturally prompted by the results in this paper.

The $2$-transitive and $2$-homogeneous groups are known; \ffi{Q} groups are also known. What we do not know is whether there are spreading non-\ffi{Q} groups.

\begin{prob}
Is there any group which is spreading but not \ffi{Q}?
\end{prob}

If the previous question turns out to have a negative answer, then the next two problems will have a trivial answer.

\begin{prob}
Is there an infinite family of  groups which are spreading but not \ffi{Q}?
\end{prob}

\begin{prob}
Classify the spreading groups.
\end{prob}

It is also natural to consider the class of strongly separating groups.  Call a transitive permutation group $G$ on $\Omega$ \emph{strongly separating} if whenever $A,B$ are non-trivial subsets of $\Omega$ such that $|B|$ divides $|\Omega|$ and $|A|\cdot |B|=k\cdot |\Omega|$, then there exists $g\in G$ such that $|A\cap Bg|\neq k$.  Clearly, spreading groups are strongly separating and strongly separating groups are separating.  Note  that $S_m$ acting on $2$-sets is separating for $m$ odd, but not strongly separating, and that an affine group is strongly separating if and only if it is \ffi{Q}.  Also, $S_m$ acting on $k$-sets with $k\geq 3$ is never strongly separating and if the Hadamard conjecture is true, then $S_{2m}$ acting on $(m,m)$ partitions is not strongly separating.  

 One can show that $G$ is strongly separating if and only if, for each singular mapping $f$ with uniform kernel, and each proper subset $B$ of $\Omega$ that is a union of $\Ker(f)$-classes, there exists $g\in G$ such that $|Bgt^{-1}|>|B|$.  It follows that if $G$ is spreading and $f$ is a uniform singular mapping with $|\Omega f|=d$, then there is a reset word over $G\cup \{f\}$ with at most $d$ occurrences of $f$.

\begin{prob}
Are there strongly separating groups that are not spreading?  Are there strongly separating groups that are not \ffi{Q}?
\end{prob}

We know that separating groups properly contain the class of spreading groups.

\begin{prob}
Classify separating groups modulo a classification of spreading groups.
\end{prob}

Similarly we know that synchronizing groups properly contain separating groups. However we only have finitely many such semigroups.

\begin{prob}	
Is there an infinite family of groups which are synchronizing but not separating?
\end{prob}

A particular instance of the previous problem is the following.

\begin{prob}
Is it true that $\POm(5,q)$, for $q$ odd, form an infinite family of synchronizing, but not separating groups?
\end{prob}

This is equivalent to asking for a proof that the polar spaces associated with
these groups do not have a partition into ovoids. As we noted earlier, this
would follow if it could be shown that ovoids in this space
are necessarily classical (that is, hyperplane sections).

Of course, in this respect,  the ultimate goal would be an answer for the following problem.

\begin{prob}
Classify synchronizing groups modulo a classification of separating groups.
\end{prob}

A particular instance of the previous problem is the following.

\begin{prob}
Is it true that in the class of affine groups,  synchronizing and separating coincide?
\end{prob}

\begin{prob}
Do there exist subsets $A$ and $B$ of a finite simple group $S$, neither of which is a coset of a subgroup, such that $|g^{-1}Ah\cap B|=1$ for all $g,h\in S$?
\end{prob}

Since the symmetric group $S_m$ acting on pairs of points of $\{1,\ldots ,m\}$,
for $m$ even, is basic and almost synchronizing, but not synchronizing, it follows that the intersection of the former two classes properly contain the latter. More examples of almost synchronizing, but not synchronizing groups, can be found in~\cite{abc}.

\begin{prob}
Classify basic almost synchronizing groups modulo a classification of synchronizing groups.
\end{prob}

We already saw that there are basic not almost synchronizing groups, and there are non-basic almost synchronizing groups.

\begin{prob}
Classify basic groups modulo a classification of the basic and almost synchronizing groups.
\end{prob}

\begin{prob}
Classify almost synchronizing groups modulo a classification of the basic and almost synchronizing groups.
\end{prob}

A slight variation of the previous problems gives the following.

\begin{prob}
Classify almost synchronizing groups that are not basic.

Classify basic groups that are not almost synchronizing.
\end{prob}

A first step would be to decide whether the wreath product $S_k\wr S_m$
(in the power action) is almost synchronizing. It is known that there are
uniform maps of rank $k^i$ for $i=1,2,\ldots,m-1$ which are not synchronized
by this group; but does it synchronize any non-uniform maps?

Since primitive groups properly contain basic and almost synchronizing groups, the following problem is natural.

\begin{prob}
Classify primitive groups modulo the classification of basic groups and almost synchronizing groups.
\end{prob}

Our last questions on the relations of these groups are the following.

\begin{prob}	
How does almost synchronizing relate to partition separating?
\end{prob}

\begin{prob}	
How does basic relate to partition separating?
\end{prob}

\begin{prob}	
Classify partition separating modulo a classification of almost synchronizing and basic.
\end{prob}

Many of the above classification problems will be difficult, since
they include notorious unsolved problems in extremal combinatorics, design
theory and finite geometry.

\begin{prob}
Is there a sublinear bound for the number of non-synchronizing
ranks of a primitive group?
\end{prob}

A weaker (and perhaps easier) question would be:

\begin{prob}
Is there a sublinear bound for the number of ranks of endomorphisms of a
vertex-primitive graph?
\end{prob}

A graph $\Gamma$ is a \emph{pseudocore} if every endomorphism is either an
automorphism or a proper colouring. Clearly the automorphism group of a
pseudocore (if it is transitive) is almost synchronizing and has only one
non-synchronizing rank. Several examples of such graphs are known~\cite{gr}.
Indeed, there is no known graph with primitive automorphism group with
permutation rank~$3$ which is not a pseudocore. If it is true that every
strongly regular graph is a pseudocore, as a recent preprint by David
Roberson claims (personal communication from Gordon Royle), it would follow
that if $G$ is primitive with permutation rank~$3$, then $|\NS(G)|\le1$.
(Recall that the permutation rank is the number of $G$ orbits on $\Omega^2$.)

\begin{prob}
Is it true that, for any primitive permutation group $G$, $|\NS(G)|$ is
bounded by a function of the permutation rank?
\end{prob}

Regarding closures, we ask the following.

\begin{prob}
Is it true that the \ffi{R}-closure of a permutation group is equal to its strong $2$-closure?
\end{prob}

\begin{prob}
What properties does the \ffi{F}-closure of a permutation group has when
$\mathbb{F}$ is a field of non-zero characteristic, or a local
field such as the $p$-adic numbers?
\end{prob}

The following question, related to closures, has been nagging the third author for years.
\begin{prob}
Let $M\subseteq T(\Omega)$ be  transformation monoid.  What is the relationship between $2$-transitivity of $M$ and irreducibility of the augmentation submodule of $\mathbb C\Omega$?  Note that if $M$ is not a group, then both properties imply that $M$ is primitive and synchronizing.
\end{prob}

Regarding the groups linked to association schemes, we have a number of problems that parallel those above about synchronizing groups.

\begin{prob}
Is there any AS-free primitive diagonal group?
\end{prob}

More generally, we have the following.

\begin{prob}
Classify AS-free groups.
\end{prob}

\begin{prob}
Classify stratifiable groups modulo a classification of generously transitive groups.
\end{prob}
	
\begin{prob}
Classify AS-friendly groups modulo a classification of AS-free and stratifiable groups.
\end{prob}

We introduced two hierarchies in this paper: one  on page \pageref{hierarchy}, involving synchronizing groups and friends, and another  in Theorem~\ref{9.2}.

\begin{prob}	
Draw a single \emph{Venn Diagram} of the two hierarchies, together with partition separating groups and strongly separating groups, pointing out which regions of the diagram are not empty, contain infinite families, etc.
\end{prob}

There are efficient algorithms to check if a given set of permutations generates a primitive group.

\begin{prob}
Find an efficient [polynomial-time] algorithm to decide if a given set of permutations generates a synchronizing [spreading, separating, almost synchronizing, partition separating, AS-free, generously transitive, stratifiable, AS-friendly] group.
\end{prob}

Existing algorithms for deciding whether a given primitive group is
synchronizing or separating involve solving \textsf{NP}-hard problems, such
as clique number and chromatic number, but on rather special graphs (those
which are vertex-primitive). Does the information available about primitive
groups using CFSG allow faster algorithms to be found?

\begin{prob}
What is the computational complexity of computing the $2$-closure of a
(primitive) permutation group?
\end{prob}

The previous problems are linked to the following problem.

\begin{prob}
For the extent of GAP's library of primitive groups, include in GAP the list of synchronizing [spreading, separating, almost synchronizing, AS-free, generously transitive, stratifiable, AS-friendly] groups.
\end{prob}
	
The two first authors,  Gordon Royle, James Mitchell \cite{semigroups}, Artur Schafer, Csaba Schneider and Pablo Spiga, independently, wrote GAP code that produced lists for some of these classes (almost synchronizing, synchronizing, spreading, separating, or the association schemes classes), reaching, in the best cases (Royle's and Spiga's), degrees of a few hundreds.  In order to better understand all these classes of groups, examples of larger degrees are needed and more sophisticated code/algorthims must come into play.

\begin{prob}
For the extent of GAP's library of primitive groups, include in GAP a library of all pairs of partitions and sections preserved by a given non-synchronizing group.
\end{prob}

\begin{prob}
For the extent of GAP's library of primitive groups, include in GAP a library of all classes of groups in this paper (almost synchronizing, synchronizing, spreading, separating, strongly separating, the association schemes classes, etc).
\end{prob}

There are a number of natural problems relating the subject of this text to the \v{C}ern\'y conjecture.

\begin{prob}
Is it true that if $A\subseteq S_n$ generates a primitive group and $f$ is a rank $n-1$ mapping, then there is a reset word over $A\cup \{f\}$ of length at most $(n-1)^2$?  Rystsov proved a quadratic bound if $f$ is an idempotent, but John Dixon (in an unpublished example) showed that Rystsov's proof scheme cannot yield the bound of $(n-1)^2$, even in the case that $A$ generates the symmetric group for $n\geq 5$!
\end{prob}

\begin{prob}
Is every group a \v{C}ern\'y group? Perhaps one can generalize Dubuc's scheme~\cite{dubuc} to abelian groups; however, it is not immediately clear how to generalize Dubuc's result to arbitrary generating sets of cyclic groups of order not a prime power!
\end{prob}

\begin{prob}
Is the \v{C}ern\'y conjecture true for synchronizing automata such that some subset of the transitions generates a transitive permutation group?  How about the same question replacing the adjective ``transitive'' by any of the following adjectives:  primitive, synchronizing, separating, spreading, \ffi{Q}, $2$-homogeneous or $2$-transitive?
\end{prob}

\begin{prob}
Let $G$ be a primitive group contained in $S_n$, and generated by $S\subseteq G$. Let $P$ be a partition of $\{1,\ldots ,n\}$. Is it true that if $G$ takes a $2$-subset of $X$ into some part of $P$, then it can do so in a word over $S$ of size linear in $n$?
\end{prob}

The best general bound so far for the length of a reset word in a synchronizing automaton was proved by Pin~\cite{Pincerny} and Frankl~\cite{frankl}; see also~\cite{klyachko}. The idea  is the following. Suppose we have a reset word $w$ of minimal length. We have an $n$-set $X$ and $|Xw|=1$. Let $k$ be arbitrary in $1<k<n$. Then there are a two prefixes of $w$, say $w_k$ and $w_{k+1}$, such that $|Xw_k|=k$ and $|Xw_{k+1}|=k-1$, and $w_k$, $w_{k+1}$ are  the smallest prefixes satisfying these properties. In particular $w_{k+1}=w_ka_1\ldots a_m$. Let $P_i:=Xw_ka_1\ldots a_i$. Clearly $|P_i|=k$, for all $i\in \{1,\ldots ,m-1\}$. Since $|P_{m-1}|=k$ and $|P_m|=k-1$, there are two elements $p,q\in P_{m-1}$ in the kernel of $a_m$. Let $x^i:=xa^{-1}_{m-1}\ldots a^{-1}_i$, for $x\in \{p,q\}$ and $i\in \{2,\dots,m-1\}$.   Observe that $x^i\in P_{i-1}$, and, given the minimality of $w$, $i-1$ is the smallest index of the $P$s containing both $p^i$ and $q^i$. Therefore Pin asked how large can be $m$ in a family of sets subject to this condition; the answer to this question by Frankl~\cite{frankl} yielded the current best bound for reset words (please see also \cite{volkov}).
However, if we are dealing with a spreading group in which the non-invertible map only comes into play to reduce the rank, all the sets $P_i$ (for $i\in \{1,\ldots,m-1\}$), in the argument above, are in the orbit of a $k$-set. Therefore, hopefully, the bound will be much lower than the one found by Frankl~\cite{frankl}. Therefore the following problem is very natural.

\begin{prob}
Solve the problem analogous to the one solved in~\cite{frankl} but with the extra hypothesis that the $k$-sets are all contained in the same orbit under the action of some spreading group.
\end{prob}		

A final general problem on the philosophy behind our investigation in this
paper:

\begin{prob}
We have seen in this paper several instances where a characterization of
primitivity leads to a condition which can be generalized leading to a new
class of permutation groups. For example,
\begin{itemize}\itemsep0pt
\item $G$ is primitive if and only if it synchronizes every map of rank $n-1$;
this leads us to the class of synchronizing groups.
\item $G$ is primitive if and only if it preserves no divisible association
scheme; this leads us to the class of AS-free groups.
\item $G$ is primitive if and only if every orbital graph for $G$ is connected;
replacing ``connected'' by ``strongly connected'' leads to the class of
strongly primitive groups.
\item $G$ is primitive if and only if, for every $2$-set $A$ and $2$-partition
$P$, there exists $g\in G$ such that $Ag$ is a section for $P$. (This is the
assertion that a graph is connected if, for every $2$-partition, there is an
edge of the graph which is a section.) If we replace $2$ by $k$ here, we get
the definition of the \emph{$k$-universal transversal property}, which is
discussed (together with its implications for semigroup theory) in~\cite{ac2}.
\end{itemize}
In the authors' view, it is worthwhile searching for other characterizations
of primitivity, in the hope of uncovering other interesting classes to study!
\end{prob}

\end{document}